\documentclass[10pt, leqno]{amsart}
\usepackage{array}
\usepackage{longtable}
\usepackage{url}
\usepackage{enumerate}
\usepackage{amsmath}
\usepackage{amssymb}
\usepackage{amsthm}
\usepackage{amsfonts}
\usepackage{tikz,pgf}
\usepackage{graphicx}

\allowdisplaybreaks

  \usepackage{ifthen}
 \newcommand{\mymarginpar}[1]{%
    \marginpar{\ifthenelse{\isodd{\arabic{page}}}{\flushleft 
#1}{\flushright #1}}}

\numberwithin{equation}{section}

 \newcommand{\eps}{\varepsilon}           
 
 \newcommand{\IC}{\mathbb{C}}                  
                   
 \newcommand{\IF}{\mathbb{F}}

 \newcommand{\IR}{\mathbb{R}}                  
                   
 \newcommand{\IT}{\mathbb{T}}                  
 \newcommand{\IZ}{\mathbb{Z}}         
                  
\newcommand{\CA}{\mathcal{A}}
\newcommand{\CB}{\mathcal{B}}

\newcommand{\CD}{\mathcal{D}}
\newcommand{\CE}{\mathcal{E}}
\newcommand{\CF}{\mathcal{F}}

\newcommand{\CO}{\mathcal{O}}

\newcommand{\CT}{\mathcal{T}}
\newcommand{\CR}{\mathcal{R}}
\newcommand{\CJ}{\mathcal{J}}
\newcommand{\CK}{\mathcal{K}}
\newcommand{\CU}{\mathcal{U}}

\newcommand{\CW}{\mathcal{W}}

\newcommand{\CZ}{\mathcal{Z}}

 \theoremstyle{plain} 
 \newtheorem{Theorem}{Theorem}[section]
 \newtheorem{Lemma}[Theorem]{Lemma}
 \newtheorem{Proposition}[Theorem]{Proposition}
 
 \newtheorem{Corollary}[Theorem]{Corollary}
 
 \theoremstyle{definition} 
 \newtheorem{Definition}[Theorem]{Definition}
 \newtheorem{Remark}[Theorem]{Remark}
 \newtheorem{Example}[Theorem]{Example}

\begin{document}

\author{Jack Spielberg}
\title{groupoids and $C^*$-algebras for left cancellative small categories}
\date{9 June 2018}
\address{School of Mathematical and Statistical Sciences \\ Arizona State University \\ P.O. Box 871804 \\ Tempe, AZ 85287-1804}
\email{jack.spielberg@asu.edu}

\begin{abstract}

Categories of paths are a generalization of several kinds of oriented discrete data that have been used to construct $C^*$-algebras.  The techniques introduced to study these constructions apply almost verbatim to the more general situation of left cancellative small categories.  We develop this theory and derive the structure of the $C^*$-algebras in the most general situation.  We analyze the regular representation, and the Wiener-Hopf algebra in the case of a subcategory of a groupoid.

\end{abstract}

\maketitle

\section{Introduction}
\label{sec intro}

The current uses of oriented discrete data to define $C^*$-algebras began with the 1980 paper of Cuntz and Krieger \cite{CK}, which in turn was inspired by Cuntz' seminal 1977 paper \cite{cuntz}.  The paper \cite{cuntz} introduced the algebras $\CO_n$ by a very simple and natural presentation.  Even here the Bernoulli $n$-shift was implicit in the proofs.  In \cite{CK} the algebras $\CO_A$ were introduced, where the finite integer matrix $A$ was used both to generalize the presentation of $\CO_n$ and to make explicit the connection with the shift defined by $A$.  In fact it is shown there how to construct (the stablization of) $\CO_A$ directly from the symbolic dynamics.  The subsequent development had two branches.  On the one hand, a great deal of work was devoted to generalizing the presentations of the algebras $\CO_A$ to define $C^*$-algebras for directed graphs (see \cite{raeburn} for historical comments), infinite integer matrices (\cite{exelac}), and higher rank graphs (\cite{kumpas}, et. al.).  On the other hand, Nica (\cite{nica}) indicated how the algebra $\CT\CO_n$ is the Weiner-Hopf algebra of the ordered group $(\IF_n,\IF_n^+)$, and he introduced the class of {\it quasi-lattice ordered} groups generalizing this discovery.  He constructed a $C^*$-algebra from the left regular representation of the semigroup, and gave a presentation for a universal $C^*$-algebra based on the quasi-lattice order.  This idea was pushed further by Laca et. al. (\cite{lacrae}, \cite{crilac}), and took dramatic steps forward with the work of Li (\cite{li2012}; see also the survey \cite{li_notes} ).  In particular, Li considered chiefly {\it right LCM monoids}, that is having order structure like that of the positive cone in a quasi-lattice ordered group.  Like Nica, he defined $C^*$-algebras by considering various presentations based on the regular representation of the monoid.

Categories of paths were introduced in \cite{spi_pathcat} as a generalization of higher rank graphs as well as ordered groups.  The motivations were twofold.  First, it was desired to introduce more flexibile generalizations of directed graphs than the higher rank graphs.  Part of the power of the graph algebra construction lay in the freedom one has in constructing examples --- any diagram of dots and arrows will define a $C^*$-algebra.  This is not a feature shared by the higher rank graphs.  Moreover, the theory of higher rank graphs to this point focuses exclusively on the finitely aligned case.  But many interesting examples can be given that are not finitely aligned, so it was natural to try to remove this restriction.  The nonfinitely aligned case of higher rank graphs leads to the second motivation for \cite{spi_pathcat}.  Even in the case of directed graphs, the development from finite irreducible 0 - 1 matrices to locally finite, then row-finite without sources, and on to arbitrary, directed graphs required repeated tinkering with the presentation of the $C^*$-algebra (but see \cite{spi_graphalg} for an alternate approach).  This evolution was replayed in the case of higher rank graphs (e.g. from row finite to locally convex to finitely aligned).  It seems clear that an intrinsic construction of the algebra should be given without reference to generators and relations.

The construction in \cite{spi_pathcat} is based on a {\it category of paths}, i.e.~a cancellative small category with no (nontrivial) inverses.  In analogy with the subshifts underlying the Cuntz-Krieger algebras, shift maps were defined on a category of paths.  These were used to construct a local compactification of the category.  Then the same shifts induce local homeomorphisms of this space, from which the groupoid of germs can be constructed.  The $C^*$-algebras of these groupoids generalize the Toeplitz algebras of directed graphs and finitely aligned higher rank graphs as well as the algebras of Nica and Li.  Once this algebra is constructed, a presentation modeled on these examples can be derived.  The notion of finite alignment also generalizes to categories of paths. The finitely aligned case is developed in detail in \cite{spi_pathcat}.  In particular, the boundary of the local compactification is defined, and the {\it Cuntz-Krieger algebra} is obtained by restricting the groupoid to this boundary.  Then the presentation is extended to one for this ``boundary quotient''.

It was remarked in \cite{spi_pathcat} that most of the results of the paper remain true without the hypothesis of right cancellation.  In fact, most of the results remain true (after minor modifications) without the hypothesis that the category not contain inverses.  In this paper we give the details of this generalization.  The presence of inverses has a significant impact on the groupoid underlying our construction.  In the general setting of left cancellative small categories there are two groupoids that are subtly different.  It turns out that in the presence of inverses the groupoid of germs loses some structure of the category.  We give a detailed treatment of these two groupoids and of the corresponding $C^*$-algebras, and both for the Toeplitz algebras and Cuntz-Krieger algebras.  In particular we derive finer structure of the unit space than was done in \cite{spi_pathcat} for the general nonfinitely aligned case, and we give several conditions sufficient for the groupoids to be Hausdorff.  We study the regular representation, distinguishing the regular representation of the groupoid from the left regular representation of the category.  We give examples where these are different.  This has consequences for the notion of amenability, as defined in \cite{nica}.  Similarly, we generalize the Wiener-Hopf algebra of \cite{nica} to subcategories of groupoids, and study its relation to the $C^*$-algebra generated by the regular representation of the category.  We also extend the results of \cite{spi_pathcat} on the boundary algebra to the nonfinitely aligned case, deriving the presentation of the Cuntz-Krieger algebra by generators and relations.  Some other work also generalizing the example of higher rank graphs are \cite{broyet}, where the degree functor of higher rank graphs is replaced by a functor to a cancellative abelian monoid, \cite{exel08} which treats a more general notion of semigroupoid, but only generalizes row finite higher rank graphs, and the announcement \cite{exeste} which promises to give a further generalization.  All of these push further than this paper in some directions, but less completely in others.

We end this introduction with a description of the contents of the paper.  We begin with basic definitions and a brief recap of statements and results from \cite{spi_pathcat}.  As mentioned above, many results from that paper hold for the more general situation of left cancellative small categories, with unchanged proofs.  This includes finite alignment, and also the amalgamation of categories presented in \cite[Section 11]{spi_pathcat}.  Amalgamation allows us to construct interesting examples of categories, and then transform them into monoids that inherit the interesting properties.  We then define the two groupoids whose study occupies much of the paper.  We give a more detailed description of the points of the unit space, generalizing results of \cite[Section 7]{spi_pathcat} for finitely aligned categories of paths, and we use this to give some sufficient conditions for the groupoids to be Hausdorff.  The heart of the paper is in the description of the $C^*$-algebras by generators and relations, where the subtle differences between the two groupoids is explored in detail, both for the Toeplitz algebras and for the Cuntz-Krieger algebras. (But it is important to notice that the presentations by generators and relations are derived {\it after} the $C^*$-algebras have been defined.) We show that Exel's notion of {\it tight} representations, or equivalently formulated as {\it cover-to-join} by Donsig and Milan, is equivalent to the natural generalization of the Cuntz-Krieger relations that relate the Toeplitz algebra to the Cuntz-Krieger algebra in the classical cases. (This generalizes the corresponding result in the finitely aligned case from \cite{bkqs}.)  We then study the reduced $C^*$-algebra of the groupoid, and hence the role played in the theory by the $C^*$-algebra generated by the regular representation of the category.  Perhaps surprisingly, we find that these two regular representations do not always coincide.  This can occur even if the groupoid is amenable; in fact, we give an example in which the Toeplitz algebra is type I.  Thus Nica's notion of amenability, that the universal $C^*$-algebra and the $C^*$-algebra generated by the regular representation of the category naturally coincide, does not reflect nuclearity of the $C^*$-algebra or amenability of the underlying groupoid.  In the case of a subcategory of a groupoid we study the analog of the Wiener-Hopf algebra, and generalize some of the results of \cite[Section 8]{spi_pathcat} for ordered groups.  Last we discuss the situation of subcategories.  The general theory developed in \cite{spi_pathcat} treats a {\it relative category of paths}, i.e. a pair $(\Lambda_0,\Lambda)$ where $\Lambda_0$ is a subcategory of $\Lambda$.  This was an attempt to build some functoriality into the construction of the $C^*$-algebra.  While writing the current paper it became apparent that this is made more complicated by the presence of inverses.  Moreover some small errors in the use of relative categories of paths were noticed.  We correct these errors in the current paper at the same time that we indicate which aspects of functoriality remains true in the case of general left cancellative small categories.

The author would like to thank Erik B\'edos, Steve Kaliszewski, John Quigg, David Milan and Alan Donsig for many helpful conversations and suggestions, and also Tron Omland and Nico Stammeier for the invitation to the Mathematics Institute in Oslo where this paper was finished.

\section{LCSC's}
\label{sec lcsc}

\begin{Definition}

A {\it left cancellative small category} (or {\it LCSC}) is a small category satisfying left cancellation, that is, if $\alpha \beta = \alpha \gamma$ then $\beta = \gamma$.

\end{Definition}

\begin{Definition} \label{def shift maps}

Let $\Lambda$ be an LCSC.  The {\it right} and {\it left shift maps} are the maps $\tau^\alpha : s(\alpha) \Lambda \to r(\alpha)\Lambda$ and $\sigma^\alpha : \alpha \Lambda \to s(\alpha)\Lambda$, for $\alpha \in \Lambda$, defined by $\tau^\alpha(\beta) = \alpha \beta$ and $\sigma^\alpha = (\tau^\alpha)^{-1}$.

\end{Definition}

We note that $\tau^\alpha$ is one-to-one because of left cancellation, and hence $\sigma^\alpha$ is well-defined.  We continue with the notations introduced in \cite[Definition 2.6]{spi_pathcat}:  we say that $\beta$ {\it extends} $\alpha$, and that $\alpha$ is an {\it intial segment} of $\beta$, if $\beta \in \alpha \Lambda$.  We let $[\beta]$ denote the set of all initial segments of $\beta$.   We write $\alpha \Cap \beta$ if $\alpha \Lambda \cap \beta \Lambda \not= \varnothing$, and $\alpha \perp \beta$ otherwise.

The relation $\alpha \in [\beta]$ is reflexive and transitive, but not necessarily antisymmetric.  However it is straightforward to verify the following lemma.

\begin{Lemma}

Let $\Lambda$ be an LCSC, and let $\alpha$, $\beta \in \Lambda$.  Then $\alpha \in [\beta]$ and $\beta \in [\alpha]$ if and only if $\beta \in \alpha \Lambda^{-1}$.

\end{Lemma}

\begin{Definition}

Let $\alpha$, $\alpha' \in \Lambda$.  We write $\alpha \approx \alpha'$ if $\alpha' \in \alpha \Lambda^{-1}$.

\end{Definition}

We note that $\approx$ is an equivalence relation on $\Lambda$.  The proof of the following lemma is straightforward.

\begin{Lemma} \label{lem inverses}

Let $\Lambda$ be an LCSC and let $\alpha$, $\beta \in \Lambda$.

\begin{enumerate}

\item \label{lem inverses one} If $\alpha \in [\beta]$ then $\alpha' \in [\beta']$ for all $\alpha' \approx \alpha$ and $\beta' \approx \beta$.

\item \label{lem inverses two} The following statements are equivalent:  $\alpha \approx \beta$, $\alpha \Lambda = \beta \Lambda$, $[\alpha] = [\beta]$.

\end{enumerate}

\end{Lemma}

We adapt some notation from \cite[Section 2]{spi_pathcat}.

\begin{Definition} \label{d.zigzag}

Let $\Lambda$ be an LCSC.  A \textit{zigzag} is an even tuple of the form $\zeta = (\alpha_1,\beta_1,\ldots,\alpha_n,\beta_n)$, where $\alpha_i$, $\beta_i \in \Lambda$, $r(\alpha_i) = r(\beta_i)$, $1 \le i \le n$, and $s(\alpha_{i+1}) = s(\beta_i)$, $1 \le i < n$.  We let $\CZ_{\Lambda}$ denote the set of all zigzags.  We may omit the subscript $\Lambda$ when it is clear from the context.  We define the maps $s$ and $r$ on $\CZ$ by $s(\zeta) = s(\beta_n)$ and $r(\zeta) = s(\alpha_1)$, and the \textit{reverse} of $\zeta$ as  $\overline{\zeta} =
(\beta_n,\alpha_n,\ldots,\beta_1,\alpha_1)$.  Each zigzag $\zeta \in \CZ_{\Lambda}$ defines a \textit{zigzag map} on $\Lambda$ that we denote $\varphi_\zeta \equiv \varphi^\Lambda_\zeta$, given by $\varphi_\zeta =
\sigma^{\alpha_1} \circ \tau^{\beta_1} \circ \cdots \circ \sigma^{\alpha_n} \circ \tau^{\beta_n}$.  We let $A(\zeta) \equiv A_\Lambda(\zeta)$ denote the domain of $\varphi_\zeta$.  Thus $A(\zeta) = \varphi_{\overline{\zeta}}(\Lambda)\subseteq s(\zeta)\Lambda$, and the range of $\varphi_\zeta$ equals $A(\overline{\zeta})$.  We call $A(\zeta)$ a \textit{zigzag set}.

\end{Definition}

\begin{Remark}
\label{r.zigzag}

\begin{enumerate}

\item The sets $A(\zeta)$ are the analog of {\it constructible right ideals} in \cite{li2012}.
\label{r.zigzag.a}

\item $\varphi_{(r(\alpha),\alpha)} = \tau^\alpha$, $\varphi_{(\alpha,r(\alpha))} = \sigma^\alpha$, and $\alpha\Lambda = A(\alpha,r(\alpha))$.
\label{r.zigzag.b}

\item $\CZ$ is closed under (composable) concatenation, and $\varphi_{\zeta_1\zeta_2} = \varphi_{\zeta_1} \circ \varphi_{\zeta_2}$ when $s(\zeta_1) = r(\zeta_2)$.  For $v \in \Lambda^0$ we let $v \CZ = \{ \zeta : r(\zeta) = v \}$, and similarly for $\CZ v$.  
\label{r.zigzag.c}

\item $\varphi_\zeta$ is one-to-one, with inverse $\varphi_{\overline{\zeta}}$.
\label{r.zigzag.d}

\item $A(\zeta) = A(\overline{\zeta}\zeta)$, and if $\zeta_1$, $\zeta_2 \in \CZ v$, then $A(\zeta_1) \cap A(\zeta_2) = A(\overline{\zeta_1}\zeta_1\overline{\zeta_2}\zeta_2)$.  (Thus the collection of zigzag sets is closed under intersection.)
\label{r.zigzag.e}

\item For $\zeta \in \CZ$ and $\alpha \in A(\zeta)$, we have $\alpha\Lambda \subseteq A(\zeta)$.  For $\beta \in s(\alpha) \Lambda$, $\varphi_\zeta(\alpha \beta) = \varphi_\zeta(\alpha) \beta$.
\label{r.zigzag.f}

\end{enumerate}

\end{Remark}

\begin{Definition} \label{def lcsc}

For $v \in \Lambda^0$ let $\CD^{(0)}_v$ (or $\CD^{(0)}_{\Lambda,v}$) denote all nonempty sets of the form $A(\zeta)$ for $\zeta \in \CZ_{\Lambda} v$.  Let $\CA_v$ (or $\CA_{\Lambda,v}$) denote the ring of sets generated by $\CD^{(0)}_{\Lambda,v}$.  We let $\CD_v$ (or $\CD_{\Lambda,v}$) denote the collection of nonempty sets of the form $E \setminus \bigcup_{i=1}^n F_i$, where $E$, $F_1$, $\ldots$, $F_n \in \CD^{(0)}_v$, and $F_i \subseteq E$.  It follows from routine algebra, and Remark \ref{r.zigzag}\eqref{r.zigzag.e}, that $\CA_v$ equals the collection of finite disjoint unions of sets from $\CD_v$ (cf. \cite[Lemma 2.6]{spi_pathcat}).

\end{Definition}

\begin{Remark}

It is routine to check that $\varphi_\zeta (\CA_{s(\zeta)}) \subseteq \CA_{r(\zeta)}$.  As in \cite[Section 2]{spi_pathcat} one can show that $\{ \CA_v : v \in \Lambda^0 \}$ are the smallest rings of sets containing $\{ \alpha \Lambda : \alpha \in \Lambda \}$ and invariant under $\{\tau^\alpha, \sigma^\alpha : \alpha \in \Lambda \}$.

\end{Remark}

\section{Finite alignment}
\label{sec finite alignment}

\begin{Definition}

Let $\Lambda$ be an LCSC, and let $F \subseteq \Lambda$.  The elements of $\bigcap_{\gamma \in F} \gamma \Lambda$ are called the {\it common extensions} of $F$.  A common extension $\eps$ of $F$ is {\it minimal} if for any common extension $\gamma$ with $\eps \in \gamma \Lambda$ we have $\gamma \approx \eps$.

\end{Definition}

\begin{Definition} \label{def finitely aligned}

The LCSC $\Lambda$ is \textit{finitely aligned} if for every pair of elements $\alpha$, $\beta \in \Lambda$, there is a finite subset $G$ of $\Lambda$ such that $\alpha\Lambda \cap \beta\Lambda = \bigcup_{\eps \in G} \eps\Lambda$.

\end{Definition}

\begin{Lemma} \label{lem finitely aligned}

Let $\Lambda$ be a finitely aligned LCSC, and let $F \subseteq \Lambda$ be a finite subset.  Let $\bigvee F$ denote the set of minimal common extensions of $F$.  Then

\begin{enumerate}

\item \label{lem finitely aligned one} $\bigvee F$ contains only finitely many $\approx$-classes.

\item \label{lem finitely aligned two} $\bigcap_{\alpha \in F} \alpha \Lambda = \bigcup_{\beta \in \bigvee F} \beta \Lambda$.

\end{enumerate}

\end{Lemma}

\begin{proof}
By Definition \ref{def finitely aligned} and induction there is a finite set $G \subseteq \Lambda$ such that $\bigcap_{\alpha \in F} \alpha \Lambda = \bigcup_{\eps \in G} \eps \Lambda$.  If $\eps$, $\eps' \in G$ are such that $\eps' \in \eps \Lambda$ we can delete $\eps'$ from $G$ without affecting the equality.  Therefore we may assume that $\eps' \not\in \eps \Lambda$ for all $\eps$, $\eps' \in G$.  We claim that $G \subseteq \bigvee F$.  To see this, let $\eps \in G$.  Suppose $\gamma \in \bigcap_{\alpha \in F} \alpha \Lambda$ with $\eps \in \gamma \Lambda$.  There is $\eps' \in G$ such that $\gamma \in \eps' \Lambda$.  Then $\eps \in \eps' \Lambda$, so $\eps = \eps'$.  Hence $\gamma \in \eps \Lambda$.  Therefore $\gamma \approx \eps$.  Therefore $\eps \in \bigvee F$.  We further claim that for all $\beta \in \bigvee F$ there is $\eps \in G$ such that $\beta \approx \eps$.  To see this, let $\beta \in \bigvee F$.  Then $\beta \in \bigcap_{\alpha \in F} \alpha \Lambda$, so there is $\eps \in G$ such that $\beta \in \eps \Lambda$.  Since $\beta$ is minimal we have that $\beta \approx \eps$.  We now prove \eqref{lem finitely aligned two}. By the first claim above, we have $\bigcap_{\alpha \in F} \alpha \Lambda = \bigcup_{\eps \in G} \eps \Lambda \subseteq \bigcup_{\eps \in \bigvee F} \eps \Lambda$.  By the second claim above, we have  that for each $\beta \in \bigvee F$ there exists $\eps \in G$ such that $\beta \approx \eps$, and hence that $\beta \Lambda = \eps \Lambda$. Therefore $\bigcup_{\beta \in \bigvee F} \beta \Lambda \subseteq \bigcup_{\eps \in G} \eps \Lambda = \bigcap_{\alpha \in F} \alpha \Lambda$.  To see that \eqref{lem finitely aligned one} holds, note that by the second claim above $G$ contains a representative of every $\approx$-class of elements of $\bigvee F$.
\end{proof}

\begin{Remark}

Since $\alpha \Lambda = \beta \Lambda$ when $\alpha \approx \beta$, in Lemma \ref{lem finitely aligned} we may replace $\bigvee F$ by a finite subset that is ``independent'' in that if $\alpha \in \beta \Lambda$ then $\alpha = \beta$.

\end{Remark}

We now give the statements of the results from \cite[Section 3]{spi_pathcat} needed explicitly elsewhere in this article.  The proofs from \cite{spi_pathcat} apply also in the situation of LCSC's.

\begin{Definition} \label{def tail sets}

Let $\Lambda$ be an LCSC.  We let $\CE^{(0)}_v = \{ \alpha \Lambda : \alpha \in v \Lambda \}$, and we let $\CE_v$ denote the collection of all nonempty sets of the form $E \setminus \bigcup_{i=1}^n F_i$ with $E$, $F_1$, $\ldots$, $F_n \in \CE^{(0)}_v$.

\end{Definition}

\begin{Lemma} \label{lem finitely aligned generating family}

Let $\Lambda$ be a finitely aligned LCSC.  Then $\CA_v = \{ \bigsqcup_{j=1}^m A_j : A_j \in \CE_v \}$. (\cite[Remark 3.5]{spi_pathcat})

\end{Lemma}

\begin{Lemma} \label{lem finitely aligned zigzag map}

Let $\Lambda$ be a finitely aligned LCSC, and let $\zeta \in \CZ$.  Then $\varphi_\zeta$ is a finite union of maps of the form $\tau^\gamma \circ \sigma^\delta$ with $\gamma$, $\delta \in \Lambda$. (\cite[Lemma 3.3]{spi_pathcat})

\end{Lemma}

\begin{Corollary} \label{cor finitely aligned zigzag set}

Let $\Lambda$ be a finitely aligned LCSC, and let $\zeta \in \CZ$.  Then $A(\zeta)$ is a finite union of sets from $\CE^{(0)}$. (\cite[Corollary 3.4]{spi_pathcat})

\end{Corollary}

\section{Amalgamation of LCSC's}
\label{sec amalgamation}

In \cite[Section 11]{spi_pathcat} a construction is given for the amalgamation of a collection of categories of paths over an equivalence relation on the set of all vertices.  Most of this construction carries over word-for-word for LCSC's.  We recall the basic definitions and results here.  In fact, the beginning of the construction applies without change to arbitrary small categories, and we state the facts in this generality.

\begin{Definition} \label{def amalgamation}

Let $\{ \Lambda_i : i \in I \}$ be a collection of small categories, and let $\sim$ be an equivalence relation on $\bigcup_{i \in I} \Lambda_i^0$.  Let $L$ be the set of {\it composable tuples}:  $L = \{ (\alpha_1, \ldots, \alpha_m) : \alpha_j \in \bigcup_{i \in I} \Lambda_i,\ s(\alpha_j) \sim r(\alpha_{j+1} \}$.  There is a partially defined concatenation on $L$:  $(\alpha_1, \ldots, \alpha_m) (\beta_1, \ldots, \beta_n) = (\alpha_1, \ldots, \alpha_m,\beta_1, \ldots, \beta_n)$ if $s(\alpha_m) \sim r(\beta_1)$.  We let $\approx$ be the equivalence relation on $L$ generated by
\begin{align*}
(\cdots,\theta,\theta',\cdots) &\approx (\cdots,\theta\theta',\cdots), \text{ if } s(\theta) = r(\theta'), \\
(\cdots,\mu,w,\nu,\cdots) &\approx (\cdots,\mu,\nu,\cdots), \text{ if } w \in \bigcup_{i\in I}\Lambda_i^0.
\end{align*}
Let $\Lambda = L/\approx$ and $\Lambda^0 = (\bigcup_{i \in I}\Lambda_i)/\sim$.  We refer to $\Lambda$ as an {\it amalgamation} of the $\Lambda_i$.

\end{Definition}

\begin{Proposition} \label{prop amalgamation}

$\Lambda$ as constructed above is a small category.

\end{Proposition}

\begin{Lemma} \label{lem normal form}

Let $\alpha \in L$.  There exists a unique element $\beta \in L$ such that

\begin{enumerate}

\item $\alpha \approx \beta$.

\item $\beta = (\beta_1, \ldots, \beta_n)$ with

\begin{enumerate}

\item $\beta_j \not\in \bigcup_{i \in I} \Lambda_i^0$ for all $j$.

\item $s(\beta_j) \not= r(\beta_{j+1})$ for all $j$.

\end{enumerate}

\end{enumerate}

\end{Lemma}

The element $\beta$ is called the \textit{normal form} of $\alpha$.

\begin{Lemma} \label{lem amalgamation and cancellation}

Let $\Lambda$ be an amalgamation of a collection $\{\Lambda_i : i \in I \}$ of small categories.  If $\Lambda_i$ is left cancellative (respectively right cancellative, without nontrivial inverses) for all $i \in I$, then so is $\Lambda$.

\end{Lemma}

\begin{Corollary}

The amalgamation of a family of LCSC's is again an LCSC.

\end{Corollary}

\begin{Proposition} \label{prop common extensions in amalgamations}

(Recall the notation $\Cap$ introduced after Definition \ref{def shift maps}.) Let $\alpha = (\alpha_1, \ldots, \alpha_m)$ and $\beta = (\beta_1, \ldots, \beta_k)$ be in normal form. Then $[\alpha]_\approx \Cap [\beta]_\approx$ if and only if 

\begin{enumerate}

\item in case $m \not= k$, we have $\alpha_j = \beta_j$ for $j < \min \{m,k\}$, and if, e.g., $m < k$ then $\beta_m \in \alpha_m \Lambda_i$, where $\alpha_m$, $\beta_m \in \Lambda_i$ for some $i$ (i.e. one extends the other);

\item in case $m = k$, we have $\alpha_j = \beta_j$ for $j < m$; $\alpha_m$, $\beta_m \in \Lambda_i$ for some $i$; and $\alpha_m \Cap \beta_m$.

\end{enumerate}

\end{Proposition}

\begin{Corollary} \label{cor finite alignment in amalgamations}

$\Lambda$ is finitely aligned if and only if $\Lambda_i$ is finitely aligned for all $i \in I$.

\end{Corollary}

\section{The groupoid of an LCSC}
\label{sec groupoid}

Much of \cite[Section 4]{spi_pathcat} holds in the context of LCSC's without change.  We summarize the necessary facts here.  Let $\Lambda$ be an LCSC.  For $v \in \Lambda^0$ we recall the ring $\CA_v$ of subsets of $v\Lambda$ from Definition \ref{def lcsc}.  Let $A_v = \overline{\text{span}} \{\chi_E : E \in \CA_v\} \subseteq \ell^\infty(v\Lambda)$, a commutative $C^*$-algebra.

Recall that a \textit{filter} in a ring of sets $\CA$ is a nonempty collection, $\CU$, of nonempty elements of $\CA$, which is closed under the formation of intersections and supersets. An \textit{ultrafilter} is a maximal filter. Ultrafilters are characterized by the property: for each $E \in \CA$, either $E \in \CU$ or there is $F \in \CU$ with $E \cap F = \emptyset$. A \textit{filter base} is a nonempty collection of nonempty subsets of $\CA$ such that the intersection of any two of its elements is a superset of a third element. A filter base defines a unique filter by closing with respect to supersets. An \textit{ultrafilter base} is a filter base such that the filter it defines is an ultrafilter. Ultrafilter bases ($\CU_0$) are characterized by the property: for each $E \in \CA$, either there is $F \in \CU_0$ with $E \supseteq F$, or there is $F \in \CU_0$ with $E \cap F = \emptyset$. 

As in \cite[p. 250]{spi_graphalg} we identify $\widehat{A_v}$ with the set of ultrafilters in $\CA_v$:  the complex homomorphism $\omega \in \widehat{A_v}$ corresponds to the ultrafilter $\{ E \in \CA_v : \omega(\chi_E) = 1 \}$.  We will write  $X_v$ (or $X_v(\Lambda)$) for the spectrum, $\widehat{A_v}$, of $A_v$.  We let $\CU_x$ denote the ultrafilter corresponding to $x \in X_v$.  For $E \in \CA_v$ let $\widehat{E} \subseteq X_v$ denote the support of $\chi_E$.  Thus $x \in \widehat{E}$ if and only if $E \in \CU_x$.  The collection $\{\widehat{E} : E \in \CD_v\}$ is a base for the topology of $X_v$ consisting of compact-open sets.  We write $X = X(\Lambda) := \bigsqcup_{v \in \Lambda^0} X_v$. 

\begin{Definition} \label{def fixed ultrafilter}
(\cite[Section 4]{spi_pathcat})
A {\it fixed ultrafilter} is an ultrafilter of the form $\CU_{\{\alpha\}} := \{E \in \CA : \alpha \in E \}$, for some $\alpha \in \Lambda$.

\end{Definition}

\begin{Lemma} \label{lem fixed ultrafilters are dense}

The fixed ultrafilters are dense in $X$.

\end{Lemma}

\begin{proof}
Let $E \in \CA$, so that $\widehat{E}$ is a basic open set in $X$.  Choose $\gamma \in E$.  Then $E \in \CU_{\{\gamma\}}$, and hence $\CU_{\{\gamma\}} \in \widehat{E}$.
\end{proof}

The following is proved exactly as in \cite[Section 4]{spi_pathcat}.

\begin{Proposition} \label{prop maps on the spectrum}

Let $\Lambda$ be an LCSC.

\begin{enumerate}

\item Let $\alpha \in \Lambda$.  If $\CU$ is an ultrafilter in $\CA_{s(\alpha)}$ then $\tau^\alpha(\CU)$ is an ultrafilter base in $\CA_{r(\alpha)}$.

\item Let $\widehat{\tau^\alpha}(\CU)$ denote the ultrafilter generated by $\tau^\alpha(\CU)$.  Then $\widehat{\tau^\alpha} : X_{s(\alpha)} \to X_{r(\alpha)}$ is a homeomorphism onto its range, which is a compact-open subset of $X_{r(\alpha)}$. For $y \in X_{r(\alpha)}$, $y \in \widehat{\tau^\alpha}(X_{s(\alpha)})$ if and only if $\alpha \Lambda \in \CU_y$.

\item For $\zeta = (\alpha_1,\beta_1, \ldots, \alpha_n,\beta_n) \in \CZ_\Lambda v$ we write $\Phi_\zeta = \widehat{\tau^{\alpha_1}}^{-1} \circ \widehat{\tau^{\beta_1}} \circ \cdots \circ \widehat{\tau^{\alpha_n}}^{-1} \circ \widehat{\tau^{\beta_n}}$.  Then $\Phi_\zeta$ is a homeomorphism of $\widehat{A(\zeta)}$ onto $\widehat{A(\overline{\zeta})}$.  Moreover $\Phi_{\zeta_1 \zeta_2} = \Phi_{\zeta_1} \circ \Phi_{\zeta_2}$ and $\Phi_{\overline{\zeta}} = \Phi_\zeta^{-1}$.

\end{enumerate}

\end{Proposition}

The groupoid defined in \cite[Section 4]{spi_pathcat} is the groupoid of germs associated to the partial homeomorphisms $\Phi_\zeta$ of $X$.  In the presence of inverses, however, there can be loss of information when forming the groupoid.

\begin{Example} \label{example group old}

Let $\Lambda$ be a (discrete) group, viewed as an LCSC with one object, the unit 1.  For each $\lambda \in \Lambda$ we have $\lambda \Lambda = \Lambda$, and hence we have that $\CD_1^{(0)} = \{ \Lambda \}$, $\CA_1 = \{ \varnothing, \Lambda \}$, $\Phi_\lambda = $ id for all $\lambda$, and $X_1 = \{ \text{pt} \}$.  Therefore $G = \{ \text{pt} \}$.

\end{Example}

Phenomena as in this example may occur in various ways in LCSC's, and in some situations it may be desirable to use the above definition to ignore certain parts of the inverse structure.  However, there is a variation on the definition given in \cite{spi_pathcat} that does not ignore this structure.  Moreover in the case where there are no inverses, it gives the same result as the original definition.  We now present this alternate definition, and investigate its relation to the original.  In order to distinguish them, we will denote by $\sim_1$ the equivalence relation on $\CZ * X$ defined in \cite{spi_pathcat}, and by $\sim_2$ the alternate version.

\begin{Definition} \label{def groupoids}

\begin{enumerate}

\item Let $\Lambda$ be an LCSC.  Define $r : X \to \Lambda^0$ by $r(x) = v$ for $x \in X_v$. We set $\CZ * X = \{ (\zeta,x) \in \CZ \times X : s(\zeta) = r(x) \}$.

\item Define equivalence relations $\sim_1$ and $\sim_2$ on $\CZ * X$ by
\begin{align*}
(\zeta,x) &\sim_1 (\zeta',x') \text{ if } x = x' \text{ and } \Phi_\zeta|_{\widehat{E}} = \Phi_{\zeta'}|_{\widehat{E}} \text{ for some } E \in \CU_x, \\
(\zeta,x) &\sim_2 (\zeta',x') \text{ if } x = x' \text{ and } \varphi_\zeta|_E = \varphi_{\zeta'}|_E \text{ for some } E \in \CU_x.
\end{align*}

\item Define groupoids $G_i \equiv G_i(\Lambda)$ by $G_i = (\CZ * X)/\sim_i$ for $i = 1$, 2.  We let $[\zeta,x]_i$ (or just $[\zeta,x]$ if the context makes clear which groupoid is involved) denote the equivalence class of $(\zeta,x)$.  The set of composable pairs is $G_i^2 = \bigl\{ \bigl([\zeta,x]_i, \; [\zeta',x']_i \bigr) : x = \Phi_{\zeta'}x' \bigr\}$.  Composition and inversion are given by $[\zeta,\Phi_{\zeta'}x]_i[\zeta',x]_i = [\zeta\zeta',x]_i$ and $[\zeta,x]_i^{-1} = [\overline{\zeta},\Phi_\zeta x]_i$.

\end{enumerate}

\end{Definition}

It follows from Proposition \ref{prop maps on the spectrum} that $\Phi_\zeta(x) = y$ if and only if $\varphi_\zeta(\CU_x)$ is a filter base for $\CU_y$.

\begin{Remark}

We note that $s([\zeta,x]_i) = [\overline{\zeta},\Phi_\zeta(x)]_i [\zeta,x]_i = [id,x]_i$, so we may identify the unit space $G_i(\Lambda)^{(0)}$ with $X$.  Therefore the unit space of $G_i(\Lambda)$ depends only on the algebraic structure of $\Lambda$, and is independent of $i$.

\end{Remark}

\begin{Lemma} \label{lem comparable equivalence relations}

$\sim_2 \; \subseteq \; \sim_1$.

\end{Lemma}

\begin{proof}
Let $(\zeta,x) \sim_2 (\zeta',x')$.  We have $x = x'$ and $\varphi_\zeta|_E = \varphi_{\zeta'}|_E$ for some $E \in \CU_x$.  We will show that $\Phi_\zeta|_{\widehat{E}} = \Phi_{\zeta'}|_{\widehat{E}}$, which will imply that $(\zeta,x) \sim_1 (\zeta',x)$.  Let $y \in \widehat{E}$.  Then $E \in \CU_y$.  We know that $\CU_{\Phi_\zeta(y)}$ is generated by the filter base $\varphi_\zeta(\CU_y)$, or equivalently by $\varphi_\zeta(\CU_y \cap E)$.  By hypothesis this equals $\varphi_{\zeta'}(\CU_y \cap E)$, which generates the same filter as $\varphi_{\zeta'}(\CU_y)$, namely $\CU_{\Phi_{\zeta'}(y)}$.  Therefore $\Phi_\zeta(y) = \Phi_{\zeta'}(y)$.
\end{proof}

\begin{Definition} \label{def base for topology}

For $i = 1$, 2, and for $\zeta \in \CZ$ and $E \in \CA_{s(\zeta)}$, we let $[\zeta,E]_i = \{ [\zeta,x]_i : x \in \widehat{E} \}$.  Let $\CB_i = \{ [\zeta,E]_i : E \in \CA_{s(\zeta)} \}$.

\end{Definition}

The next proposition is adapted from \cite[Proposition 4.10]{spi_pathcat}.

\begin{Proposition} \label{pathcat 4.10}

$\CB_i$ is a base for a locally compact topology on $G_i$ for which the elements of $\CB_i$ are compact and Hausdorff, and making $G_i$ into an ample \'etale groupoid.

\end{Proposition}

\begin{proof}
We suppress the subscript indicating which of the two groupoids is involved.  Suppose $[\zeta,E] \cap [\zeta',E'] \not= \varnothing$.  Then $s(\zeta) = s(\zeta')$.  Let $[\zeta,x] \in [\zeta,E] \cap [\zeta',E']$.  Then $x \in \widehat{E} \cap \widehat{E'}$, and there is  $F \in \CA_{s(\zeta)}$ such that $x \in \widehat{F}$ and $\varphi_\zeta|_F = \varphi_{\zeta'}|_F$.  Then $[\zeta,x] \in [\zeta,E \cap E' \cap F] \subseteq [\zeta,E] \cap [\zeta',E']$.  Therefore $\CB$ is a base for a topology on $G$.

Next we show that multiplication and inversion are continuous.  Let $([\zeta,\Phi_{\zeta'}(x)],[\zeta,x]) \in G^2$.  Put $E = A(\zeta \zeta')$.  Then $E \in \CA_{r(x)}$ and $x \in \widehat{E}$, and $[\zeta \zeta', E]$ is a basic neighborhood of $[\zeta \zeta',x]$.  We have that $\varphi_{\zeta'}(E) \in \CA_{r(\zeta')}$ and $\Phi_{\zeta'}(x) \in \widehat{\varphi_{\zeta'}(E)}$.  Then $[\zeta,\varphi_{\zeta'}(E)]$ is a neighborhood of $[\zeta,\Phi_{\zeta'}(x)]$, $[\zeta,E]$ is a neighborhood of $[\eta',x]$, and $[\zeta, \varphi_{\zeta'}(E)] \cdot [\zeta',E] = [\zeta \zeta',E]$.  Therefore multiplication is continuous.  Also, $[\zeta,E]^{-1} = [\overline{\zeta},\phi_\zeta(E)]$, so inversion is continuous.

Finally, since $\Phi_\zeta : \widehat{A(\zeta)} \to \widehat{A(\overline{\zeta})}$ is a homeomorphism, it follows that $r|_{[\zeta,E]} : [\zeta,x] \mapsto \Phi_\zeta(x)$ is injective and open, and similarly for $s$.  Therefore $r$ and $s$ are local homeomorphisms.  Therefore $[\zeta,E]$ is compact and Hausdorff, and hence $G$ is \'etale and ample.
\end{proof}

\begin{Remark}

The results in the last part of \cite[Section 4]{spi_pathcat} giving simplifications for the case of a finitely aligned category of paths hold without change for LCSC's.

\end{Remark}

The next proposition is an explicit version of a parenthetical comment in the first paragraph of \cite{spi_pathcat}, section 6.  In that paper, the comment is valid for categories of paths, but not generally for relative categories of paths (see the remarks under {\it Relative LCSC's} in Section \ref{sec subcategories}).

\begin{Proposition} \label{prop lcsc without inverses}

Let $\Lambda$ be an LCSC without nontrivial inverses.  Then $\sim_2 = \sim_1$ on $\CZ*X$.

\end{Proposition}

\begin{proof}
We already observed in Lemma \ref{lem comparable equivalence relations} that $\sim_2 \subseteq \sim_1$.  Now let $(\zeta,x) \sim_1 (\zeta',x')$.  Then $x' = x$ and $\Phi_\zeta|_{\widehat{E}} = \Phi_{\zeta'}|_{\widehat{E}}$ for some $E \in \CU_x$.  We will show that $\varphi_\zeta|_E = \varphi_{\zeta'}|_E$.  Suppose not, i.e. suppose that there is $\alpha \in E$ such that $\varphi_\zeta(\alpha) \not= \varphi_{\zeta'}(\alpha)$.  Recall the notation for fixed ultrafilters from Definition \ref{def fixed ultrafilter}. Note that $\Phi_\zeta(\CU_{\{\alpha\}}) = \CU_{\{\varphi(\alpha)\}}$.  We will show generally that if $\lambda \not= \lambda'$ then $\CU_{\{\lambda\}} \not= \CU_{\{\lambda'\}}$.  Then it will follow that $\Phi_\zeta(\CU_{\{\alpha\}}) \not= \Phi_{\zeta'}(\CU_{\{\alpha\}})$, contradicting the assumption that $\Phi_\zeta|_{\widehat{E}} = \Phi_{\zeta'}|_{\widehat{E}}$.

Let $\lambda \not= \lambda'$.  Note that $\lambda \Lambda \in \CU_{\{\lambda\}}$.  If $\CU_{\{\lambda\}} = \CU_{\{\lambda'\}}$ then $\lambda \Lambda \in \CU_{\{\lambda'\}}$, and hence $\lambda' \in \lambda \Lambda$.  Similarly we have $\lambda \in \lambda' \Lambda$.  Therefore there are $\mu$, $\mu'$ such that $\lambda' = \lambda \mu$ and $\lambda = \lambda' \mu'$.  Then $\lambda = \lambda \mu \mu'$ and $\lambda' = \lambda' \mu' \mu$.  By left cancellation we obtain $\mu \mu' = s(\lambda)$ and $\mu' \mu = s(\lambda')$.  Since we are assuming that $\Lambda$ has no inverses it follows that $\mu = \mu' = s(\lambda) = s(\lambda')$, and hence that $\lambda = \lambda'$.  This finishes the proof.
\end{proof}






We return to Example \ref{example group old}, but using $\sim_2$.

\begin{Example} \label{example group new}

Let $\Lambda$ be a (discrete) group, viewed as an LCSC with one object, the unit 1.  We consider the groupoid $G_2(\Lambda)$.  For $\zeta = (\alpha_1, \beta_1, \ldots, \alpha_n, \beta_n)$ we have $\varphi_\zeta = \alpha_1^{-1} \beta_1 \cdots \alpha_n^{-1} \beta_n \cdot$ on $A(\zeta) = \Lambda$.  Then $(\zeta,1) \sim_2 (\zeta',1)$ if and only if $\varphi_\zeta = \varphi_{\zeta'}$, that is if and only if $\alpha_1^{-1} \beta_1 \cdots \alpha_n^{-1} \beta_n = (\alpha_1')^{-1} \beta_1' \cdots (\alpha_{n'}')^{-1} \beta_{n'}'$.  Therefore $G_2(\Lambda) = \Lambda$.

\end{Example}

Examples \ref{example group old} and \ref{example group new} demonstrate why we use $G_2$ rather than $G_1$ in the next definition (though for categories of paths they coincide).

\begin{Definition} \label{def toeplitz algebra}

Let $\Lambda$ be an LCSC.  The {\it Toeplitz algebra} of $\Lambda$ is the $C^*$-algebra $\CT(\Lambda) := C^*(G_2(\Lambda))$.

\end{Definition}

\section{The spectrum of an LCSC}
\label{sec spectrum}

The general description of the space $X = G_i^0$ given at the beginning of Section \ref{sec groupoid} can be made more precise.  In \cite[Section 7]{spi_pathcat} this is done for countable finitely aligned categories of paths.  Those results hold without change for finitely aligned LCSC's. (However we are able to prove some of these results without the hypothesis of countability used in \cite{spi_pathcat}, see Theorem \ref{thm spectrum finitely aligned}.)  We now give a more precise description for the general nonfinitely aligned case.  Let $\Lambda$ be an LCSC.  For $v \in \Lambda^0$ we let $\CD_v^{(0)}$ denote the collection of nonempty zigzag sets associated to $\CZ_\Lambda v$ (as in Definition \ref{def lcsc}).  Recall from Remark \ref{r.zigzag}\eqref{r.zigzag.e} that $\CD^{(0)}_v$ is closed under (nonempty) intersection.  We will show that there is a bijective correspondence between the set of ultrafilters in $\CA_v$ and certain filters in $\CD_v^{(0)}$.  As was the case in \cite[Section 7]{spi_pathcat} with the description of ultrafilters by means of directed hereditary subsets of $\Lambda$, it will be useful to have a suitable characterization of ultrafilters in the case of nonfinitely aligned LCSC's.

\begin{Definition}
\label{d.dirheredset}
A nonempty collection $C \subseteq \CD_v^{(0)}$ is a \textit{filter} if it is closed under intersection, and under the formation of supersets.

\end{Definition}

We modify Exel's terminology in the following (\cite[Section 11]{exel08}.

\begin{Definition}
\label{d.cover}
Let $C$ be a filter in $\CD_v^{(0)}$, and let $\CF \subseteq \CD_v^{(0)}$ be finite.  We say that $\CF$ \textit{covers} $C$ if there is $E \in C$ such that $E \subseteq \cup \CF$.

\end{Definition}

Now we can identify the spectrum of an LCSC.

\begin{Definition}
\label{d.lambdastar}
We let $v \Lambda^*$ denote the collection of filters $C$ in $\CD_v^{(0)}$ such that for any finite collection $\CF \subseteq \CD_v^{(0)}$, if $\CF \cap C = \varnothing$, then $\CF$ does not cover $C$.  For $C \in v \Lambda^*$, we define $\CU_C^0 = \{ E \setminus \cup \CF : E \in C,\ \CF \subseteq \CD_v^{(0)} \text{ finite, } \CF \text{ does not cover } C \}$.
  
\end{Definition}


\begin{Proposition}
\label{p.ultrafilterbase}
Let $C \in v \Lambda^*$.  Then $\CU_C^0$ is an ultrafilter base in $\CA_v$.

\end{Proposition}

\begin{proof}

The elements of $\CU_C^0$ are nonempty by construction.  If $A_i = E_i \setminus (\cup \CF_i) \in \CU_C^0$ for $i = 1$, 2, then $A_1 \cap A_2 = (E_1 \cap E_2) \setminus (\CF_1 \cup \CF_2)$.  Moreover, $E_1 \cap E_2 \in C$.  Since $\CF_i \cap C = \varnothing$ for $i = 1$, 2, then $(\CF_1 \cup \CF_2) \cap C = \varnothing$.  Therefore $\CF_1 \cup \CF_2$ does not cover $C$, and so $A_1 \cap A_2 \in \CU_C^0$.  Hence $\CU_C^0$ is a filter base.  To see that it is an ultrafilter base, let $A \in \CA_v$ be such that $A$ does not contain any element of $\CU_C^0$.  Write $A = \sqcup_{i=1}^k A_i$, where $A_i \in \CD_v$.  Then $A_i = E_i \setminus (\cup \CF_i)$, where $\{E_i\} \cup \CF_i \subseteq \CD_v^{(0)}$ with $\cup \CF_i \subseteq E_i$.  Now it suffices to find $B_i \in \CU_C^0$ for $i = 1$, $\ldots$, $k$, such that $B_i \cap A_i = \varnothing$; (for then, $B = \cap_i B_i \in \CU_C^0$, and $B \cap (\cup_i A_i) = \varnothing$).  So we may assume that $A = E \setminus \cup \CF$.

If $\CF$ covers $C$, then there is $G \in C$ with $G \subseteq \cup \CF$.  Then $G \in \CU_C^0$ and $G \cap A = \varnothing$.  So let us assume that $\CF$ does not cover $C$.  Since $E \setminus (\cup \CF) \not\in \CU_C^0$, it follows that $E \not\in C$.  But then $G \not\subseteq E$ for all $G \in C$.  Thus $\{ E \}$ does not cover $C$.  Then $v\Lambda \setminus E \in \CU_C^0$ and $(v\Lambda \setminus E) \cap A = \varnothing$.
\end{proof}

\begin{Definition}
\label{d.ultrafilter}
For $C \in v \Lambda ^*$ let $\CU_C$ be the ultrafilter in $\CA_v$ determined by the ultrafilter base $\CU_C^0$.

\end{Definition}

\begin{Proposition}
\label{p.ultrafilter}
The map $C \mapsto \CU_C$ is a bijection from $v \Lambda^*$ to the set of ultrafilters in $\CA_v$. Moreover, the inverse map is given by $\CU \mapsto \CU \cap \CD_v^{(0)}$.

\end{Proposition}

\begin{proof}

Let $\CU \subseteq \CA_v$ be an ultrafilter.  Let $C = \CU \cap \CD_v^{(0)}$.  Then $C$ is a filter in $\CD_v^{(0)}$.  Let $\CF \subseteq \CD_v^{(0)} \setminus C$ be finite.  Then $\CF \cap \CU = \varnothing$.  Since $\CU$ is an ultrafilter, for each $F \in \CF$ there is $A_F \in \CU$ such that $F \cap A_F = \varnothing$.  Then $A = \cap_{F \in \CF} A_F \in \CU$, and $A \cap (\cup \CF) = \varnothing$.  Then $\cup \CF \not\in \CU$, and hence $\CF$ does not cover $C$.  Therefore $C \in v \Lambda^*$.

Now the same argument shows that for any finite collection $\CF \subseteq \CD_v^{(0)}$ that does not cover $C$, $\cup \CF \not\in \CU$, and hence for any $E \in C$, we have $E \setminus (\cup \CF) \in \CU$.  This means that $\CU_C \subseteq \CU$.  Since $\CU_C$ is an ultrafilter, it follows that $\CU = \CU_C$.  Therefore the assignment in the statement is onto.

To see that it is one-to-one, let $C_1 \not= C_2$ in $v \Lambda^*$.  We may assume, e.g., that there is $E \in C_1 \setminus C_2$.  Since $C_2$ is closed under the formation of supersets, we know that $\{ E \}$ does not cover $C_2$, and hence that $E \not\in \CU_{C_2}$.  Since $E \in \CU_{C_1}$, we have $\CU_{C_1} \not= \CU_{C_2}$.
\end{proof}

It follows from Proposition \ref{p.ultrafilter} that $v \Lambda^* = X$, the unit space of the groupoid $G(\Lambda)$.

It was shown in \cite[Section 7]{spi_pathcat} that if $\Lambda$ is a finitely aligned category of paths then $X$ can be identified with the collection of all directed hereditary subsets of $\Lambda$.  In fact more was proved there, and the proofs relied on the assumption that $\Lambda$ be countable.  Here we use the above results for the general case of an LCSC to prove the above mentioned identification without using countability. We recall \cite[Definition 7.1]{spi_pathcat}.

\begin{Definition}

Let $\Lambda$ be an LCSC.  A nonempty subset $H \subseteq \Lambda$ is {\it directed} if for all $\alpha$, $\beta \in H$, $H \cap \alpha\Lambda \cap \beta\Lambda \not= \varnothing$, and $H$ is {\it hereditary} if $[\alpha] \subseteq H$ for all $\alpha \in H$.

\end{Definition}

\begin{Theorem} \label{thm spectrum finitely aligned}
(c.f. \cite[Theorem 7.6]{spi_pathcat}) Let $\Lambda$ be a finitely aligned LCSC, and let $v \in \Lambda^0$.

\begin{enumerate}

\item \label{thm spectrum finitely aligned one} For $C \in v \Lambda^*$ let $H(C) = \{ \alpha \in v\Lambda : \alpha \Lambda \in C \}$.  Then $C \mapsto H(C)$ is a bijection between $v \Lambda^*$ and the collection of all directed hereditary subsets of $v\Lambda$.  The inverse map is given by $H \mapsto \{E \in \CD^{(0)}_v : E \supseteq \alpha \Lambda \text{ for some } \alpha \in H \}$.  Moreover, $C_1 \subseteq C_2$ if and only if $H(C_1) \subseteq H(C_2)$.

\item \label{thm spectrum finitely aligned two} For $C \in v \Lambda^*$, $\CU_C = \{ A \in \CA_v : A \supseteq H(C) \cap \alpha \Lambda \text{ for some } \alpha \in H(C) \}$.

\end{enumerate}

\end{Theorem}

\begin{proof}
\eqref{thm spectrum finitely aligned one}: Let $C \in v \Lambda^*$ and define $H(C)$ as in the statement.  Let $\alpha$, $\beta \in H(C)$.  Then $\alpha \Lambda$, $\beta \Lambda \in C$. Since $C$ is a filter, $\alpha\Lambda \cap \beta\Lambda \in C$. We may write $\alpha \Lambda \cap \beta \Lambda = \bigcup_{\gamma \in \alpha \vee \beta} \gamma \Lambda$.  Then $\{ \gamma \Lambda : \gamma \in \alpha \vee \beta \}$ covers $C$.  Since $C \in v \Lambda^*$ there is $\gamma \in \alpha \vee \beta$ such that $\gamma \Lambda \in C$.  Then $\gamma \in H(C)$.  Therefore $H(C)$ is directed.  If $\delta \in [\alpha]$ then $\alpha \Lambda \subseteq \delta \Lambda$.  Since $C$ is a filter, $\delta \Lambda \in C$, and hence $\delta \in H(C)$. Therefore $H(C)$ is hereditary.

Let $C_1$, $C_2 \in v \Lambda^*$ and suppose that $H(C_1) = H(C_2)$.  Let $E \in C_1$.  By Corollary \ref{cor finitely aligned zigzag set} there are $\alpha_1$, $\ldots$, $\alpha_k \in v\Lambda$ such that $E = \bigcup_{i=1}^k \alpha_i \Lambda$.  Since $C \in v \Lambda^*$ there is some $i$ such that $\alpha_i \Lambda \in C_1$.  Then $\alpha_i \in H(C_1) = H(C_2)$, and hence $\alpha_i \Lambda \in C_2$.  But then $E \in C_2$ since $C_2$ is a filter.  Therefore $C_1 \subseteq C_2$.  By symmetry, $C_1 = C_2$, and the map is injective.

Now let $H \subseteq v\Lambda$ be directed and hereditary. Put $C = \{E \in \CD^{(0)}_v : E \supseteq \alpha\Lambda \text{ for some } \alpha \in H\}$.  We show that $C \in v \Lambda^*$.  Let $E_1$, $E_2 \in C$.  There are $\alpha_1$, $\alpha_2 \in H$ such that $\alpha_i \Lambda \subseteq E_i$.  Then $\alpha_1 \Lambda \cap \alpha_2 \Lambda \subseteq E_1 \cap E_2$.  Since $H$ is directed there is $\gamma \in \alpha_1 \vee \alpha_2$ such that $\gamma \in H$.  Then $\gamma \Lambda \subseteq E_1 \cap E_2$, so $E_1 \cap E_2 \in C$.  Therefore $C$ is closed under intersection.  Let $F \in \CD^{(0)}_v$ with $E_1 \subseteq F$.  Then $F \supseteq \alpha_1\Lambda$, so $F \in C$.  Therefore $C$ is a filter.  Now let $F_1$, $\ldots$, $F_k \in \CD^{(0)}_v$ cover $C$.  There is $E \in C$ such that $E \subseteq \bigcup_{i=1}^k F_k$.  Since $E \in C$ there is $\alpha \in H$ such that $\alpha \Lambda \subseteq E$.  Then $\alpha \in F_i$ for some $i$.  But then $\alpha \Lambda \subseteq F_i$, so $F_i \in C$.  Therefore $C \in v \Lambda^*$.  It is clear that $H = H(C)$.  Therefore the map is surjective, with inverse as described in \eqref{thm spectrum finitely aligned one}.

Finally, for $C_1$, $C_2 \in v \Lambda^*$ it is clear that $C_1 \subseteq C_2$ if and only if $H(C_1) \subseteq H(C_2)$.

\noindent
\eqref{thm spectrum finitely aligned two} Let $C \in v \Lambda^*$. We begin with the forward containment.  $\CU_C$ is generated by the filter base $\CU^0 = \{E \setminus \bigcup_{i=1}^n F_i : E \in C,\ F_1, \ldots, F_n \in \CD^{(0)}_v \text{ do not cover } C \}$.  Note that if $E \setminus \bigcup_{i=1}^n F_i \in \CU^0$ then $(\bigcup_{i=1}^n F_i) \cap H(C) = \varnothing$.  The reason is that if $\alpha \in F_i \cap H(C)$ for some $i$, then $\alpha \Lambda \subseteq F_i$ and $\alpha \Lambda \in C$, contradicting the assumption that $F_1$, $\ldots$, $F_n$ do not cover $C$.  Now, since $E \in C$, $E \supseteq \alpha \Lambda$ for some $\alpha \in H(C)$.  Then $E \supseteq \alpha\Lambda \cap H(C)$.  Therefore $E \setminus \bigcup_{i=1}^n F_i \supseteq \alpha \Lambda \cap H(C)$.  Since every element of $\CU_C$ contains an element of $\CU^0$, the containment follows.

Now we prove the reverse containment.  Let $A \in \CA_v$ with $A \supseteq \alpha\Lambda \cap H(C)$ for some $\alpha \in H(C)$.  We have $A = \bigsqcup_{j=1}^m A_j$ for $A_j \in \CD_v$, so there is $j$ such that $\alpha \in A_j$.  It suffices to show that $A_j \in \CU_C$.  Let $A_j = E \setminus \bigcup_{i=1}^n F_i$ with $E$, $F_i \in \CD^{(0)}_v$.  We claim that $\bigl( \bigcup_{i=1}^n F_i \bigr) \cap H(C) = \varnothing$.  The reason is that $\alpha \Lambda \cap H(C) \cap \bigl( \bigcup_{i=1}^n F_i \bigr) = \varnothing$ since $\alpha \Lambda \cap H(C) \subseteq A_j$.  If $\gamma \in H(C)$ then there is $\delta \in \gamma \Lambda \cap \alpha \Lambda \cap H(C)$.  Then $\delta \not\in \bigcup_{i=1}^n F_i$.  If $\gamma \in F_i$ for some $i$ then $\delta \in F_i$, a contradiction.  Thus $\gamma \not\in \bigcup_{i=1}^n F_i$.  Therefore $\gamma \Lambda \not\subseteq \bigcup_{i=1}^n F_i$ for all $\gamma \in H(C)$.  Then $F_1$, $\ldots$, $F_n$ do not cover $C$, so $\alpha \Lambda \setminus \bigcup_{i=1}^n F_i \in \CU_C$.  Therefore $A_i \in \CU_C$.
\end{proof}


\section{The Hausdorff property for the groupoid}
\label{sec hausdorff}

\begin{Lemma} \label{lem finitely aligned hausdorff}

Let $\Lambda$ be a finitely aligned LCSC with right cancellation. Then $G_1(\Lambda)$ and $G_2(\Lambda)$ are Hausdorff groupoids.

\end{Lemma}

\begin{proof}
The proof of \cite[Proposition 4.17]{spi_pathcat} implies that $G_1(\Lambda)$ is Hausdorff.  We next consider $G_2(\Lambda)$.  Let $[\zeta_1,x_1]_2 \not= [\zeta_2,x_2]_2$.  If $x_1 \not= x_2$ then there are $E_i \in \CU_{x_i}$ with $E_1 \cap E_2 = \varnothing$, and hence $[\zeta_1,E_1]_2$ and $[\zeta_2,E_2]_2$ are disjoint neighborhoods of $[\zeta_1,x_1]_2$ and $[\zeta_2,x_2]_2$.  Therefore we may as well assume that $x_1 = x_2 =: x$.  Let $\zeta = \overline{\zeta_2} \zeta_1$.  Then $[\zeta,x]_2 \not= [\text{id},x]_2$.  By Lemma \ref{lem finitely aligned zigzag map} we know that $\phi_\zeta = \bigvee_{i=1}^n \tau^{\alpha_i} \circ \sigma^{\beta_i}$ for some $\alpha_i$, $\beta_i \in \Lambda$.  Then $A(\zeta) = \bigcup_{i=1}^n \beta_i \Lambda$.  Since $x \in \widehat{A(\zeta)} = \bigcup_{i=1}^n \widehat{\beta_i \Lambda}$ there is $j$ such that $x \in \widehat{\beta_j \Lambda}$. Since $[\zeta,x]_2 \not= [\text{id},x]_2$ there is $\gamma \in \beta_j \Lambda$ such that $\varphi_\zeta (\gamma) \not= \gamma$.  Then $\gamma = \beta_j \gamma'$, and hence $\varphi_\zeta(\gamma) = \alpha_j \gamma'$.  Therefore $\alpha_j \not= \beta_j$.  By right cancellation it follows that $\varphi_\zeta(\delta) \not= \delta$ for all $\delta \in \beta_j \Lambda$.  Now we claim that $[\zeta,y]_2 \not= [\text{id},y]_2$ for all $y \in \widehat{\beta_j \Lambda}$.  To see this, let $y \in \widehat{\beta_j \Lambda}$.  Then $\beta_j \Lambda \in \CU_y$.  If $E \in \CU_y$, i.e. if $\widehat{E}$ is a neighborhood of $y$, then $E \cap \beta_j \Lambda \in \CU_y$, hence $E \cap \beta_j \Lambda \not= \varnothing$.  Then $\varphi_\zeta(\delta) \not= \delta$ for all $\delta \in E \cap \beta_j \Lambda$, verifying the claim. It follows that $[\zeta_1,y]_2 \not= [\zeta_2,y]_2$ for all $y \in \widehat{\beta_j \Lambda}$, a neighborhood of $x$.  Therefore $[\zeta_1,\beta_j \Lambda]_2 \cap [\zeta_2, \beta_j \Lambda]_2 = \varnothing$.
\end{proof}

  We next consider the case where $\Lambda$ is not finitely aligned.  We recall that in the nonfinitely aligned case, right cancellation is in general insufficient to guarantee that $G_1(\Lambda)$ is Hausdorff even when $\Lambda$ has no inverses (\cite{spi_pathcat}, Example 4.11).

\begin{Theorem} \label{thm contained in groupoid hausdorff}

Let $\Lambda$ be an LCSC such that $\Lambda$ is a subcategory of a groupoid.  Then $G_2(\Lambda)$ is Hausdorff.

\end{Theorem}

\begin{proof}
Let $[\zeta_1,x_1] \not= [\zeta_2,x_2]$.  If $x_1 \not= x_2$ then there are $E_i \in \CU_{x_i}$ with $E_1 \cap E_2 = \varnothing$, and hence $[\zeta_1,E_1]$ and $[\zeta_2,E_2]$ are disjoint neighborhoods of $[\zeta_1,x_1]$ and $[\zeta_2,x_2]$.  Therefore we may as well assume that $x_1 = x_2 =: x$.  Let $\zeta_1 = (\alpha_1,\beta_1, \ldots, \alpha_m, \beta_m)$ and $\zeta_2 = (\gamma_1,\delta_1, \ldots, \gamma_n, \delta_n)$, and put $s_1 = \alpha_1^{-1} \beta_1 \cdots \alpha_m^{-1} \beta_m$, $s_2 = \gamma_1^{-1} \delta_1 \cdots \gamma_n^{-1} \delta_n$.  For $\eps \in A(\zeta_i)$ we have $\varphi_{\zeta_i}(\eps) = s_i \eps$.  If $s_1 = s_2$ then $\varphi_{\zeta_1}|_{A(\zeta_1) \cap A(\zeta_2)} = \varphi_{\zeta_2}|_{A(\zeta_1) \cap A(\zeta_2)}$, contradicting the assumption that $[\zeta_1,x] \not= [\zeta_2,x]$.  Therefore $s_1 \not= s_2$.  But then $\varphi_{\zeta_1}(\alpha) \not= \varphi_{\zeta_2}(\alpha)$ for all $\alpha \in A(\zeta_1) \cap A(\zeta_2)$.  Therefore $[\zeta_1,A(\zeta_1) \cap A(\zeta_2)] \cap [\zeta_2,A(\zeta_1) \cap A(\zeta_2)] = \varnothing$.
\end{proof}

We now prove that the previous result holds for $G_1(\Lambda)$ if we assume in addition that $\Lambda$ has no inverses.

\begin{Corollary} \label{cor hausdorff}

Let $\Lambda$ be an LCSC, with $\Lambda$ a subcategory of a groupoid $Y$, and suppose that $\Lambda \cap \Lambda^{-1} \subseteq Y^0$.  Then the groupoid $G_1(\Lambda)$ is Hausdorff.

\end{Corollary}

\begin{proof}
Since $\Lambda$ does not contain (nontrivial) inverses, Proposition \ref{prop lcsc without inverses} implies that $\sim_1 = \sim_2$.  Then the Corollary follows from Theorem \ref{thm contained in groupoid hausdorff}.
\end{proof}

\begin{Remark}

We have not been able to determine whether  $G_1(\Lambda)$ is Hausdorff in general for subcategories of groupoids.  We give the following criterion for equality of $G_1(\Lambda)$ and $G_2(\Lambda)$.  Thus if a subcategory $\Lambda$ of a groupoid satisfies this criterion then $G_1(\Lambda)$ is Hausdorff.

\end{Remark}

\begin{Proposition} \label{prop equality of groupoids}

Let $\Lambda$ be an LCSC.  Consider the conditions

\begin{enumerate}

\item \label{prop equality of groupoids one} $\sim_1 \; = \; \sim_2$.

\item \label{prop equality of groupoids two} For each nonempty set $E \in \CA$, for each $\alpha \in E$, and for each $\mu \in s(\alpha) \Lambda^{-1} s(\alpha)$ with $\mu \not= s(\alpha)$, there is an element $\beta \in s(\alpha)\Lambda$ such that $\alpha \beta \in E$ and either $\mu \beta \not\in \beta \Lambda$ or $\mu^{-1} \beta \not\in \beta \Lambda$.

\end{enumerate}

Then \eqref{prop equality of groupoids two} implies \eqref{prop equality of groupoids one}.  Moreover, if $\Lambda$ has right cancellation then \eqref{prop equality of groupoids one} implies \eqref{prop equality of groupoids two}.

\end{Proposition}

\begin{proof}
\eqref{prop equality of groupoids two} $\Rightarrow$ \eqref{prop equality of groupoids one}:  We must show that $\sim_1 \subseteq \sim_2$.  Since $\varphi_\zeta$ and $\Phi_\zeta$ are bijections it follows that $(\zeta,x) \sim_i (\zeta',x)$ if and only if $(\overline{\zeta} \zeta',x) \sim_i (\text{id},x)$, for $i = 1$, 2.  So we assume that $(\zeta,x) \sim_1 (\text{id},x)$ and prove that $(\zeta,x) \sim_2 (\text{id},x)$.  Suppose to the contrary that $(\zeta,x) \not\sim_2 (\text{id},x)$.  Then for all $E \in \CA$ such that $x \in \widehat{E}$, there is $\alpha \in E$ such that $\varphi_\zeta(\alpha) \not= \alpha$.  By our assumption there is such an $E$ with $\Phi_\zeta|_{\widehat{E}} = \text{id}$. Choose $\alpha \in E$ as above.  We let $C_\alpha = \{ D \in \CD^{(0)}_{r(\alpha)} : \alpha \in D \}$.  Then $C_\alpha \in \Lambda^*$, and $\CU_{C_\alpha} = \{ A \in \CA_{r(\alpha)} : \alpha \in A \}$.  Moreover, $E \in \CU_{C_\alpha}$ and $C_\alpha \in \widehat{E}$.  Then $\Phi_\zeta(\CU_{C_\alpha}) = \CU_{C_\alpha}$, and hence for $A \in \CA_{r(\alpha)}$ we have that $\alpha \in A$ if and only if $\varphi_\zeta(\alpha) \in A$.  Applying this to $A = \alpha \Lambda$, and to $A = \varphi_\zeta(\alpha) \Lambda$, we obtain that $\varphi_\zeta(\alpha) \in \alpha \Lambda$ and $\alpha \in \varphi_\zeta(\alpha) \Lambda$.  Thus there are $\mu$ and $\mu'$ such that $\varphi_\zeta(\alpha) = \alpha \mu$ and $\alpha = \varphi_\zeta(\alpha) \mu'$.  Then $\alpha = \alpha \mu \mu'$ and $\varphi_\zeta(\alpha) = \varphi_\zeta(\alpha) \mu' \mu$, and hence $\mu\mu' = s(\alpha)$, $\mu'\mu = s(\varphi_\zeta(\alpha)) = s(\alpha)$.  Therefore $\varphi_\zeta(\alpha) = \alpha \mu$ with $\mu \in s(\alpha) \Lambda^{-1} s(\alpha)$.

Choose $\beta$ as in \eqref{prop equality of groupoids two}.  Thus $\alpha \beta \in E$, and either $\mu \beta \not\in \beta \Lambda$ or $\mu' \beta \not\in \beta \Lambda$.  Suppose for definiteness that $\mu \beta \not\in \beta \Lambda$.  Recall from Remark \ref{r.zigzag}\eqref{r.zigzag.f} that for $\alpha \in A(\zeta)$ we have $\varphi_\zeta(\alpha \beta) = \varphi_\zeta(\alpha) \beta$.  Then $\varphi_\zeta(\alpha \beta) = \varphi_\zeta(\alpha) \beta = \alpha \mu \beta \not\in \alpha \beta \Lambda$.  Now the previous argument applies to $\alpha \beta$: there is $\nu \in s(\beta) \Lambda^{-1} s(\beta)$ such that $\varphi_\zeta(\alpha \beta) = \alpha \beta \nu$.  Therefore $\alpha \mu \beta = \alpha \beta \nu$.  By left cancellation we have $\mu \beta = \beta \nu \in \beta \Lambda$, a contradiction.

\eqref{prop equality of groupoids one} $\Rightarrow$ \eqref{prop equality of groupoids two}:  Now suppose that $\Lambda$ has right cancellation.  Suppose that \eqref{prop equality of groupoids two} fails.  Then there are $E \in \CA$, $\alpha \in E$, and $\mu \in s(\alpha) \Lambda^{-1} s(\alpha)$ with $\mu \not= s(\alpha)$, such that for all $\beta \in s(\alpha) \Lambda$, if $\alpha \beta \in E$ then $\mu \beta$, $\mu^{-1} \beta \in \beta \Lambda$.  Then there are $\nu$, $\nu' \in \Lambda$ such that $\mu \beta = \beta \nu$ and $\mu^{-1} \beta = \beta \nu'$.  Note that it follows from right cancellation that $\nu \not= s(\beta)$.  Moreover $\beta = \mu^{-1} \beta \nu = \beta \nu' \nu$ and $\beta = \mu \beta \nu' = \beta \nu \nu'$.  By left cancellation we have $\nu' \nu = s(\beta) = \nu \nu'$, so $\nu \in s(\beta) \Lambda^{-1} s(\beta)$.  Let $\zeta = (s(\alpha), \alpha, \alpha \mu^{-1}, s(\alpha)) \in \CZ$.  Then $\varphi_\zeta (\alpha) = \varphi_\zeta (\alpha \mu^{-1} \mu) = \alpha \mu \not= \alpha$.  Hence if $\alpha \beta \in E$ then $\phi_\zeta(\alpha \beta) = \alpha \mu \beta = \alpha \beta \nu \not= \alpha \beta$.  Therefore $\phi_\zeta(\gamma) \not= \gamma$ for all $\gamma \in E' := \alpha \Lambda \cap E$.  Hence $[\zeta,\CU_{\{\gamma\}}]_2 \not= [\text{id},\CU_{\{\gamma\}}]_2$ for all $\gamma \in E'$.  By Lemma \ref{lem fixed ultrafilters are dense} we have that $[\zeta,x]_2 \not= [\text{id},x]_2$ for all $x \in \widehat{E'}$.

On the other hand we claim that $[\zeta,x]_1 = [\text{id},x]_1$ for all $x \in \widehat{E'}$.  To see this, note that for $\alpha \beta \in E'$ we have $[\alpha \beta] \subseteq [\alpha \beta \nu] \subseteq [\alpha \beta \nu \nu^{-1}] = [\alpha \beta]$.  Therefore $\Phi_\zeta(\CU_{\{\gamma\}}) = \CU_{\{\gamma\}}$ for all $\gamma \in E'$.  By Lemma \ref{lem fixed ultrafilters are dense} it follows that $\Phi_\zeta|_{\widehat{E'}} = \text{id}$, proving the claim. 
\end{proof}

In Section \ref{sec ore theorem} we study {\it right reversibility} for LCSC's (Definition \ref{def right reversible}), and in particular we prove an analog of Ore's theorem, Theorem \ref{thm ore theorem}.  That theorem implies that if a small category is cancellative and right reversible, then it embeds into a groupoid.  The results of this section give the following.

\begin{Corollary} \label{cor right reversible hausdorff}

Let $\Lambda$ be a right cancellative right reversible LCSC.  Then $G_2(\Lambda)$ is Hausdorff.

\end{Corollary}

\begin{Corollary} \label{cor abelian hausdorff}

Let $\Lambda$ be an abelian LCSC.  Then $G_2(\Lambda)$ is Hausdorff.

\end{Corollary}

\begin{proof}
Since $\Lambda$ is abelian it is right cancellative and right reversible.
\end{proof}

\section{Boolean ring homomorphisms}
\label{sec boolean}

As pointed out in \cite{dm}, the zigzag maps form an inverse semigroup.  One can study the $C^*$-algebra that is universal for representations of this inverse semigroup.  However this ignores important structure that becomes significant in many interesting examples.  In \cite{spi_pathcat} it is shown how this extra structure develops naturally from the shift maps $\tau$ and $\sigma$ on $\Lambda$ (in \cite{dm} this is referred to as ``finite join preservation'' --- see also \cite[Subsection 7.3]{bkqs}).  The two results are quoted here; the proofs are identical to those given in \cite[Section 5]{spi_pathcat}.

\begin{Theorem} \label{thm boolean ring homomorphism}

Let $\Lambda$ be an LCSC, $v \in \Lambda^0$, and $\CR$ a Boolean ring.  A map $\mu : \CD^{(0)}_v \to \CR$ extends to a Boolean ring homomorphism $\CA_v \to \CR$ if and only if the following conditions hold:

\begin{enumerate}

\item $\mu(E\cap F) = \mu(E) \cap \mu(F)$, for $E$, $F \in \CD^{(0)}_v$.
\label{thm boolean ring homomorphism 1}

\item $\mu(E) = \bigcup_{i=1}^n \mu(F_i)$ for $E$, $F_1$, $\ldots$, $F_n \in \CD^{(0)}_v$ with $E = \bigcup_{i=1}^n F_i$.
\label{thm boolean ring homomorphism 2}

\end{enumerate}

In this case, the extension to $\CA_v$ is unique.

\end{Theorem}


\begin{Theorem} \label{thm boolean ring homomorphism finitely aligned}

Let $\Lambda$ be a finitely aligned LCSC, $v \in \Lambda^0$, $\CR$ a Boolean ring, and $\mu : \CE^{(0)}_v \to \CR$.  Then $\mu$ extends to a Boolean ring homomorphism $\CA_v \to \CR$ if and only if the following condition holds:

\begin{enumerate}
\setcounter{enumi}{2}

\item $\mu(\alpha\Lambda) \cap \mu(\beta\Lambda) = \bigcup_{\gamma \in \alpha \vee \beta} \mu(\gamma\Lambda)$ for all $\alpha$, $\beta \in v\Lambda$.
\label{thm boolean ring homomorphism finitely aligned 3}

\end{enumerate}

In this case, the extension is unique.

\end{Theorem}

We remark that the union in Theorem \ref{thm boolean ring homomorphism finitely aligned} is finite even if $\alpha \vee \beta$ is infinite, since $\gamma \Lambda = \gamma' \Lambda$ if $\gamma \approx \gamma'$.

\section{Generators and relations}
\label{sec gens and rels}

In this section we give the presentation of $C^*(G_i(\Lambda))$ by generators and relations.  Much of the argument is as in \cite[Section 6]{spi_pathcat}; however, the distinctions caused by the two equivalence relations are a bit subtle, and we feel it is best to give the arguments in full.  We will use $\CZ$, $G_1$, $G_2$, etc., instead of $\CZ_\Lambda$, etc., when the LCSC is understood from the context.  (However, we will write $[\zeta,E]$ without a subscript whenever it is clear from context which groupoid is under consideration.)

\begin{Definition} \label{def toeplitz relations}

Let $\Lambda$ be an LCSC.  Let $\{ T_\zeta : \zeta \in \CZ \}$ be a family of elements of a $C^*$-algebra.  We will consider the following relations on the $T_\zeta$:

\begin{itemize}

\item[(1)\phantom{$_2$}] $T_{\zeta_1} T_{\zeta_2} = T_{\zeta_1 \zeta_2}$ if $s(\zeta_1) = r(\zeta_2)$, and equals 0 otherwise.

\item[(2)\phantom{$_2$}] $T_{\overline{\zeta}} = T_\zeta^*$.

\item[(3)\phantom{$_2$}] $T_\zeta^* T_\zeta = \bigvee_{i=1}^n T_{\zeta_i}^* T_{\zeta_i}$ if $A(\zeta) = \bigcup_{i=1}^n A(\zeta_i)$.

\item[(4)$_1$] $T_\zeta = T_\zeta^* T_\zeta$ if $\Phi_\zeta = \text{id}_{\widehat{A(\zeta)}}$.

\item[(4)$_2$] $T_\zeta = T_\zeta^* T_\zeta$ if $\varphi_\zeta = \text{id}_{A(\zeta)}$.

\end{itemize}

\end{Definition}

We note the following elementary consequences.

\begin{Lemma} \label{lem elementary consequences}

\begin{enumerate}

\item \label{lem elementary consequences 1} Suppose that $\{ T_\zeta : \zeta \in \CZ \}$ satisfy (3). If $A(\zeta) = \varnothing$ then $T_\zeta = 0$.

\item \label{lem elementary consequences 2} Suppose that $\{ T_\zeta : \zeta \in \CZ \}$ satisfy (1) and (2).  For $u \in \Lambda^0$ let $T_u := T_{(u,u)}$.  Then $\{ T_u : u \in \Lambda^0 \}$ are pairwise orthogonal projections.

\item \label{lem elementary consequences 3}  Suppose that $\{ T_\zeta : \zeta \in \CZ \}$ satisfy (1) - (3).  Then $T_\zeta$ is a partial isometry.

\item \label{lem elementary consequences 4} Suppose that $\{ T_\zeta : \zeta \in \CZ \}$ satisfy (1) - (3) and (4)$_2$.  If $\varphi_{\zeta_1} = \varphi_{\zeta_2}$ then $T_{\zeta_1} = T_{\zeta_2}$.

\item \label{lem elementary consequences 5} Suppose that $\{ T_\zeta : \zeta \in \CZ \}$ satisfy (1) - (3) and (4)$_1$.  If $\Phi_{\zeta_1} = \Phi_{\zeta_2}$ then $T_{\zeta_1} = T_{\zeta_2}$.

\item \label{lem elementary consequences 6} Suppose that $\{ T_\zeta : \zeta \in \CZ \}$ satisfy (1) - (3) and (4)$_2$.  If $\varphi_\zeta = \bigcup_{i=1}^n \varphi_{\zeta_i}$ then $T_\zeta = \bigvee_{i=1}^n T_{\zeta_i}$.

\item \label{lem elementary consequences 7} Suppose that $\{ T_\zeta : \zeta \in \CZ \}$ satisfy (1) - (3) and (4)$_1$.  If $\Phi_\zeta = \bigcup_{i=1}^n \Phi_{\zeta_i}$ then $T_\zeta = \bigvee_{i=1}^n T_{\zeta_i}$.

\end{enumerate}

\end{Lemma}

\begin{proof}
\eqref{lem elementary consequences 1}: If $A(\zeta) = \varnothing$ then $A(\zeta)$ is the union of the empty collection (of zigzag sets).  By (3) then $T_\zeta^* T_\zeta$ equals the empty sum (of operators), i.e. $T_\zeta^* T_\zeta = 0$.  Therefore $T_\zeta = 0$.

\noindent
\eqref{lem elementary consequences 2}: Since $\overline{(u,u)} = (u,u) = (u,u)(u,u)$, it follows from (1) and (2) that $T_u$ is a projection.  The orthogonality follows from (1).

\noindent
\eqref{lem elementary consequences 3}: First note that by (1) and (2) we have $T_{\overline{\zeta} \zeta} = T_\zeta^* T_\zeta$ is self-adjoint.  Since $A(\zeta) = A(\overline{\zeta}\zeta)$, by (3) we have $T_\zeta^* T_\zeta = T_{\overline{\zeta}\zeta}^* T_{\overline{\zeta}\zeta} = T_{\overline{\zeta}\zeta}^2$.  Therefore $T_{\overline{\zeta}\zeta}$ is a projection, and hence $T_\zeta$ is a partial isometry.

\noindent
\eqref{lem elementary consequences 4}: Suppose that (1) - (3), (4)$_2$ hold, and let $\varphi_{\zeta_1} = \varphi_{\zeta_2}$.  Then $A(\zeta_1) = A(\zeta_2)$ and $A(\overline{\zeta_1}) = A(\overline{\zeta_2})$.  By (3) and (2) we have $T_{\zeta_1}^* T_{\zeta_1} = T_{\zeta_2}^* T_{\zeta_2}$ and $T_{\zeta_1} T_{\zeta_1}^* = T_{\zeta_2} T_{\zeta_2}^*$.  Since $\varphi_{\overline{\zeta_1} \zeta_2} = \text{id}_{A(\overline{\zeta_1} \zeta_2)}$, (4)$_2$ gives $T_{\overline{\zeta_1} \zeta_2} = T_{\overline{\zeta_1} \zeta_2}^* T_{\overline{\zeta_1} \zeta_2}$.  By (1), (2), and the above we get
\[
T_{\zeta_1}^* T_{\zeta_2}
= T_{\zeta_2}^* T_{\zeta_1} T_{\zeta_1}^* T_{\zeta_2}
= T_{\zeta_2}^* (T_{\zeta_2} T_{\zeta_2}^*) T_{\zeta_2}
= T_{\zeta_2}^* T_{\zeta_2}
= T_{\zeta_1}^* T_{\zeta_1}.
\]
Now we have $T_{\zeta_1} = T_{\zeta_1} (T_{\zeta_1}^* T_{\zeta_1}) = T_{\zeta_1} (T_{\zeta_1}^* T_{\zeta_2}) = (T_{\zeta_2} T_{\zeta_2}^*) T_{\zeta_2} = T_{\zeta_2}$.

\noindent
\eqref{lem elementary consequences 5}: A similar argument shows that if $\Phi_{\zeta_1} = \Phi_{\zeta_2}$ then (1) - (3) and (4)$_1$ give the same conclusion.

\noindent
\eqref{lem elementary consequences 6}: For each $i$ we have $\varphi_{\zeta \overline{\zeta_i} \zeta_i} = \varphi_{\zeta_i}$.  By \eqref{lem elementary consequences 4}, and (1) and (2), we have $T_\zeta T_{\zeta_i}^* T_{\zeta_i} = T_{\zeta_i}$.  Then $T_\zeta \bigvee_{i=1}^n T_{\zeta_i}^* T_{\zeta_i} = \bigvee_{i=1}^n T_{\zeta_i}$.  Also we have that $A(\zeta) = \bigcup_{i=1}^n A(\zeta_i)$.  By (3) we have $T_\zeta^* T_\zeta = \bigvee_{i=1}^n T_{\zeta_i}^* T_{\zeta_i}$.  Thus we have
\[
T_\zeta = T_\zeta (T_{\zeta}^* T_\zeta) = T_\zeta \bigvee_{i=1}^n T_{\zeta_i}^* T_{\zeta_i} = \bigvee_{i=1}^n T_{\zeta_i}.
\]

\noindent
\eqref{lem elementary consequences 7}: The proof is the same as for \eqref{lem elementary consequences 6} except that we use \eqref{lem elementary consequences 5} instead of \eqref{lem elementary consequences 4}.
\end{proof}

\begin{Lemma} \label{lem gens and rels in G}

For $\zeta \in \CZ$ let $t_\zeta$ denote the function $\chi_{[\zeta,A(\zeta)]} \in C_c(G_i)$, for $i = 1$, 2.  Then $\{ t_\zeta : \zeta \in \CZ \}$ satisfy (1) - (3) and (4)$_i$.

\end{Lemma}

\begin{proof}
First note that
$\varphi_{\overline{\zeta_2}} (A(\zeta_1))
= \varphi_{\overline{\zeta_2}} \varphi_{\overline{\zeta_1}} (\Lambda)
= \varphi_{\overline{\zeta_1 \zeta_2}}(\Lambda) = A(\zeta_1 \zeta_2)$.
Then
\[
t_{\zeta_1} t_{\zeta_2}
= \chi_{[\zeta_1,A(\zeta_1)]} \chi_{[\zeta_2,A(\zeta_2)]}
= \chi_{[\zeta_1,A(\zeta_1)] \, [\zeta_2,A(\zeta_2)]}
= \chi_{[\zeta_1 \zeta_2, \varphi_{\overline{\zeta_2}}(A(\zeta_1)) \cap A(\zeta_2))]}
= \chi_{[\zeta_1 \zeta_2, A(\zeta_1 \zeta_2)]}
= t_{\zeta_1 \zeta_2},
\]
verifying (1).  Next,
$t_{\overline{\zeta}}
= \chi_{[\overline{\zeta}, A(\overline{\zeta})]}
= \chi_{[\zeta, A(\zeta)]^{-1}}
= \chi_{[\zeta, A(\zeta)]}^*
= t_\zeta^*$,
verifying (2).  Next note that $A(\zeta) \mapsto \chi_{[\overline{\zeta} \zeta, A(\zeta)]} = t_{\overline{\zeta} \zeta}$ is a Boolean ring homomorphism.  By Theorem \ref{thm boolean ring homomorphism}, if $A(\zeta) = \bigcup_{i=1}^n A(\zeta_i)$ then $t_{\overline{\zeta}\zeta} = \bigvee_{i=1}^n t_{\overline{\zeta_i}\zeta_i}$, verifying (3).  Now let $i = 1$.  Suppose that $\Phi_\zeta = \text{id}_{\widehat{A(\zeta)}}$.  Then also $\Phi_{\overline{\zeta}\zeta} = \text{id}_{\widehat{A(\zeta)}}$.    Since $\Phi_\zeta = \text{id} = \Phi_{\overline{\zeta}\zeta}$ near each point of $\widehat{A(\zeta)}$, we have $[\zeta,A(\zeta)] = [\overline{\zeta} \zeta, A(\zeta)]$, i.e. $t_\zeta = t_{\overline{\zeta}\zeta}$ in $C_c(G_1)$, verifying (4)$_1$.  Finally, let $i = 2$.  Suppose that $\varphi_\zeta = \text{id}_{A(\zeta)}$.  We claim that $\Phi_\zeta = \text{id}_{\widehat{A(\zeta)}}$.  To see this, let $x \in \widehat{A(\zeta)}$.  We will let $\CU_x$ denote the ultrafilter (in $\CA_{s(\zeta)}$) corresponding to $x$.  Then $\CU_{\Phi_\zeta(x)}$ is the ultrafilter generated by $\varphi_\zeta(\CU_x)$ (recall that $\varphi_\zeta(\CU_x)$ is an ultrafilter base).  Let $E \in \CU_x$.  Since $x \in \widehat{A(\zeta)}$ we have $A(\zeta) \in \CU_x$.  Therefore $E \cap A(\zeta) \in \CU_x$, and hence $E \cap A(\zeta) = \varphi_\zeta(E \cap A(\zeta)) \in \CU_{\Phi_\zeta(x)}$.  Since $\CU_{\Phi_\zeta(x)}$ is a filter, $E \in \CU_{\Phi_\zeta(x)}$.  Thus we have shown that $\CU_x \subseteq \CU_{\Phi_\zeta(x)}$.  Since these are both ultrafilters it follows that $\CU_x = \CU_{\Phi_\zeta(x)}$, i.e. $\Phi_\zeta(x) = x$.  We now see that $\Phi_\zeta = \text{id}_{\widehat{A(\zeta)}}$, hence $t_\zeta = t_{\overline{\zeta} \zeta}$ in $C_c(G_2)$ verifying (4)$_2$.
\end{proof}

\begin{Theorem} \label{thm toeplitz gens and rels}

Let $\Lambda$ be an LCSC.  For $i = 1$, 2, $C^*(G_i)$ is the universal $C^*$-algebra generated by a family $\{ T_\zeta : \zeta \in \CZ \}$ satisfying (1) - (3) and (4)$_i$.

\end{Theorem}

\begin{proof}
It follows from Lemma \ref{lem gens and rels in G} that the generating family $\{ t_\zeta : \zeta \in \CZ \} \subseteq C_c(G_i)$ satisfy (1) - (3) and (4)$_i$.  For the converse, let $\{ T_\zeta : \zeta \in \CZ \}$ be elements in a $C^*$-algebra $B$ satisfying (1) - (3) and (4)$_i$.  We define $\mu : \CD^{(0)} \to B$ by $\mu(A(\zeta)) = T_\zeta^* T_\zeta$.  Since $ T_{\overline{\zeta_1} \zeta_1}  T_{\overline{\zeta_2} \zeta_2} =  T_{\overline{\zeta_1} \zeta_1 \overline{\zeta_2} \zeta_2}$, condition \eqref{thm boolean ring homomorphism 1} of Theorem \ref{thm boolean ring homomorphism} holds.  Condition \eqref{thm boolean ring homomorphism 2} of Theorem \ref{thm boolean ring homomorphism} holds by (3).  Then  by Theorem \ref{thm boolean ring homomorphism} we obtain a $*$-homomorphism $\pi_0 : C_0(G^0) \to B$ such that $\pi_0(\chi_{A(\zeta)}) = T_\zeta^* T_\zeta$.

Now let us consider the case $i = 2$.  In order to extend $\pi_0$ to all of $C^*(G)$, we proceed locally.  Let $f \in C_c(G_2)$ be such that there is $\zeta \in \CZ$ with $\text{supp}\,(f) \subseteq [\zeta,A(\zeta)]$.  Define $\widetilde{f} \in C\bigl(\widehat{A(\zeta)}\bigr)$ by $\widetilde{f}(x) = f([\zeta,x])$.  To show that $\widetilde{f}$ is well-defined, suppose also that $\text{supp}\,(f) \subseteq [\zeta',A(\zeta')]$.  For each $w \in \text{supp}\,(f)$, $w = [\zeta,x] = [\zeta',x]$ for some $x \in \widehat{A(\zeta)} \cap \widehat{A(\zeta')}$.  Then $f([\zeta,x]) = f([\zeta',x])$, and hence $\widetilde{f}$ doesn't depend on the choice of $\zeta$.

Now we define $\pi(f) = T_\zeta \pi_0(\widetilde{f})$ if $\text{supp}\,(f) \subseteq [\zeta,A(\zeta)]$.
To see that this is well-defined, let supp$(f) \subseteq [\zeta_1, A(\zeta_1)] \cap [\zeta_2, A(\zeta_2)]$.  First suppose that $f = \chi_{[\xi,A(\xi)]}$.  Then $A(\xi) \subseteq A(\zeta_1) \cap A(\zeta_2)$, and for $x \in \widehat{A(\xi)}$ we have $[\zeta_1,x] = [\zeta_2,x] = [\xi,x]$.  Then there is $E \in \CU_x$ such that $\varphi_{\zeta_1}|_E = \varphi_{\zeta_2}|_E = \varphi_\xi|_E$.  For each $\alpha \in A(\xi)$ we let $\{\alpha\}$ denote the filter in $X$ consisting of all zigzag sets containing $\alpha$.  Then $\CU_{\{\alpha\}}$ is the fixed ultrafilter at $\alpha$.  Then $\{ \alpha \} \in \widehat{A(\xi)}$ and then the above implies that there is $E \in \CU_{\{\alpha\}}$ with $\varphi_{\zeta_j}|_E = \varphi_\xi|_E$.  But $E \in \CU_{\{\alpha\}}$ implies that $\alpha \in E$, and hence $\varphi_{\zeta_j}(\alpha) = \varphi_\xi(\alpha)$.  Since $\alpha \in A(\xi)$ was arbitrary, we have $\varphi_{\zeta_j \overline{\xi}\xi} = \varphi_\xi$.  By Lemma \ref{lem elementary consequences}\eqref{lem elementary consequences 4} we have $T_{\zeta_1} T_\xi^* T_\xi = T_{\zeta_2} T_\xi^* T_\xi$.  We have $\widetilde{f} = \chi_{\widehat{A(\xi)}}$, and hence $\pi_0(\widetilde{f}) = T_\xi^* T_\xi$.  But then $T_{\zeta_1} \pi_0(\widetilde{f}) = T_{\zeta_1}T_\xi^* T_\xi = T_{\zeta_2} T_\xi^* T_\xi = T_{\zeta_2} \pi_0(\widetilde{f})$.  It follows that this also holds for $f$ in the span of such characteristic functions.  By continuity of $\pi_0$, it follows for all $f$ supported in basic sets $[\zeta,A(\zeta)]$.

For an arbitrary $f \in C_c(G)$, we may find $\zeta_i \in \CZ$ and $A_i \in \CA$, $1 \le i \le n$, such that $\text{supp}\,(f) \subseteq \bigsqcup_i [\zeta_i,A_i]$.  Then $f = \sum_i f |_{[\zeta_i,A_i]}$.  If also $\text{supp}\,(f) \subseteq \bigsqcup_j [\xi_j,B_j]$, then since $[\zeta_i,A_i] \cap [\xi_j,B_j] = [\zeta_i,A_i \cap B_j] = [\xi_j, A_i \cap B_j]$, we have
\[
\sum_i f |_{[\zeta_i,A_i]} = \sum_{i,j} f |_{[\zeta_i,A_i \cap B_j]} = \sum_{i,j} f |_{[\xi_j,A_i \cap B_j]} = \sum_j f |_{[\xi_j,B_j]}.
\]
Thus $\sum_i \pi \bigl( f |_{[\zeta_i,A_i]} \bigr) = \sum_j \pi \bigl( f |_{[\xi_j,B_j]} \bigr)$.  Therefore this last expression is a well-defined extension of $\pi$ to all of $C_c(G)$, and is a self-adjoint linear map.  We note that $\pi$ is continuous for the inductive limit topology, since by the above it reduces to uniform convergence on the sets $[\zeta,\widehat{A(\zeta)}]$.  Finally, since $\pi$ is multiplicative on the characteristic functions of the basic sets $[\zeta,A(\zeta)]$, the continuity implies that $\pi$ is multiplicative on $C_c(G)$.  Therefore $\pi$ extends to all of $C^*(G)$ by \cite[Definition 3.17]{exel08} (and the remarks preceding it).

Finally, we consider the case $i = 1$.  The only difference with the above argument is in showing that if $[\xi,A(\xi)] \subseteq [\zeta_1,A(\zeta_1)] \cap [\zeta_2,A(\zeta_2)]$ then $T_{\zeta_1 \overline{\xi} \xi} = T_{\zeta_2 \overline{\xi} \xi} = T_\xi$.  In this case, since $[\zeta_1,x] = [\zeta_2,x] = [\xi,x]$ for all $x \in \widehat{A(\xi)}$ we immediately get $\Phi_{\zeta_j}|_{\widehat{A(\xi)}} = \Phi_\xi$.  Then Lemma \ref{lem elementary consequences}\eqref{lem elementary consequences 5} implies that $T_{\zeta_1} T_\xi^* T_\xi = T_{\zeta_2} T_\xi^* T_\xi$, and the rest of the above argument remains valid.
\end{proof}

\begin{Remark}

It follows that Definition \ref{def toeplitz algebra} generalizes the definition given in \cite[Definition 4.6]{bkqs} to the nonfinitely aligned case.
\end{Remark}

\begin{Corollary} \label{cor 2 maps to 1}

There is a surjective homomorphism $C^*(G_2) \to C^*(G_1)$ carrying generators to generators.

\end{Corollary}

\begin{proof}
Let $\{ t_\zeta : \zeta \CZ \}$ denote the generators of $C^*(G_1)$.  We know that they satisfy (1) - (3) and (4)$_1$.  Let $\zeta \in \CZ$ and suppose that $\varphi_\zeta = \text{id}_{A(\zeta)}$.  Then also $\Phi_\zeta = \text{id}_{\widehat{A(\zeta)}}$.  By (4)$_1$ we have $t_\zeta = t_\zeta^* t_\zeta$.  Therefore $\{ t_\zeta : \zeta \in \CZ \}$ satisfy (4)$_2$.
\end{proof}

The finitely aligned case requires less adjustment from the treatment in \cite{spi_pathcat}.  Nevertheless we wish to give a careful proof for LCSC's, as the distinction between the two groupoids is crucial.  We mention in particular that while the hypotheses of \cite[Theorem 6.3]{spi_pathcat} mix the definitions of the two groupoids, by Proposition \ref{prop lcsc without inverses} the two coincide for the case of a finitely aligned category of paths.

\begin{Theorem} \label{thm finitely aligned 2}

Let $\Lambda$ be a finitely aligned LCSC.  Then $C^*(G_2)$ is the universal $C^*$-algebra generated by a family $\{ T_\alpha : \alpha \in \Lambda \}$ satisfying

\begin{itemize}

\item[$(1)'$] $T_\alpha^* T_\alpha = T_{s(\alpha)}$.

\item[$(2)'$] $T_\alpha T_\beta = T_{\alpha \beta}$, if $s(\alpha) = r(\beta)$.

\item[$(3)'$] $T_\alpha T_\alpha^* T_\beta T_\beta* = \bigvee_{\gamma \in \alpha \vee \beta} T_\gamma T_\gamma^*$.

\end{itemize}

\end{Theorem}

\begin{proof}
First suppose that we have a representation of $C^*(G_2)$.  By Theorem \ref{thm toeplitz gens and rels} we have a family of elements $\{ T_\zeta : \zeta \in \CZ \}$ in a $C^*$-algebra satisfying (1) - (3) and (4)$_2$.  For $\alpha \in \Lambda$ we define $T_\alpha = T_{(r(\alpha),\alpha)}$.  Note that $A(r(\alpha),\alpha) = s(\alpha) \Lambda = A(s(\alpha),s(\alpha))$.  Since $\varphi_{(\alpha,r(\alpha),r(\alpha),\alpha)} = \varphi_{(s(\alpha),s(\alpha))}$, by (1), (2) and Lemma \ref{lem elementary consequences}\eqref{lem elementary consequences 4} we have
\[
T_\alpha^* T_\alpha = T_{(r(\alpha),\alpha)}^* T_{(r(\alpha),\alpha)} = T_{\overline{(r(\alpha),\alpha)} (r(\alpha),\alpha)} = T_{(s(\alpha),s(\alpha))} = T_{s(\alpha)},
\]
verifying $(1)'$.  Next, let $s(\alpha) = r(\beta)$.  Since $\varphi_{(r(\alpha),\alpha,r(\beta),\beta)} = \varphi_{(r(\alpha),\alpha\beta)}$, by (1) and Lemma \ref{lem elementary consequences}\eqref{lem elementary consequences 4} we have
\[
T_\alpha T_\beta = T_{(r(\alpha),\alpha)} T_{(r(\beta),\beta)} = T_{(r(\alpha),\alpha,r(\beta),\beta)} = T_{(r(\alpha),\alpha\beta)} = T_{\alpha\beta},
\]
verifying $(2)'$.  Finally, let $\alpha$, $\beta \in \Lambda$.  Then by (1) and (2), $T_\alpha T_\alpha^* T_\beta T_\beta^* = T_{(r(\alpha),\alpha) \overline{(r(\alpha),\alpha)} (r(\beta),\beta) \overline{(r(\beta),\beta)}}$.  Note that
\begin{align*}
\varphi_{(r(\alpha),\alpha) \overline{(r(\alpha),\alpha)} (r(\beta),\beta) \overline{(r(\beta),\beta)}}
&= \varphi_{(r(\alpha),\alpha,\alpha,r(\alpha),r(\beta),\beta,\beta,r(\beta))}
= \tau^\alpha \sigma^\alpha \tau^\beta \sigma^\beta \\
&= \bigcup_{\gamma \in \alpha \vee \beta} \tau^\alpha \tau^{\displaystyle (\sigma^\alpha \gamma)}  \sigma^{\displaystyle (\sigma^\beta \gamma)} \sigma^\beta
= \bigcup_{\gamma \in \alpha \vee \beta} \tau^{\displaystyle (\alpha \sigma^\alpha \gamma)} \sigma^{\displaystyle (\beta \sigma^\beta \gamma)} \\
&= \bigcup_{\gamma \in \alpha \vee \beta} \tau^\gamma \sigma^\gamma
= \bigcup_{\gamma \in \alpha \vee \beta} \varphi_{(r(\gamma),\gamma,\gamma,r(\gamma))}.
\end{align*}
By Lemma \ref{lem elementary consequences}\eqref{lem elementary consequences 6} we have $T_{(r(\alpha),\alpha) \overline{(r(\alpha),\alpha)} (r(\beta),\beta) \overline{(r(\beta),\beta)}} = \bigvee_{\gamma \in \alpha \vee \beta} T_\gamma T_\gamma^*$, verifying $(3)'$.

Conversely, let $\{T_\alpha : \alpha \in \Lambda \}$ be given satisfying $(1)'$ - $(3)'$.  For $\zeta = (\alpha_1, \beta_1, \ldots, \alpha_n, \beta_n) \in \CZ$ define $T_\zeta = T_{\alpha_1}^* T_{\beta_1} \cdots T_{\alpha_n}^* T_{\beta_n}$.  Then (1) and (2) clearly hold.  We will verify (3).  

We first prove the following claim.  If $\gamma_i$, $\delta_i$, $\xi_j$, $\eta_j \in \Lambda$ are finite collections such that $\bigcup_i \tau^{\gamma_i} \sigma^{\delta_i} = \bigcup_j \tau^{\xi_j} \sigma^{\eta_j}$, then $\bigvee_i T_{\gamma_i} T_{\delta_i}^* = \bigvee_j T_{\xi_j} T_{\eta_j}^*$.  To prove this claim, first fix $i_0$.  Since $\delta_{i_0}$ is in the domain of $\bigcup_i \tau^{\gamma_i} \sigma^{\delta_i}$, there exists $j_0$ such that $\eta_{j_0} \in [\delta_{i_0}]$.  Similarly, there is $i_1$ such that $\delta_{i_1} \in [\eta_{j_0}]$.  Therefore $\delta_{i_1} \in [\delta_{i_0}]$.  Let $\delta_{i_0} = \delta_{i_1} \mu$.  Since any two terms of $\bigcup_i \tau^{\gamma_i} \sigma^{\delta_i}$ must agree on the intersection of their domains, we have 
\[
\gamma_{i_0} = \tau^{\gamma_{i_0}} \sigma^{\delta_{i_0}}(\delta_{i_0}) = \tau^{\gamma_{i_1}} \sigma^{\delta_{i_1}}(\delta_{i_0}) = \gamma_{i_1} \mu.
\]
Therefore $\tau^{\gamma_{i_0}} \sigma^{\delta_{i_0}} = \tau^{\gamma_{i_1}} \sigma^{\delta_{i_1}} |_{\delta_{i_0} \Lambda}$.  Thus the $i_0$ term may be deleted from $\bigcup_i \tau^{\gamma_i} \sigma^{\delta_i}$.  We repeat this process until we have that $\delta_i \not\in [\delta_{i'}]$ for all $i \not= i'$.  Moreover, we have $T_{\gamma_{i_1}} T_{\delta_{i_1}}^* = T_{\gamma_{i_0}} T_{\delta_{i_0}}^* + T_{\gamma_{i_1}} (T_{s(\gamma_{i_1})} - T_\mu T_\mu^*) T_{\delta_{i_1}}^*$.  Therefore $T_{\gamma_{i_0}} T_{\delta_{i_0}}^*$ can be deleted from $\bigvee_i T_{\gamma_i} T_{\delta_i}^*$.  Repeating this for the other map and operator, we may also assume that $\eta_j \not\in [\eta_{j'}]$ for all $j \not= j'$.  Now for each $i$ there is $j$ such that $\eta_j \in [\delta_i]$.  Then there is $i'$ such that $\delta_{i'} \in [\eta_j]$.  Hence $\delta_{i'} \in [\delta_i]$, so we must have $\delta_i = \delta_{i'} \approx \eta_j$.  Let $\delta_i = \eta_j \nu$, where $\nu \in s(\eta_j) \Lambda^{-1}$.  Applying both maps to $\delta_i = \eta_j \nu$ we find that $\gamma_i = \xi_j \nu$.  Thus the two presentations of the map are identical, and thus so are the operators.  This finishes the proof of the claim.  

Next we claim that if $\zeta \in \CZ$ and $\varphi_\zeta = \bigcup_i \tau^{\gamma_i} \sigma^{\delta_i}$ is a finite union, then $T_\zeta = \bigvee_i T_{\gamma_i} T_{\delta_i}^*$.  (By Lemma \ref{lem finitely aligned zigzag map} every zigzag map has this form.)  We prove this by induction on the length of $\zeta$.  First suppose that $\zeta = (\alpha,\beta)$.  Then $\varphi_\zeta = \sigma^\alpha \tau^\beta = \bigcup_{\gamma \in \alpha \vee \beta} \tau^{(\sigma^\alpha \gamma)} \sigma^{(\sigma^\beta \gamma)}$.  Moreover, by $(3)'$ we have
\[
T_\zeta = T_\alpha^* T_\beta = T_\alpha^* (T_\alpha T_\alpha^* T_\beta T_\beta^*) T_\beta = \bigvee_{\gamma \in \alpha \vee \beta} T_\alpha^* T_\gamma T_\gamma^* T_\beta = \bigvee_{\gamma \in \alpha \vee \beta} T_{\sigma^\alpha \gamma} T_{\sigma^\beta \gamma}^*.
\]
By the previous claim, we know that this doesn't depend on the decomposition chosen for $\varphi_\zeta$.  Now suppose that the current claim is true for zigzags of length at most $n$.  Let $\zeta = (\alpha_1, \beta_1, \ldots, \alpha_{n+1}, \beta_{n+1})$.  Let $\zeta_0 = (\alpha_1, \beta_1, \ldots, \alpha_n, \beta_n)$.  Write $\varphi_{\zeta_0} = \bigcup_i \tau^{\gamma_i} \sigma^{\delta_i}$ and $\varphi_{(\alpha_{n+1},\beta_{n+1})} = \bigcup_j \tau^{\mu_j} \sigma^{\nu_j}$.  Then
\[
\varphi_\zeta = \varphi_{\zeta_0} \circ \varphi_{(\alpha_{n+1},\beta_{n+1})} = \bigcup_{i,j} \tau^{\gamma_i} \sigma^{\delta_i} \tau^{\mu_j} \sigma^{\nu_j} = \bigcup_{i,j,k} \tau^{\gamma_i} \tau^{\xi_k} \sigma^{\eta_k} \sigma^{\nu_j} = \bigcup_{i,j,k} \tau^{\gamma_i\xi_k} \sigma^{\nu_j\eta_k}.
\]
Then the inductive hypothesis gives
\[
T_\zeta = T_{\zeta_0} T_{(\alpha_{n+1},\beta_{n+1})} = \bigvee_{i,j} T_{\gamma_i} T_{\delta_i}^* T_{\mu_j} T_{\nu_j}^* = \bigvee_{i,j,k} T_{\gamma_i} T_{\xi_k} T_{\eta_k}^* T_{\nu_j}^* = \bigvee_{i,j,k} T_{\gamma_i \xi_k} T_{\nu_j \eta_k}^*.
\]
Again, the first claim shows that this is independent of the choice of decomposition of $\varphi_\zeta$.

Now let $A(\zeta) = \bigcup_{i=1}^n A(\zeta_i)$.  Write $\varphi_\zeta = \bigcup_j \tau^{\gamma_j} \sigma^{\delta_j}$ and $\varphi_{\zeta_i} = \bigcup_j \tau^{\gamma_{ij}} \sigma^{\delta_{ij}}$.  Then $\varphi_{\overline{\zeta}\zeta} = \bigcup_j \tau^{\delta_j} \sigma^{\delta_j}$, and similarly for $\varphi_{\overline{\zeta_i} \zeta_i}$.  Therefore we also have $\varphi_{\overline{\zeta} \zeta} = \bigcup_{i,j} \tau^{\delta_{ij}} \sigma^{\delta_{ij}}$.  By the last claim we get
\[
\bigvee_j T_{\delta_j} T_{\delta_j}^* = T_{\overline{\zeta} \zeta} = \bigvee_{i,j} T_{\delta_{ij}} T_{\delta_{ij}}^*.
\]
and hence $T_\zeta^* T_\zeta = \bigvee_i T_{\zeta_i}^* T_{\zeta_i}$, verifying (3).

Finally, if $\varphi_\zeta = \text{id}_{A(\zeta)} = \varphi_{\overline{\zeta} \zeta}$, then the last claim gives $T_\zeta = \bigvee_j T_{\delta_j} T_{\delta_j}^* = T_\zeta^* T_\zeta$, verifying (4)$_2$.
\end{proof}

Since the choice of the groupoid is not present in the relations $(1)'$ - $(3)'$, it is clear that the analogous result for $C^*(G_1)$ will require additional relations.  (This is clear from consideration of Examples \ref{example group old} and \ref{example group new}.)

\begin{Theorem} \label{thm finitely aligned 1}

Let $\Lambda$ be a finitely aligned LCSC.  Then $C^*(G_1)$ is the universal $C^*$-algebra generated by a family $\{ T_\alpha : \alpha \in \Lambda \}$ satisfying $(1)'$ - $(3)'$ of Theorem \ref{thm finitely aligned 2} and also

\begin{itemize}

\item[$(4)'$] If $\Phi_{(r(\alpha),\alpha)} = \Phi_{(r(\beta),\beta)}$ then $T_\alpha = T_\beta$.

\end{itemize}

\end{Theorem}

\begin{proof}
First suppose that we have a representation of $C^*(G_1)$.  By Theorem \ref{thm toeplitz gens and rels} we have a family of elements $\{ T_\zeta : \zeta \in \CZ \}$ in a $C^*$-algebra satisfying (1) - (3) and (4)$_1$.  For $\alpha \in \Lambda$ we define $T_\alpha = T_{(r(\alpha),\alpha)}$.  Since in general, if $\varphi_\zeta = \varphi_{\zeta'}$ then $\Phi_\zeta = \Phi_{\zeta'}$, relations $(1)'$ - $(3)'$ follow as in the proof of Theorem \ref{thm finitely aligned 2}, but using Lemma \ref{lem elementary consequences} \eqref{lem elementary consequences 5} and \eqref{lem elementary consequences 7} instead of \eqref{lem elementary consequences 4} and \eqref{lem elementary consequences 6}.  Relation $(4)'$ follows from Lemma \ref{lem elementary consequences} \eqref{lem elementary consequences 5}.

Conversely, let $\{T_\alpha : \alpha \in \Lambda \}$ be given satisfying $(1)'$ - $(4)'$.  For $\zeta = (\alpha_1, \beta_1, \ldots, \alpha_n, \beta_n) \in \CZ$ define $T_\zeta = T_{\alpha_1}^* T_{\beta_1} \cdots T_{\alpha_n}^* T_{\beta_n}$.  Then (1) and (2) clearly hold.  Relation (3) follows as in the proof of Theorem \ref{thm finitely aligned 2}.  Suppose that $\Phi_\zeta = \text{id}_{A(\zeta)}$.  By Lemma \ref{lem finitely aligned zigzag map} we may write $\varphi_\zeta = \bigcup_i \tau^{\gamma_i} \sigma^{\delta_i}$ (a finite sum), so that $\Phi_\zeta = \bigcup_i \Phi_{(r(\gamma_i),\gamma_i)} \circ \Phi_{(r(\delta_i),\delta_i)}^{-1}$.  Then for each $i$ we have $\Phi_{(r(\gamma_i),\gamma_i)} \circ \Phi_{(r(\delta_i),\delta_i)}^{-1} = \text{id}_{\widehat{\delta_i \Lambda}}$, hence $\Phi_{(r(\gamma_i),\gamma_i)} = \Phi_{(r(\delta_i),\delta_i)}$.  By $(4)'$ we have $T_{\gamma_i} = T_{\delta_i}$.  Note that the proof of the second claim in the proof of Theorem \ref{thm finitely aligned 2} did not rely on (4)$_2$, and is still valid here.  Therefore $T_\zeta = \bigvee_i T_{\gamma_i} T_{\delta_i}^* = \bigvee T_{\delta_i} T_{\delta_i}^* = T_\zeta T_\zeta^*$, verifying (4)$_1$.
\end{proof}

\section{The boundary of an LCSC}
\label{sec boundary}

Next we define the boundary of an LCSC $\Lambda$.  By Proposition \ref{p.ultrafilter} we may identify $v \Lambda^* = X_v$, where $X = \sqcup_{v \in \Lambda^0} X_v$ is the unit space of the groupoid of $\Lambda$.  We intend to define the boundary to be the closure of the maximal elements of $v \Lambda^*$.  There is a potential ambiguity that we dispose of first; namely, while it is clear that Zorn's lemma applies to the set of filters in $\CD^{(0)}_v $, ordered by inclusion, it is less clear that it applies to the set $v \Lambda^*$.  Also, it is unclear that maximality is the same for the two settings.

\begin{Lemma} \label{lem unambiguous maximal filters}

Let $\Lambda$ be an LCSC and let $v \in \Lambda^0$.  Every maximal filter in $\CD^{(0)}_v$ is an element of $v \Lambda^*$.  Moreover, Zorn's lemma applies to $v \Lambda^*$, and maximal elements of $v \Lambda^*$ are maximal as filters in $\CD^{(0)}_v$.

\end{Lemma}

\begin{proof}
Let $C$ be a maximal filter in $\CD^{(0)}_v$.  We show that $C \in v \Lambda^*$.  Let $F_1$, $\ldots$, $F_n \in \CD^{(0)}_v$ cover $C$.  If for each $i$ there is $E_i \in C$ such that $E_i \cap F_i = \varnothing$, then $E = E_1 \cap \cdots \cap E_n \in C$ and $E \cap F_i = \varnothing$ for $i = 1$, $\ldots$, $n$.  But then, since $E$ meets every element of $C$, it would follow that $F_1$, $\ldots$, $F_n$ do not cover $C$.  Therefore there must be $i$ for which $F_i \cap E \not= \varnothing$ for all $E \in C$. Now we have that $\{ F_i \cap E : E \in C \}$ is a filter base in $\CD^{(0)}_v$.  The filter it generates contains $C$, as well as $F_i$.  By maximality of $C$ it follows that $F_i \in C$.  Thus we have that $C \in v \Lambda^*$.

Next we show that Zorn's lemma applies to $v \Lambda^*$.  Let $S \subseteq v \Lambda^*$ be a totally ordered subset, and put $C = \bigcup S$.  It is clear that $C$ is a filter.  Suppose that $F_1$, $\ldots$, $F_n \in \CD^{(0)}_v$ cover $C$.  Then there is $E \in C$ such that $E \subseteq F_1 \cup \cdots \cup F_n$.  There is $C_0 \in S$ such that $E \in C_0$.  Since $C_0 \in v \Lambda^*$, there must be $i$ such that $F_i \in C_0$.  But then $F_i \in C$.  Therefore $C \in v \Lambda^*$ is an upper bound for $S$, verifying that Zorn's lemma applies to $v \Lambda^*$.

Finally let $C$ be a maximal element of $v \Lambda^*$. We show that $C$ is a maximal filter in $\CD^{(0)}_v$.  Let $C_1$ be a filter in $\CD^{(0)}_v$ with $C \subseteq C_1$.  There is a maximal filter $C_2$ in $\CD^{(0)}_v$ with $C_1 \subseteq C_2$.  By the above, $C_2 \in v \Lambda^*$.  By maximality of $C$ in $v \Lambda^*$ we must have $C = C_2$, and hence $C = C_1$.
\end{proof}

\begin{Definition}
\label{d.boundary}
For $v \in \Lambda^0$ we let $v \Lambda^{**}$ denote the set of maximal elements of $v \Lambda^*$.  We define the \textit{boundary} of $\Lambda$ to be the closure of the maximal elements: 
\[
v \partial \Lambda = \overline{v\Lambda^{**}}, \quad \partial \Lambda = \bigsqcup_{v \in \Lambda^0} v \partial \Lambda.
\]

\end{Definition}

\begin{Definition} \label{def ck algebra}

Let $\Lambda$ be an LCSC.  The {\it Cuntz-Krieger algebra} of $\Lambda$ is the $C^*$-algebra $\CO(\Lambda) := C^*(G_2(\Lambda)|_{\partial \Lambda})$.

\end{Definition}

In order to derive the presentation of $\CO(\Lambda)$ by generators and relations we must first characterize points of the boundary.

\begin{Proposition}
\label{p.maxelements}
Let $C \subseteq \CD_v^{(0)}$ be a filter. $C \in v \Lambda^{**}$ if and only if for each $F \in \CD_v^{(0)}$, if $F \cap E \not= \varnothing$ for all $E \in C$, then $F \supseteq E$ for some $E \in C$.

\end{Proposition}

\begin{proof}
First suppose that $C \in v \Lambda^{**}$.  Let $F \in \CD_v^{(0)}$, and suppose that $F \cap E \not= \varnothing$ for all $E \in C$.  Then $\{ F \cap E : E \in C \}$ is closed under intersection and does not contain the empty set.  This collection then generates a filter containing $C$.  Since $C$ is maximal, it follows from Lemma \ref{lem unambiguous maximal filters} that $F \in C$.  Conversely, suppose the condition in the statement holds.  Let $F \in \CD_v^{(0)} \setminus C$.  Since $C$ is closed under the formation of supersets, we have that $F \not\supseteq E$ for all $E \in C$.  Then the condition of the statement implies that there exists $E \in C$ with $F \cap E = \varnothing$.  This implies that $C$ is maximal.
\end{proof}

Now we will characterize elements of the boundary.

\begin{Theorem}
\label{t.boundary}
Let $C \in v \Lambda^*$.  The following are equivalent:

\begin{enumerate}

\item
\label{t.boundarya}
$C \in v\partial \Lambda$.

\item
\label{t.boundaryb}
For all $\CF \subseteq \CD_v^{(0)}$ finite, if $\CF$ does not cover $C$, then for each $E \in C$ there is $G \in \CD_v^{(0)}$ such that $G \subseteq E \setminus (\cup \CF)$.

\end{enumerate}

\end{Theorem}

\begin{proof}
\noindent
$\eqref{t.boundarya} \Rightarrow \eqref{t.boundaryb}$: Let $\CF \subseteq \CD_v^{(0)}$ be finite.  Suppose that $\CF$ does not cover $C$.  Let $E \in C$.  Then $E \setminus (\cup \CF) \in \CU_C$.  This means that $(E \setminus (\cup \CF))\,\widehat{\null}$\, is a neighborhood of $C$.  Since $C \in v\partial \Lambda$, there exists $C' \in (E \setminus (\cup \CF))\,\widehat{\null}\, \cap v \Lambda^{**}$.  Then $E \setminus (\cup \CF) \in \CU_{C'}$.  But then $E \in \CU_{C'}$, so $E \in C'$, by Proposition \ref{p.ultrafilter}.  Since $\CU_{C'}$ is a filter, no element of $\CF$ contains an element of $C'$.  By Proposition \ref{p.maxelements}, there is $E' \in C'$ such that $E' \cap (\cup \CF) = \varnothing$.  Let $G = E \cap E'$.  Then $G \not= \varnothing$, since $E$, $E' \in C'$.  Since $G \subseteq E$ and $G \cap (\cup \CF) = \varnothing$, we have $G \subseteq E \setminus (\cup \CF)$.

\noindent
$\eqref{t.boundaryb} \Rightarrow \eqref{t.boundarya}$: Let $N$ be a neighborhood of $C$.  We may assume that $N = \widehat{A}$ for some $A \in \CD_v$.  Then $A = E \setminus (\cup \CF)$, where $\{E\} \cup \CF \subseteq \CD_v^{(0)}$ and $\cup \CF \supsetneq E$.  Then $A \in \CU_C$, so $E \in C$ and $\CF$ does not cover $C$.  By \eqref{t.boundaryb} there is $G \in \CD_v^{(0)}$ such that $G \subseteq E \setminus (\cup \CF)$. Let $C' \in v \Lambda^{**}$ with $G \in C'$.  Then $E \setminus (\cup \CF) \in \CU_{C'}$, so $C' \in \widehat{A} = N$.  Therefore $\CU_C \in \overline{v \Lambda^{**}} = v \partial \Lambda$.
\end{proof}

We now prepare to give the analog of the Cuntz-Krieger relation(s) that characterize boundary representations of $\CT(\Lambda)$.  The next definition follows \cite[p. 124]{dm}.

\begin{Definition}
\label{def cover of a set}
A finite subset $\CF \subseteq \CD_v^{(0)}$ {\it covers} an element $E \in \CD_v^{(0)}$ if $\cup \CF \subseteq E$, and if for every $G \in \CD_v^{(0)}$ with $G \subseteq E$ there exists $F \in \CF$ such that $F \cap G \not= \varnothing$.  (Equivalently, $\CF$ covers $E$ if $E \setminus (\cup \CF)$ does not contain an element of $\CD_v^{(0)}$.)  We mention that this definition is modeled on \cite{dm}, but stands on its own here.
\end{Definition}

\begin{Remark}

Note that this use of the word {\it cover} must be distinguisheed from the use in Definition \ref{d.cover}.  In fact this should be clear from the usage, since the object covered is an element of $\CD_v^{(0)}$ in Definition \ref{def cover of a set}, rather than a filter in $\CD_v^{(0)}$ as in Definition \ref{d.cover}.

\end{Remark}

We give a Cuntz-Krieger-type relation based on Definition \ref{def cover of a set}.  This extends the list given in Definition \ref{def toeplitz relations}.

\begin{Definition} \label{def cuntz-krieger relation}

Let $\Lambda$ be an LCSC.  Let $\{ T_\zeta : \zeta \in \CZ_\Lambda \}$ be a family of elements of a $C^*$-algebra.  We will consider the following relation on the $T_\zeta$:

\begin{itemize}

\item[(5)] For all $v \in \Lambda^0$, and all $\zeta \in \CZ_\Lambda v$ and finite set $J \subseteq \CZ_\Lambda v$, if $\{ A(\xi) : \xi \in J \}$ covers $A(\zeta)$ then $T_\zeta = \bigvee_{\xi \in \CJ} T_\xi^* T_\xi$.

\end{itemize}

\end{Definition}

\begin{Remark}

The condition (5) is related to Exel's notion of {\it tightness} (\cite{exel08}), which in turn is described by Donsig and Milan as {\it cover-to-join} (\cite{dm}).  Following \cite[p. 124]{dm} we say that in an inverse semigroup $S$, a finite set $C$ {\it covers} an element $s$ if $y \le s$ for all $y \in C$, and if for every nonzero $x \le s$ there exists $y \in C$ and $0 \not= z \in S$ such that $z \le x$, $y$.  A representation $\pi$ of $S$ to an inverse semigroup $U$ is called {\it cover-to-join} if for every cover $C$ of an element $s$ we have $\pi(s) = \bigvee \pi(C)$.  It follows from \cite[Corollary 2.3]{dm} that for a representation $\pi$ of $\CT(\Lambda)$, the restriction of $\pi$ to $\{ T_\zeta^* T_\zeta : \zeta \in \CZ_\Lambda \}$ (which is isomorphic to the semilattice $\CD^{(0)}(\Lambda)$) is cover-to-join if and only if it is tight in the sense of Exel, and hence if and only if it satisfies relation (5).
\end{Remark}

The following theorem generalizes \cite[Theorem 8.2]{spi_pathcat}, both in that it applies to LCSC's instead of categories of paths, and also in that it applies in the general (nonfinitely aligned) case.  We mention that the amenability and countability assumptions in that paper are not necessary.

\begin{Theorem}\label{thm boundary gens and rels}

Let $\Lambda$ be an LCSC.  For $i = 1$, 2, $C^*(G_i|_{\partial \Lambda})$ is the universal $C^*$-algebra generated by a family $\{ S_\zeta : \zeta \in \CZ \}$ satisfying (1) - (3), (4)$_i$, and (5).

\end{Theorem}

\begin{proof}
First suppose that $\{S_\zeta\}$ satisfy (1) - (3), (4)$_i$, and (5).  By (1) - (3) and (4)$_i$, and Theorem \ref{thm toeplitz gens and rels}, there is a unique representation $\pi$ of $C^*(G_i)$ such that $\pi(\chi_{[\zeta,A(\zeta)]}) = S_\zeta$.  We claim that $\pi\bigr|_{C_0(\partial \Lambda^c)} = 0$.  To see this, let $C \in v \Lambda^* \setminus v\partial \Lambda$.  By Theorem \ref{t.boundary}, there exists a finite collection $\CF \subseteq \CD_v^{(0)}$ such that $\CF$ does not cover $C$, and there exists $E \in C$ such that for all $G \in 
\CD_v^{(0)}$, $G \not\subseteq E \setminus (\cup \CF)$.  Thus for all $G \in \CD_v^{(0)}$, if $G 
\subseteq E$ then $G \cap F \not= \varnothing$ for some $F \in \CF$.  Let $\CF' = \{ E \cap F : F \in \CF \}$.  It follows that $\CF'$ covers $E$. Let $D = E \setminus (\cup \CF) = E \setminus (\cup \CF')$.  We claim that $\widehat{D} \cap v \partial \Lambda = \varnothing$.  For if not, then $\widehat{D} \cap v \partial \Lambda$ contains an element $C'$ from $v \partial\Lambda$. Then $D \in \CU_{C'}$, hence $\CF' \cap C' = \varnothing$.  Therefore $\CF'$ does not cover $C'$.  Then by Theorem \ref{t.boundary}, $D$ contains an element of $\CD_v^{(0)}$, contradicting the fact that $\CF'$ covers $E$. Now let $\zeta \in \CZ_\Lambda v$, and $\CJ \subseteq \CZ_\Lambda v$ finite, be such that $E = A(\zeta)$ and $\CF = \{ A(\xi) : \xi \in \CJ \}$.  Then $\CF' = \{ A(\xi \overline{\zeta} \zeta) : \xi \in \CJ \}$.  By (5) we have
\[
\pi(\chi_D)
= \pi(\chi_E - \bigvee_{F \in \CF'} \chi_F)
= S_\zeta^* S_\zeta - \bigvee_{\xi \in \CJ} S_{\xi \overline{\zeta} \zeta}^* S_{\xi \overline{\zeta} \zeta}
= 0.
\]
Thus $\pi\bigr|_{C_0(\partial \Lambda^c)} = 0$.  It follows that $\pi(C^*(G\bigr|_{\partial \Lambda^c})) = 0$.  There is an exact sequence
\[
0 \to C^*(G\bigr|_{\partial \Lambda^c}) \to C^*(G) \to C^*(G\bigr|_{\partial \Lambda}) \to 0
\]
(\cite[Remark 4.10]{ren91}).  It follows that $\pi$ factors through $C^*(G\bigr|_{\partial \Lambda})$.

Conversely, let $\pi$ be a representation of $C^*(G\bigr|_{\partial \Lambda})$.  For $\zeta \in \CZ_\Lambda $, let $S_\zeta = \pi(\chi_{[\zeta,A(\zeta)]})$.  Composing $\pi$ with the quotient map gives a representation of $C^*(G)$, so by Theorem \ref{thm toeplitz gens and rels} we have that (1) - (3) and (4)$_i$ hold.  We will prove (5). Let $\zeta$ and $\CJ$ be as in (5).  Let $D = A(\zeta) \setminus \bigcup_{\xi \in \CJ} A(\xi)$.  The $G \not\subseteq D$ for all $G \subseteq \CD_v^{(0)}$.  We claim that $\widehat{D} \cap v \partial \Lambda = \varnothing$.  For suppose that $C \in \widehat{D} \cap v \partial \Lambda$. Then $D \in \CU_C$, so $\{ A(\xi) : \xi \in \CJ \}$ does not cover $C$. Since $E \in C$, Theorem \ref{t.boundary} implies that there is $G \in \CD_v^{(0)}$ such that $G \subseteq D$, a contradiction.  Therefore $0 = \pi(\chi_D) = S_\zeta^* S_\zeta - \bigvee_{\xi \in \CJ} S_\xi^* S_\xi$.
\end{proof}

\begin{Remark}

It follows that Definition \ref{def ck algebra} generalizes the definition given in \cite[Definition 4.7]{bkqs}, and extends it to the general nonfinitely aligned case.

\end{Remark}

\begin{Corollary} \label{cor boundary 2 maps to 1}

Let $\Lambda$ be an LCSC.  There is a surjective homomorphism $C^*(G_2|_{\partial \Lambda}) \to C^*(G_1|_{\partial \Lambda})$ carrying generators to generators.

\end{Corollary}

\begin{proof}
The argument is the same as for Corollary \ref{cor 2 maps to 1}.
\end{proof}

Now we adapt Theorem \ref{thm boundary gens and rels} to the finitely aligned case.  For the next results we note that the statements and proofs in \cite[Section 7]{spi_pathcat} do not rely on the hypotheses of right cancellation and lack of inverses in the definition of categories of paths.  Thus those results are true for any countable finitely aligned LCSC.  The following is a modification of \cite[Definition 7.7]{spi_pathcat} (which in turn was borrowed from \cite{raesimyee}).  (In the case where $\Lambda$ is a singly aligned monoid, i.e. LCM, this is termed {\it foundation set} in \cite{simyee}.) As in \cite[Section 8]{spi_pathcat} we need to require that $\Lambda$ be countable in order to give the alternate characterization of the boundary quotient in the finitely aligned case.

\begin{Definition} \label{def finitely aligned exhaustive}

Let $\Lambda$ be a finitely aligned LCSC, and let $v \in \Lambda$.  A subset $F \subseteq v \Lambda$ is {\it exhaustive} if for every $\beta \in v\Lambda$ there is $\alpha \in F$ such that $\alpha \Lambda \cap \beta \Lambda \not= \varnothing$.

\end{Definition}

Now we give the analog of (5) for the finitely aligned case.

\begin{Definition} \label{def finitely aligned ck relation}

Let $\Lambda$ be a finitely aligned LCSC, and let $\{ T_\alpha : \alpha \in \Lambda \}$ be a family of elements of a $C^*$-algebra.  We will consider the following relation on the $T_\alpha$:

\begin{itemize}

\item[$(5)'$] $T_v = \bigvee_{\alpha \in F} T_\alpha T_\alpha^*$ if $F \subseteq v\Lambda$ is a finite exhaustive set.

\end{itemize}

\end{Definition}

\begin{Theorem} \label{thm finitely aligned boundary}

Let $\Lambda$ be a countable finitely aligned LCSC.

\begin{enumerate}

\item \label{thm finitely aligned boundary 2} $C^*(G_2|_{\partial\Lambda})$ is the universal $C^*$-algebra generated by a family $\{S_\alpha : \alpha \in \Lambda \}$ satisfying $(1)'$ - $(3)'$ and $(5)'$.

\item \label{thm finitely aligned boundary 1} $C^*(G_1|_{\partial\Lambda})$ is the universal $C^*$-algebra generated by a family $\{S_\alpha : \alpha \in \Lambda \}$ satisfying $(1)'$ - $(5)'$.

\end{enumerate}

\end{Theorem}

\begin{proof}
First let $\pi$ be a representation of $C^*(G|_{\partial \Lambda})$.  By Theorem \ref{t.boundary} we know that (1) - (3), (4)$_i$, and (5) hold. We know from Theorems \ref{thm finitely aligned 2} and \ref{thm finitely aligned 1} that for $G_2$, $(1)'$ - $(3)'$ are equivalent to (1) - (3) and (4)$_2$, while for $G_1$, $(1)'$ - $(4)'$ are equivalent to (1) - (3) and (4)$_1$.  Thus it suffices to prove that $(5)'$ holds.  Let $F \subseteq v\Lambda$ be finite exhaustive.  Define $\CJ = \{ (\alpha,r(\alpha)) : \alpha \in F \}\subseteq \CZ v$.  We claim that $\{ A(\xi) : \xi \in \CJ \}$ covers $v \Lambda$.  Since $A(\alpha,r(\alpha)) = \alpha \Lambda$, this follows from the Definition \ref{def finitely aligned exhaustive}.  Then by (5) we have $T_v = \bigvee_{\zeta \in \CJ} T_\zeta^* T_\zeta = \bigvee_{\alpha \in F} T_\alpha T_\alpha^*$, verifying $(5)'$.

Now let $\pi$ be a representation of $\CT(\Lambda)$, and assume $(5)'$.  Let $C \in v \Lambda^* \setminus v \partial \Lambda$. By \cite[Theorem 7.8]{spi_pathcat} there is $\alpha \in C$ such that for all $\alpha' \in C \cap \alpha \Lambda$ there exists a finite exhaustive set $F \subseteq s(\alpha') \Lambda$ such that $\sigma^{\alpha'}(C) \cap F = \varnothing$.  We apply this with $\alpha' = \alpha$ to obtain the corresponding finite exhaustive set $F$ at $s(\alpha)$.  Let $E = \alpha \Lambda \setminus \bigcup_{\beta \in F} \alpha \beta \Lambda$.  Then $\sigma^\alpha(C) \cap F = \varnothing$, or equivalently, $C \cap \alpha F = \varnothing$. By \cite[Lemma 7.9]{spi_pathcat} we have $C \in \widehat{E} \subseteq v \Lambda^* \setminus v \Lambda^{**}$.  Then
\[
\pi(\chi_E)
= S_\alpha S_\alpha^* - \bigvee_{\beta \in F} S_{\alpha\beta} S_{\alpha\beta}^*
= S_\alpha (S_{s(\alpha)} - \bigvee_{\beta \in F} S_{\beta} S_{\beta}^*) S_\alpha^*
= 0,
\]
by $(5)'$.  Therefore $\pi(C_0(\partial \Lambda^c)) = 0$.
\end{proof}

\section{The regular representation}
\label{sec regular representation}

Let $\Lambda$ be an LCSC.  We recall that $\CT(\Lambda) = C^*(G_2(\Lambda))$.  We will write $G_2$ for $G_2(\Lambda)$, $\CZ$ for $\CZ_\Lambda $, etc.  In this section we compare the regular representation of $\Lambda$ with the regular representation of $G_2$.  (Since the regular representation on $\ell^2(\Lambda)$ is essentially defined by the maps $\varphi_\zeta$, $G_1(\Lambda)$ is generally not relevant to this matter.  Examples \ref{example group old} and \ref{example group new} show that one cannot in general expect to have a map from $C^*(G_1)$ to the algebra generated by the regular representation.)  We let $\{ e_\alpha : \alpha \in \Lambda \}$ denote the standard orthonormal basis of $\ell^2(\Lambda)$.  The following generalizes \cite[Proposition 7.2]{bkqs}.

\begin{Lemma} \label{lem regular representation}

There is a representation $\pi_\ell : \CT(\Lambda) \to B(\ell^2(\Lambda))$ defined by $\pi_\ell(\chi_{[\zeta,A(\zeta)]}) = T_\zeta$, where for $\zeta \in \CZ$,
\[
T_\zeta e_\alpha = \begin{cases} e_{\varphi_\zeta(\alpha)}, &\text{ if }\alpha \in A(\zeta) \\ 0, &\text{ if } otherwise. \end{cases}
\]

\end{Lemma}

\begin{proof}
By Theorem \ref{thm toeplitz gens and rels} there are four properties to verify.  First we show that $T_{\zeta_1} T_{\zeta_2} = T_{\zeta_1 \zeta_2}$.  We have for $\alpha \in \Lambda$,
\begin{align*}
T_{\zeta_1} T_{\zeta_2} e_\alpha
&= \begin{cases} T_{\zeta_1} e_{\varphi_{\zeta_2}(\alpha)}, &\text{ if } \alpha \in A(\zeta_2) \\ 0, &\text{ otherwise.} \end{cases} \\
&= \begin{cases} e_{\varphi_{\zeta_1} \circ \varphi_{\zeta_2}(\alpha)}, &\text{ if } \alpha \in A(\zeta_2) \text{ and } \varphi_{\zeta_2}(\alpha) \in A(\zeta_1) \\ 0, &\text{ otherwise} \end{cases} \\
&= \begin{cases} e_{\varphi_{\zeta_1 \zeta_2}(\alpha)}, &\text{ if } \alpha \in A(\zeta_2) \cap \varphi_{\zeta_2}^{-1}(A(\zeta_1)) = A(\zeta_1 \zeta_2) \\ 0, &\text{ otherwise} \end{cases} \\
&= T_{\zeta_1\zeta_2} e_\alpha.
\end{align*}
Next we show that $T_\zeta^* = T_{\overline{\zeta}}$.  For $\alpha$, $\beta \in \Lambda$ we have
\begin{align*}
\langle T_\zeta^* e_\alpha,e_\beta \rangle
&= \langle e_\alpha, T_\zeta e_\beta \rangle \\
&= \begin{cases} \langle e_\alpha, e_{\varphi_\zeta(\beta)} \rangle, &\text{ if } \beta \in A(\zeta) \\ 0, &\text{ otherwise} \end{cases} \\
&= \begin{cases} 1, &\text{ if } \beta \in A(\zeta) \text{ and } \alpha = \varphi_\zeta(\beta) \\ 0, &\text{ otherwise} \end{cases} \\
&= \begin{cases} 1, &\text{ if } \alpha \in A(\overline{\zeta}) \text{ and } \beta = \varphi_{\overline{\zeta}}(\alpha) \\ 0, &\text{ otherwise} \end{cases} \\
&= \cdots \\
&= \langle T_{\overline{\zeta}} e_\alpha,e_\beta \rangle.
\end{align*}
Next let $A(\zeta) = \bigcup_{i=1}^n A(\zeta_i)$.  We show $T_\zeta^* T_\zeta = \bigvee_{i=1}^n T_{\zeta_i}^* T_{\zeta_i}$.  First note that by the first two properties we have $T_\zeta^* T_\zeta = T_{\overline{\zeta}\zeta}$, and $\varphi_{\overline{\zeta}\zeta}(\alpha) = \alpha$ if $\alpha \in A(\zeta)$, and is undefined otherwise.  Therefore $T_\zeta^*T_\zeta e_\alpha = e_\alpha$ if $\alpha \in A(\zeta)$ and equals 0 otherwise.  A similar formula holds for each $i$, which establishes the desired equation.  Finally, suppose that $\varphi_\zeta = \text{id}_{A(\zeta)}$.  Then
\begin{align*}
T_\zeta e_\alpha
&= \begin{cases} e_{\varphi_\zeta(\alpha)}, &\text{ if } \alpha \in A(\zeta) \\ 0, &\text{ otherwise} \end{cases} \\
&= \begin{cases} e_\alpha, &\text{ if } \alpha \in A(\zeta) \\ 0, &\text{ otherwise} \end{cases} \\
&= T_\zeta^* T_\zeta e_\alpha. \qedhere
\end{align*}
\end{proof}

\begin{Definition} \label{def regular representations of Lambda}

We write $\CT_r(\Lambda)$ for the $C^*$-algebra $C^*_r(G_2)$ (the reduced groupoid $C^*$-algebra), and refer to it as the {\it reduced Toeplitz algebra} of $\Lambda$.  We write $\CT_\ell(\Lambda)$ for the $C^*$-algebra $\pi_\ell(\CT(\Lambda))$, and refer to it as the {\it regular Toeplitz algebra} of $\Lambda$.

\end{Definition}

\begin{Remark}

We have adapted the definition given in \cite[Definition 7.1]{bkqs}, and extended it to the general nonfinitely aligned case.

\end{Remark}

We next recall the basic facts about regular representations of \'etale groupoids.  Let $G$ be an \'etale groupoid.  For $x \in G^0$ we have the induced representation Ind$\,x$ of $C_c(G)$ on $\ell^2(Gx)$ defined by Ind$\,x(f) \xi = f * \xi$, where $f * \xi (\alpha) = \sum_{\beta} f(\beta) \xi(\beta^{-1} \alpha)$.  Letting $\delta_{gx}$ denote the standard basis vector of $\ell^2(Gx)$ at $gx \in Gx$, we have $f * \delta_{gx} = \sum_h f(h g^{-1}) \delta_{hx}$.  We let $\lambda = \bigoplus_{x \in G^0} \text{Ind}_x$.  ($\lambda$ is technically not the regular representation of $G$, but it is weakly equivalent to it.  See \cite{pat}.)

In the situation of $G \equiv G(\Lambda)$, let $v \in \Lambda^0$ and $x \in X_v$.  Then $Gx = \{ [\zeta,x] : \zeta \in \CZ v \text{ such that } x \in \widehat{A(\zeta)} \}$.  For $\theta \in \CZ$, let $T_\theta = \text{Ind}\, x (\chi_{[\theta,A(\theta)]})$.  (Then the $\{ T_\theta : \theta \in \CZ\}$ determine Ind$\, x$, as in Theorem \ref{thm toeplitz gens and rels}.)  We get
\begin{align*}
T_\theta \delta_{[\zeta,x]}
&= \chi_{[\theta,A(\theta)]} * \delta_{[\zeta,x]} \\
&= \sum_{\{ [\xi,x] : x \in \widehat{A(\xi)} \}} \chi_{[\theta,A(\theta)]} \bigl([\xi,x] \cdot [\zeta,x]^{-1} \bigr) \delta_{[\xi,x]} \\
&= \sum_{\{ [\xi,x] : x \in \widehat{A(\xi)} \}} \chi_{[\theta,A(\theta)]} \bigl([\xi \overline{\zeta},\Phi_\zeta(x)]\bigr) \delta_{[\xi,x]}.
\end{align*}
If $\Phi_\zeta(x) \not\in \widehat{A(\theta)}$, then $\chi_{[\theta,A(\theta)]} \bigl([\xi \overline{\zeta},\Phi_\zeta(x)]\bigr) = 0$ for any $\xi$.  Suppose that $\Phi_\zeta(x) \in \widehat{A(\theta)}$.  Then $\chi_{[\theta,A(\theta)]} ([\xi \overline{\zeta}, \Phi_\zeta(x)]) = 1$ if and only if $\Phi_{\xi \overline{\zeta}} = \Phi_\theta$ near $\Phi_\zeta(x)$, equivalently, if and only if $\Phi_\xi = \Phi_{\theta \zeta}$ near $x$, or finally, if and only if $[\xi, x] = [\theta \zeta, x]$.  Therefore $T_\theta \delta_{[\zeta,x]} = \delta_{[\theta \zeta, x]}$.

Recall that $(v,v) \in \CZ v$ is the trivial zigzag on $v\Lambda$.  Thus $\Phi_{(v,v)}$ is the identity on $X_v$.  Thus for any $\zeta \in \CZ v$ such that $x \in \widehat{A(\zeta)}$, $\delta_{[\zeta,x]} = T_\zeta \delta_{[(v,v),x]}$.  Thus $\delta_{[(v,v),x]}$ is a cyclic vector for Ind$\,x$.  For $\alpha \in \Lambda v$, note that $\tau^\alpha = \phi_{(v,\alpha)}$, and that $x \in A((v,\alpha))\widehat{\phantom{i}}$.  Thus if we choose $x_v \in v\Lambda^*$ for each $v \in \Lambda^0$, the representation $\pi_\ell$ of Lemma \ref{lem regular representation} is unitarily equivalent to a subrepresentation of $\bigoplus_{v \in \Lambda^0} \text{Ind}_{x_v}$.  We have proved the following proposition.

\begin{Proposition} \label{prop toeplitz onto regular}

The homomorphism $\pi_\ell : \CT(\Lambda) \to \CT_\ell(\Lambda)$ factors through $\CT_r(\Lambda)$.

\end{Proposition}

Thus there is a commutative diagram:
\[
\begin{tikzpicture}[scale=1.5]

\node (00) at (0,0) [rectangle] {$\CT_\ell(\Lambda)$};
\node (01) at (0,1) [rectangle] {$\CT(\Lambda)$};
\node (21) at (2,1) [rectangle] {$\CT_r(\Lambda)$};

\draw[-latex,thick] (01) -- (21) node[pos=0.5, inner sep=0.5pt, above=1pt] {$\pi_r$};
\draw[-latex,thick] (21) -- (00) node[pos=0.5, inner sep=0.5pt, below right=1pt] {$\overline{\pi_\ell}$};
\draw[-latex,thick] (01) -- (00) node[pos=0.5, inner sep=0.5pt, right=1pt] {$\pi_\ell$};

\end{tikzpicture}
\]

We will show that under some conditions, Ind$\,x$ is weakly contained in $\pi_\ell$ for all $x$, and hence that $\pi_\ell$ descends to a faithful representation ($\overline{\pi_\ell}$) of $\CT_r(\Lambda)$.  We also give conditions under which $\overline{\pi_\ell}$ is not faithful on $\CT_r(\Lambda)$.  Recall from Lemma \ref{lem gens and rels in G} that $t_\zeta = \chi_{[\zeta,A(\zeta)]}$ in $C_c(G_2)$ (or in $\CT(\Lambda)$ or $\CT_r(\Lambda)$).

\begin{Proposition} \label{prop finitely aligned regular representation}

Let $\Lambda$ be finitely aligned.  Then $\lambda \preceq \pi_\ell$.

\end{Proposition}

\begin{proof}
We will write $t_\alpha := t_{(r(\alpha),\alpha)}$ for $\alpha \in \Lambda$.  By \cite[Proposition 6.7]{spi_pathcat} (whose proof applies without change to the case of a finitely aligned LCSC), the elements $t_\gamma t_\eps^* t_{\nu_1}t_{\nu_1}^* \cdots t_{\nu_k}t_{\nu_k}^*$ form a total set in $\CT(\Lambda)$.  Let $x \in v \Lambda^*$; we think of $x$ as a directed hereditary subset of $\Lambda$.  To show that Ind$_x \preceq \pi_\ell$, let $\gamma_i$, $\eps_i$, $\nu_{ij} \in \Lambda$ for $1 \le i \le n$, $1 \le j \le k_i$, with $r(\eps_i) = r(\nu_{ij}) = r(x)$ and $s(\gamma_i) = s(\eps_i)$.  We will find $\eta \in \Lambda$ such that $\langle \text{Ind}_x(T_i) \delta_{[v,v,x]},\delta_{[v,v,x]} \rangle = \langle \pi_\ell(T_i) e_\eta,e_\eta \rangle$ for $1 \le i \le n$, where $T_i = t_{\gamma_i} t_{\eps_i}^* \prod_{j=1}^{k_i} t_{\nu_{ij}}t_{\nu_{ij}}^*$.

Let $I = \{ 1 \le i \le n : \eps_i,\; \nu_{ij} \in x \text{ for } 1 \le j \le k_i \}$, and let $J = \{1,\ldots,n\} \setminus I$.  Note that
\begin{align*}
\text{Ind}_x(t_\nu t_\nu^*) \delta_{[v,v,x]}
&= \begin{cases}
\delta_{[v,v,x]}, &\text{ if } \nu \in x \\
0, &\text{ if } \nu \not\in x,
\end{cases} \\
\text{Ind}_x(t_\eps^*) \delta_{[v,v,x]}
&= \begin{cases}
\delta_{s(\eps),\eps,\sigma^\eps x]}, &\text{ if } \eps \in x \\
0, &\text{ if } \eps \not\in x.
\end{cases}.\\
\noalign{Hence} \\
\text{Ind}_x(T_i) \delta_{[v,v,x]}
&= \begin{cases}
\delta_{[\gamma_i,\eps_i, \sigma^{\eps_i} x]}, &\text{ if } i \in I \\
0, &\text{ if } i \in J.
\end{cases}
\end{align*}
Choose $\eps \in \bigl( \bigvee \{ \eps_i, \nu_{ij} : i \in I, 1 \le j \le k_i \} \bigr) \cap x$.  Such an element exists since $x$ is directed.  For $i \in I$ there is $\eps_i'$ such that $\eps = \eps_i \eps_i'$.  If $\Lambda$ is right cancellative we may choose $\eta = \eps$ (as the following argument will show).  In general, note that for $i \in I$,
\[
\langle \text{Ind}_x(T_i) \delta_{[v,v,x]},\delta_{[v,v,x]} \rangle
= \langle \delta_{[\gamma_i,\eps_i,\sigma^{\eps_i} x]},\delta_{[v,v,x]} \rangle
= 1
\]
if and only if  $(\gamma_i,\eps_i,\sigma^{\eps_i} x) \sim (v,v,x)$ (where the equivalence relation $\sim$ is defined in \cite[Definition 4.14]{spi_pathcat}), and equals 0 otherwise.  The value 1 occurs if and only if there are $\mu_i$, $\nu_i$, $z_i$ such that $\sigma^{\eps_i} x = \mu_i z_i$, $x = \nu_i z_i$, $\gamma_i \mu_i = \nu_i$, and $\eps_i \mu_i = \nu_i$.  In this case $\gamma_i \mu_i = \eps_i \mu_i \in x$.  Let $I_0 = \{i \in I : \text{Ind}_x(T_i) \delta_{[v,v,x]} = \delta_{[v,v,x]} \}$.  (Thus
\[
\langle \text{Ind}_x(T_i) \delta_{[v,v,x]},\delta_{[v,v,x]} \rangle
= \begin{cases} 1, &\text{ if } i \in I_0 \\ 0, &\text{ if } i \not \in I_0.) \end{cases}
\]
For $i \not\in I_0$ let $\mu_i = s(\eps_i)$.  Then $\eps_i \mu_i \in x$ for all $i \in I$.  Choose $\eta \in \bigl(\bigvee\{\eps_i \mu_i, \nu_{ij} : i \in I,\ 1 \le j \le k_i \}\bigr) \cap x$.  For $i \in I_0$,
\[
\langle \pi_\ell(T_i) e_\eta,e_\eta \rangle
= \langle \pi_\ell(t_{\gamma_i} t_{\eps_i}^*) e_{\eps_i \mu_i \sigma^{\eps_i\mu_i} \eta}, e_{\eps_i \mu_i \sigma^{\eps_i\mu_i} \eta} \rangle = 1.
\]
If $i \in I \setminus I_0$,
\[
\langle \pi_\ell(T_i) e_\eta,e_\eta \rangle
= \langle \pi_\ell(t_{\gamma_i}t_{\eps_i}^*) e_{\eps_i \sigma^{\eps_i} \eta},e_{\eps_i \sigma^{\eps_i}\eta} \rangle
= \langle e_{\gamma_i \sigma^{\eps_i} \eta}, e_{\eps_i \sigma^{\eps_i} \eta} \rangle
= 0,
\]
since $\gamma_i \sigma^{\eps_i} \eta = \eps_i \sigma^{\eps_i} \eta$ would imply $i \in I_0$.  Finally, if $i \in J$, either $\eps_i \not\in x$ or $\nu_{ij} \not\in x$ for some $j$.  Then either $\eta \not\in \eps_i\Lambda$ or $\eta \not\in \nu_{ij}\Lambda$.  In both cases, $\pi_\ell(T_i) e_\eta = 0$.  Thus $\langle \text{Ind}_x(T_i) \delta_{[v,v,x]},\delta_{[v,v,x]} \rangle = \langle \pi_\ell(T_i) e_\eta,e_\eta \rangle$ for $1 \le i \le n$.  It now follows from \cite[Theorem 3.4.4]{dixmier} that $\lambda \preceq \pi_\ell$.
\end{proof}

\begin{Theorem}

Suppose that $G_2$ is Hausdorff.  Then $\lambda \preceq \pi_\ell$.

\end{Theorem}

\begin{proof}
Let $x \in \Lambda^*$.  (We will view $x$ both as a point in $G_2^{(0)}$ and as an ultrafilter in the ring of sets $\CA$.)  It suffices to prove that Ind$_x \preceq \pi_\ell$.  Let $v = r(x)$.  We first establish the following claim:  if $\theta \in \CZ v$ is such that $[\theta,x] \not= [\text{id},x]$ and $\Phi_\theta(x) = x$, then there is $E \in x$ such that $\varphi_\theta(\alpha) \not= \alpha$ for all $\alpha \in E$.  For this, let $\theta$ be as in the claim.  Since $G_2$ is Hausdorff there are disjoint neighborhoods, $[\theta,F']$ and $[\text{id},F'']$ of $[\theta,x]$ and $[\text{id},x]$.  Then $F'$, $F'' \in x$.  Put $E = F' \cap F'' \in x$.  For $\alpha \in E$ consider the fixed ultrafilter at $\alpha$:  $\CU_{\{\alpha\}}$.  Then $E \in \CU_{\{\alpha\}}$.  Therefore $[\theta,\CU_{\{\alpha\}}] \in [\theta,E]$ and $[\text{id},\CU_{\{\alpha\}}] \in [\text{id},E]$.  Since these are disjoint sets, it follows that $[\theta,\CU_{\{\alpha\}}] \not= [\text{id},\CU_{\{\alpha\}}]$.  Therefore $\varphi_\theta \not= \text{id}$ in any neighborhood of $\CU_{\{\alpha\}}$, that is, on any $A \in \CA$ such that $\alpha \in A$.  Then for any $A \subseteq \alpha\Lambda$ with $\alpha \in A$ there exists $\beta \in A$ such that $\varphi_\theta(\beta) \not= \beta$.  This implies that $\varphi_\theta(\alpha) \not= \alpha$, since $\varphi_\theta(\alpha) = \alpha$ implies $\varphi_\theta(\beta) = \beta$ for all $\beta \in \alpha\Lambda$.  This proves the claim.

Now let $\theta_1$, $\ldots$, $\theta_n \in \CZ v$.  We will find $\eta \in v \Lambda$ such that $\langle \text{Ind}_x(t_{\theta_i}) \delta_{[v,v,x]},\delta_{[v,v,x]} \rangle = \langle \pi_\ell(t_{\theta_i}) e_\eta,e_\eta \rangle$ for $1 \le i \le n$.  We partition $\{1, \ldots, n\}$ as follows:
\begin{align*}
I &= \{i : [\theta_i,x] = [\text{id},x] \} \\
J_0 &= \{i : A(\theta_i) \not\in x \} \\
J_1 &= \{i : A(\theta_i) \in x \text{ and } \Phi_{\theta_i}(x) \not= x \} \\
J_2 &= \{i : A(\theta_i) \in x,\ \Phi_{\theta_i}(x) = x, \text{ and } [\theta_i,x] \not= [\text{id},x] \}.
\end{align*}
For $i \in I$ choose $E_i \in x$ such that $\varphi_{\theta_i}|_{E_i} = \text{id}|_{E_i}$.  For $i \in J_0$ choose $E_i \in x$ such that $A(\theta_i) \cap E_i = \varnothing$.  For $i \in J_1$ there is $E_i \in x$ such that $\Phi_{\theta_i}(\widehat{E_i}) \cap \widehat{E_i} = \varnothing$.  For $\alpha \in E_i$ we have $\CU_{\{\alpha\}} \in \widehat{E_i}$, so $\Phi_{\theta_i}(\CU_{\{\alpha\}}) \not\in \widehat{E_i}$.  But $\Phi_{\theta_i}(\CU_{\{\alpha\}}) = \CU_{\{\varphi_{\theta_i}(\alpha)\}}$, so $\varphi_{\theta_i}(\alpha) \not\in E_i$.  Thus $\varphi_{\theta_i}(\alpha) \not= \alpha$ for all $\alpha \in E_i$.  Finally, for $i \in J_2$, the previous claim implies that there is $E_i \in x$ such that $\varphi_{\theta_i}(\alpha) \not= \alpha$ for all $\alpha \in E_i$.  Since $x$ is an ultrafilter, the set $E := \cap_{1 \le i \le n} E_i \in x$.  Choose $\eta \in E$.  For $i \in I$ we have $\pi_\ell(t_{\theta_i}) e_\eta = e_\eta$, for $i \in J_0$ we have $\pi_\ell(t_{\theta_i}) e_\eta = 0$, and for $i \in J_1 \cup J_2$ we have $\pi_\ell(t_{\theta_i}) e_\eta \perp e_\eta$.  Therefore for all $1 \le i \le n$ we have
\begin{align*}
\langle \pi_\ell(t_{\theta_i}) e_\eta,e_\eta \rangle
&= \begin{cases} 1, &\text{ if } i \in I \\ 0, &\text{ if } i \not\in I \end{cases} \\
&= \langle \text{Ind}_x(t_{\theta_i}) \delta_{[v,v,x]},\delta_{[v,v,x]} \rangle.
\end{align*}
It now follows from \cite[Theorem 3.4.4]{dixmier} that Ind$_x \preceq \pi_\ell$.  Since $x \in \Lambda^*$ was arbitrary, we have that $\lambda \preceq \pi_\ell$.
\end{proof}

\begin{Corollary}

If $\Lambda$ is finitely aligned, or a subcategory of a groupoid, then the canonical map $\CT_r \to \CT_\ell$ is an isomorphism.

\end{Corollary}

\begin{proof}
This follows from Proposition \ref{prop finitely aligned regular representation} and Theorem \ref{thm contained in groupoid hausdorff}.
\end{proof}

We next present examples where $\lambda$ is not weakly contained in $\pi_\ell$.  We first give a sufficient condition for this to occur.

\begin{Lemma} \label{lem not weakly contained}

Suppose that $v \in \Lambda^0$, $A$, $B \in \CD^{(0)}_v$, and $\theta_1$, $\theta_2$, $\theta_3 \in \CZ v$ have the following properties:

\begin{enumerate}

\item $B \subseteq A$.

\item $B$ is not a finite union of proper subsets from $\CD^{(0)}_v$.

\item $A = A(\theta_k) = A(\overline{\theta_k})$ for $k = 1$, 2, 3.

\item For each $\alpha \in A$ at most one $k \in \{1,2,3\}$ satisfies $\varphi_{\theta_k}(\alpha) \not= \alpha$.

\item For each $k \in \{1,2,3\}$, the set $\{ \alpha \in A : \varphi_{\theta_k}(\alpha) \not= \alpha \}$ is infinite.

\item $\varphi_{\theta_3}|_B = \text{id}|_B$.

\item $\varphi_{\theta_1}|_{A \setminus B} = \varphi_{\theta_2}|_{A \setminus B} = \text{id}|_{A \setminus B}$.

\item The orbits of $\varphi_{\theta_1}$ are of bounded odd length.

\end{enumerate}

Then there is $c > 0$ such that for every $\xi \in \ell^2(\Lambda)$, $|\langle \pi_\ell(t_{\theta_1}) \xi,\xi \rangle|
+ |\langle \pi_\ell(t_{\theta_2}) \xi,\xi \rangle|
+ 1 - \text{Re}\langle \pi_\ell(t_{\theta_3}) \xi,\xi \rangle
\ge c \| \xi \|^2$.

\end{Lemma}

\begin{Remark}
\label{rmk special point}

It follows from the second condition in the statement of Lemma \ref{lem not weakly contained}, and from Definition \ref{d.lambdastar}, that $\{ E \in \CD_v^{(0)} : B \subseteq E \}$ defines a point of $\Lambda^* = G_2^{(0)}$.  This point will be crucial to our use of Lemma \ref{lem not weakly contained} to construct examples.

\end{Remark}

We have a preparatory lemma.

\begin{Lemma} \label{lem finite shift}

Let $p > 1$ be an odd integer, and let $S \in B(\ell^2(\IZ/p\IZ))$ be the shift:  $S e_i = e_{i+1}$ for $i \in \IZ/p\IZ$.  Then for all $\xi \in \ell^2(\IZ/p\IZ)$ we have Re$\langle S \xi,\xi \rangle \ge -(\cos \frac{\pi}{p}) \| \xi \|^2$.

\end{Lemma}

\begin{proof}
$S$ is a normal operator with spectrum equal to the set of $p$th roots of unity.  Since $p$ is odd, the minimum of the spectrum of Re$\, S$ is $\cos 2\pi(\tfrac{p-1}{2})/p = -\cos \tfrac{\pi}{p}$.  The lemma follows from this observation.
\end{proof}

\begin{proof} (of Lemma \ref{lem not weakly contained})
Let $\{\gamma_i : i \in I \} \subseteq B$ denote generators of the distinct nontrivial orbits of $\varphi_{\theta_1}$.  Let $p_i > 1$ be the length of the orbit of $\gamma_i$.  Let $p = \max \{p_i : i \in I \}$, and let $C = B \setminus \bigcup_{i \in I} \{ \varphi_{\theta_i}^j(\gamma_i) : 0 \le j < p_i \}$.  Let $\xi \in \ell^2(\Lambda)$.  Then $\xi = \sum_{i \in I} \xi_i + \eta + \mu + \nu$, where $\xi_i \in \text{span} \{ e_{\varphi_{\theta_1}^j(\gamma_i)} : 0 \le j < p_i \}$, $\eta \in \overline{\text{span}} \{ e_\alpha : \alpha \in C \}$, $\mu \in \overline{\text{span}} \{e_\alpha : \alpha \in A \setminus B \}$, and $\nu = \xi - \sum_{i \in I} \xi_i - \eta - \mu \in \overline{\text{span}} \{e_\alpha : \alpha \not\in A \}$.  Let $\zeta = \sum_{i \in I} \xi_i$. Now we have
\begin{align*}
\text{Re} \langle \pi_\ell(t_{\theta_1}) \xi_i, \xi_i \rangle
&\ge -(\cos \tfrac{\pi}{p_i}) \| \xi_i \|^2, \text{ by Lemma \ref{lem finite shift}}, \\
&\ge -(\cos \tfrac{\pi}{p}) \| \xi_i \|^2 \\
\langle \pi_\ell(t_{\theta_k}) \xi_i, \xi_i \rangle
&= \| \xi \|^2,\ k = 2,3 \\
\text{Re} \langle \pi_\ell(t_{\theta_2}) \eta, \eta \rangle
&\ge - \| \eta \|^2 \\
\langle \pi_\ell(t_{\theta_k}) \eta, \eta \rangle
&= \| \eta \|^2,\ k = 1,3 \\
\text{Re} \langle \pi_\ell(t_{\theta_3}) \mu,\mu \rangle
&\le \| \mu \|^2 \\
\langle \pi_\ell(t_{\theta_k}) \mu, \mu \rangle
&= \| \mu \|^2,\ k = 1,2 \\
\langle \pi_\ell(t_{\theta_k}) \nu, \nu \rangle
&= 0,\ k = 1,2,3.
\end{align*}
Note that for $\eps_1$, $\eps_2$ distinct among $\{\xi_i : i \in I \} \cup \{\eta,\mu,\nu\}$, and for any $k = 1$, 2, 3, we have $\langle t_{\theta_k} \eps_1,\eps_2 \rangle = 0$.  We then have
\begin{align*}
|\langle \pi_\ell(t_{\theta_1}) \xi, \xi \rangle|
&\ge \text{Re} \langle \pi_\ell(t_{\theta_1}) \xi, \xi \rangle \\
&\ge -(\cos \tfrac{\pi}{p}) \| \zeta \|^2 + \| \eta \|^2 + \| \mu \|^2 \\
|\langle \pi_\ell(t_{\theta_2}) \xi, \xi \rangle|
&\ge \text{Re} \langle \pi_\ell(t_{\theta_2}) \xi, \xi \rangle \\
&\ge \| \zeta \|^2 - \| \eta \|^2 + \| \mu \|^2 \\
1 - \text{Re} \langle \pi_\ell(t_{\theta_3}) \xi, \xi \rangle
&\ge 1 - (\| \zeta \|^2 + \| \eta \|^2 + \| \mu \|^2) \\
&= \| \nu \|^2.
\end{align*}
If $\| \zeta \|^2 \le \tfrac{1}{2}$, we have
\begin{align*}
|\langle \pi_\ell(t_{\theta_1}) \xi, \xi \rangle|
+ |\langle \pi_\ell(t_{\theta_2}) \xi, \xi \rangle|
+ 1 - \text{Re} \langle \pi_\ell(t_{\theta_3}) \xi, \xi \rangle
&\ge |\langle \pi_\ell(t_{\theta_1}) \xi, \xi \rangle|
+ 1 - \text{Re} \langle \pi_\ell(t_{\theta_3}) \xi, \xi \rangle \\
&\ge -(\cos \tfrac{\pi}{p}) \| \zeta \|^2 + 1 - \| \zeta \|^2 \\
&\ge \tfrac{1}{2}(1 - \cos \tfrac{\pi}{p}).
\end{align*}
If $\| \zeta \|^2 \ge \tfrac{1}{2}$, we have
\[
|\langle \pi_\ell(t_{\theta_1}) \xi, \xi \rangle|
+ |\langle \pi_\ell(t_{\theta_2}) \xi, \xi \rangle|
+ 1 - \text{Re} \langle \pi_\ell(t_{\theta_3}) \xi, \xi \rangle
\ge (1 -\cos \tfrac{\pi}{p}) \| \zeta \|^2 + 2\| \mu \|^2 + \| \nu \|^2
\ge \tfrac{1}{2}(1 - \cos \tfrac{\pi}{p}).
\]
Therefore the statement of the lemma holds with $c = \tfrac{1}{2}(1 - \cos \tfrac{\pi}{p})$.
\end{proof}

\begin{Proposition} \label{prop not weakly contained}

Let $v \in \Lambda^0$, $A$, $B \in \CD^{(0)}_v$, and $\theta_1$, $\theta_2$, $\theta_3 \in \CZ v$ satisfy the conditions of Lemma \ref{lem not weakly contained}.  Then $\lambda \not\preceq \pi_\ell$.

\end{Proposition}

\begin{proof}
Let $x \in v \Lambda^*$ be as in Remark \ref{rmk special point}. It follows from the conditions in Lemma \ref{lem not weakly contained} that $\langle \text{Ind}_x(t_{\theta_k}) \delta_{[v,v,x]}, \delta_{[v,v,x]} \rangle = 0$ for $k = 1$ and 2, and equals 1 for $k = 3$.  By Lemma \ref{lem not weakly contained}, the state $\langle \text{Ind}_x(\cdot) \delta_{[v,v,x]}, \delta_{[v,v,x]} \rangle$ is not the weak$^*$-limit of states which are sums of positive forms associated with $\pi_\ell$.  By \cite[Theorem 3.4.4]{dixmier} Ind$_x \not\preceq \pi_\ell$, and hence $\lambda \not\preceq \pi_\ell$.
\end{proof}

\begin{Example} \label{example different regular representations 1}

Let $p > 1$ be an odd integer.  Let $\Lambda$ be the LCSC given by the following diagram with identifications.  Here $i \in \IZ$, $j \in \IZ/p\IZ$, $k = 0$, 1, 2, 3, 4.

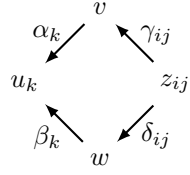
\[
\begin{tikzpicture}

\node (0_0) at (0,0) [circle] {$u_k$};
\node (2_0) at (2,0) [circle] {$z_{ij}$};
\node (1_1) at (1,1) [circle] {$v$};
\node (1_m1) at (1,-1) [circle] {$w$};


\draw[-latex,thick] (1_1) -- (0_0) node[pos=0.5, inner sep=0.5pt, anchor=south east] {$\alpha_k$};
\draw[-latex,thick] (1_m1) -- (0_0) node[pos=0.5, inner sep=0.5pt, anchor=north east] {$\beta_k$};

\draw[-latex,thick] (2_0) -- (1_1) node[pos=0.4, inner sep=0.5pt, anchor=south west] {$\gamma_{ij}$};

\draw[-latex,thick] (2_0) -- (1_m1) node[pos=0.4, inner sep=0.5pt, anchor=north west] {$\delta_{ij}$};

\end{tikzpicture}
\]
The identifications are as follows:
\begin{align*}
\alpha_0 \gamma_{ij}
&= \beta_0 \delta_{ij},\ i \in \IZ,\ j \in \IZ/p\IZ \\
\alpha_1 \gamma_{ij}
&= \begin{cases}
 \beta_1 \delta_{i,j+1}, &\text{ if } i \equiv 1 \pmod 3,\ j \in \IZ/p\IZ \\
 \beta_1 \delta_{ij}, &\text{ if } i \not\equiv 1 \pmod 3,\ j \in \IZ/p\IZ \end{cases} \\
\alpha_2\gamma_{ij}
&= \begin{cases}
 \beta_2 \delta_{i,j+1}, &\text{ if } i \equiv 2 \pmod 3,\ j \in \IZ/p\IZ \\
 \beta_2 \delta_{ij}, &\text{ if } i \not\equiv 2 \pmod 3,\ j \in \IZ/p\IZ \end{cases} \\\alpha_3 \gamma_{ij}
&= \begin{cases}
 \beta_3 \delta_{i+3,j}, &\text{ if } i \equiv 0 \pmod 3,\ j \in \IZ/p\IZ \\
 \beta_3 \delta_{ij}, &\text{ if } i \not\equiv 0 \pmod 3,\ j \in \IZ/p\IZ \end{cases} \\
\alpha_4 \gamma_{ij}
&= \beta_4 \delta_{ij}, \text{ if } i \not\equiv 0 \pmod 3,\ j \in \IZ/p\IZ.
\end{align*}
Let
\begin{align*}
A &= A(\beta_0,\alpha_0) = \{ \gamma_{ij} : i \in \IZ,\ j \in \IZ/p\IZ \} \\
B &= A(\beta_4,\alpha_4) = \{ \gamma_{ij} : i \in \IZ,\ i \not\equiv 0 \pmod 3,\ j \in \IZ/p\IZ \} \\
\theta_k &= (\alpha_k,\beta_k,\beta_0,\alpha_0),\ k = 1,\;2,\;3.
\end{align*}
Note that $\CD_v^{(0)} = \{v\Lambda,\ A,\ B\} \cup \bigl\{ \{ \gamma_{ij} \} : i \in \IZ,\ j \in \IZ/p\IZ \bigr\}$.  It is straightforward to verify the hypothesis of Lemma \ref{lem not weakly contained}.

\end{Example}

\begin{Remark} \label{remark amenable example}

It is not difficult to see that for the LCSC $\Lambda$ of Example \ref{example different regular representations 1}, the algebra $\CT(\Lambda)$ is type I. We show that $\CT(\Lambda)$ is built of copies of $M_p(\IC)$, $\CK$, and $C(\IT)$.

We first consider the situation at the vertex $v$.  Here is a list of the elements of $\CD_v^{(0)}$:
\begin{align*}
\{\gamma_{ij}\} &= A(\gamma_{ij},\gamma_{ij}) \\
A &:= A(\beta_0,\alpha_0) = \{ \gamma_{ij} : i \in \IZ,\ j \in \IZ/p\IZ \} \\
B &:= A(\beta_4,\alpha_4) = \{ \gamma_{ij} : i \in \IZ,\ i \not\equiv 0 \pmod 3,\ j \in \IZ/p\IZ \} \\
v\Lambda &= A(v,v).
\end{align*}
Next we describe the crucial zigzags:
\begin{align*}
\zeta_k &:= (\alpha_0,\beta_0,\beta_k,\alpha_k),\ k = 1,2,3 \\
A(\zeta_k) &= A, \text{ for } k = 1,2,3 \\
\varphi_{\zeta_k}(\gamma_{ij})
 &= \begin{cases}
  \gamma_{i,j+1}, &\text{ if } i \equiv 1 \pmod 3 \text{ and } k = 1, \text{ or } i \equiv 2 \pmod 3 \text{ and } k = 2 \\
  \gamma_{i+3,j}, &\text{ if } i \equiv 0 \pmod 3 \text{ and } k = 3 \\
  \gamma_{ij}, &\text{ otherwise}
 \end{cases} \\
\zeta_{ijk} &:= (\alpha_0, \beta_0, \beta_k, \alpha_k \gamma_{ij}, \gamma_{ij}, v),\ k = 1,2,3,\ i \in \IZ,\ j \in \IZ/p\IZ \\
A(\zeta_{ijk}) &= \{ \gamma_{ij} \} \\
\varphi_{\zeta_{ijk}} (\gamma_{ij}) &= \varphi_{\zeta_k}(\gamma_{ij}).
\end{align*}
(The fact that no other zigzags in $v \CZ v$ are needed follows from Lemma \ref{lem elementary consequences}\eqref{lem elementary consequences 4}.)  Now we list the elements of $X_v$:
\begin{align*}
C_{ij} &:= \{ E \in \CD_v^{(0)} : \gamma_{ij} \in E \} = \{ \{ \gamma_{ij} \}, B, A, v\Lambda \} \\
C_B &:= \{ B, A, v\Lambda \} \\
C_A &:= \{ A, v\Lambda \} \\
C_v &:= \{ v \Lambda \}.
\end{align*}
We note that
\[
\overline{ \{ C_{ij} : i \equiv k \pmod 3 \} } =
 \begin{cases}
   \{ C_{ij} : i \equiv k \pmod 3 \} \cup \{ C_B \}, &\text{ if } k = 1 \text{ or } 2 \\
   \{ C_{ij} : i \equiv 0 \pmod 3 \} \cup \{ C_A \}, &\text{ if } k = 0,
 \end{cases}
\]
since $\widehat{A \setminus B}$ is a neighborhood of $C_{ij}$ for $i \equiv 0 \pmod 3$, but not of $C_{ij}$ for $i \not\equiv 0 \pmod 3$.  We note also that $C_v$ is an isolated point since $\widehat{\{v\}} = \widehat{v\Lambda \setminus A}$ is a neighborhood of $C_v$.

Now we consider the restriction of $G(\Lambda)$ to $X_v$. $X_v$ is not a transversal, but it is nearly one: $G \setminus G X_v G = \{ C_{u_k} : k = 0,1,2,3,4 \}$, where $C_{u_k} = \{ E \in \CD_{u_k}^{(0)} : u_k \in E \} \in X_{u_k}$.  It is easy to see that $C_{u_k}$ is an isolated point of $G^{(0)}$, and also that it only contributes summands isomorphic to $\IC$ to $\CT(\Lambda)$.  We have that $G|_{X_v}$ and $G X_v G$ are equivalent in the sense of \cite{muhrenwil}, so we need only consider $G|_{X_v}$.  Let $Q = G|{X_v}$.  We will describe the regular representations of $Q$.  First we note that if $i \equiv 1 \pmod 3$ then $r([\zeta_{ij1}^\ell, C_{ij}]) = C_{i,j+\ell}$ for $\ell \in \IZ/p\IZ$, and similarly if $i \equiv 2 \pmod 3$ then $r([\zeta_{ij2}^\ell, C_{ij}]) = C_{i,j+\ell}$ for $\ell \in \IZ/p\IZ$.  If $i \equiv 0 \pmod 3$ then $r([\zeta_{ij3}^\ell, C_{ij}] = C_{i + 3\ell,j}$ for $\ell \in \IZ$. (On the other hand, if $i \not\equiv 1 \pmod 3$, e.g., then $[\zeta_{ij1}^\ell, C_{ij}] = [\text{id}, C_{ij}]$, by Lemma \ref{lem elementary consequences}\eqref{lem elementary consequences 4}.)

Next we describe the regular representations of $Q$ induced from points of $X_v$.  If $i \equiv k \pmod 3$, for $k =$ 1, 2, 3,
\[
H_{\text{Ind}\, C_{ij}} =
 \begin{cases}
 \ell^2\{ [\zeta_{ijk}^\ell, C_{ij}] : \ell \in \IZ/p\IZ \}, &\text{ if } k = 1 \text{ or } 2 \\
 \ell^2\{ [\zeta_{ijk}^\ell, C_{ij}] : \ell \in \IZ \}, &\text{ if } k = 3.
 \end{cases}
\]
Then Ind$\,C_{ij} (t_\zeta) e_{[\zeta_{ijk}^\ell, C_{ij}]} = e_{[\zeta \zeta_{ijk}^\ell, C_{ij}]}$ for $\zeta \in v \CZ v$.  Now, for $i \equiv i' \equiv k \pmod 3$, with $k =$ 1, 2, 3, define $W : H_{\text{Ind}\, C_{ij}} \to H_{\text{Ind}\, C_{ij}}$ by $W e_{[\zeta_{ijk}^\ell, C_{ij}]} = e_{[\zeta_{i'j'k}^\ell, C_{i'j'}]}$.  It is easy to check that $W \, \text{Ind}\, C_{ij} \, W^* = \text{Ind}\, C_{i'j'}$ and hence that Ind$\, C_{ij}$ and Ind$\, C_{i'j'}$ are unitarily equivalent.  Then it is also easy to check that
\[
\text{Ind}\, C_{ij}(\CT(\Lambda)) \cong
 \begin{cases}
 M_P(\IC), &\text{ if } i \not\equiv 0 \pmod 3 \\
 C^*(U) + \CK, &\text{ if } i \equiv 0 \pmod 3,
 \end{cases}
\]
where $U \in B(\ell^2 \IZ)$ is the bilateral shift.

It now follows that $\CT(\Lambda)$ is a type I $C^*$-algebra, and hence is nuclear. Then the groupoid $G_2(\Lambda)$ is amenable, and $\CT(\Lambda) = \CT_r(\Lambda)$, by \cite[Corollary 6.2.14]{anaren}.  However $\CT(\Lambda) \not= \CT_\ell(\Lambda)$ by Proposition \ref{prop not weakly contained}, so that this example is not amenable in the sense of Nica --- the universal $C^*$-algebra is not isomorphic to the $C^*$-algebra generated by the regular representation of $\Lambda$. 

\end{Remark}

\begin{Example} \label{example different regular representations 2}

Let $\Lambda$ be the LCSC from Example \ref{example different regular representations 1}.  Define $\Lambda'$ to be the LCSC obtained by identifying the vertices of $\Lambda$, as in Definition \ref{def amalgamation}.  The hypotheses of Lemma \ref{lem not weakly contained} are verified exactly as in Example \ref{example different regular representations 1}.  We observe that $\Lambda'$ is a cancellative monoid, though not embeddable in a group.

\end{Example}

\section{The Wiener-Hopf algebra}
\label{sec wienerhopf}

Let $Y$ be a groupoid, and let $\Lambda \subseteq Y$ be a subcategory such that $Y^0 \subseteq \Lambda$.  Define $J : \ell^2(\Lambda) \to \ell^2(Y)$ to be the inclusion.  We let $L : Y \to B(\ell^2(Y))$ denote the left regular representation of $Y$: $L_{t_1} e_{t_2} = e_{t_1 t_2}$ if $s(t_1) = r(t_2)$, and equals 0 otherwise.  For $t \in Y$ we let $W_t = J^* L_t J \in B(\ell^2(\Lambda))$.  We note that for $t \in Y$ and $\alpha \in \Lambda$,
\[
W_t e_\alpha = \begin{cases} e_{t\alpha}, &\text{ if } s(t) = r(\alpha) \text{ and } t\alpha \in \Lambda \\ 0, &\text{ otherwise.} \end{cases}
\]
It follows that $W_t \not= 0$ if and only if $t \in \Lambda \Lambda^{-1}$.  By Lemma \ref{lem regular representation}, $\alpha \in \Lambda \mapsto W_\alpha \in B(\ell^2(\Lambda))$ is the regular representation $\pi_\ell$.  In this section we will write $W_\zeta$ instead of $T_\zeta$ as in Lemma \ref{lem regular representation}.  Thus if $\zeta = (\alpha_1,\beta_1, \ldots, \alpha_n,\beta_n) \in \CZ$, then $W_\zeta = W_{\alpha_1}^* W_{\beta_1} \cdots W_{\alpha_n}^* W_{\beta_n}$. We note that if we let $t = \alpha_1^{-1} \beta_1 \cdots \alpha_n^{-1} \beta_n \in Y$ then $\varphi_\zeta = t|_{A(\zeta)}$.

The algebra $\CW = C^*(\{W_t : t \in Y \})$ is called the \textit{Wiener-Hopf algebra} of $\Lambda$.  When $(Y,\Lambda) = (\IR^n,P)$ for suitable cones $P$, one studies the algebra generated by operators $W_f$ for $f \in L^1(\IR^n)$; this situation was studied in \cite{muhlyrenault}.  Building on this work, Nica (\cite{nica}) initiated the study of the case of a countable ordered group (such that $\Lambda \cap \Lambda^{-1} = \{ e \}$).  Let $\CW_0$ be the subalgebra of $\CW$ generated by $\{ W_\alpha : \alpha \in \Lambda \}$.  Nica called $(Y,\Lambda)$ {\it quasi-lattice ordered} if it is {\it singly aligned} in the following sense:  for any finite subset $F \subseteq Y$, if $\bigcap_{t \in F \cup \{e\}} t\Lambda \not= \varnothing$, then there is $\alpha \in \Lambda$ such that $\bigcap_{t \in F \cup \{e\}} t\Lambda = \alpha \Lambda$.  Nica proved that $\CW = \CW_0$ for quasi-lattice ordered groups (\cite[2.4]{nica}).  In \cite[Lemma 8.9]{spi_pathcat} it was shown that if $\Lambda$ itself is finitely aligned, then $\CW = \CW_0$ if and only if $(Y,\Lambda)$ is finitely aligned.  The following is a generalization of this result.  (We remark that even though \cite[Lemma 8.9]{spi_pathcat} was stated for the special case described above where $(Y,\Lambda)$ is an ordered group, the statement and proof still hold when $Y$ is a discrete groupoid and $\Lambda$ is a subcategory of $Y$ containing $Y^0$, without the restriction that $\Lambda$ not contain inverses.)  As above, if $\Lambda$ is a subcategory of a groupoid $Y$ we let $\CW_0 = C^*(\{W_\alpha : \alpha \in \Lambda \})$.

\begin{Proposition} \label{prop wienerhopf operator}

Let $Y$ be a groupoid, and let $\Lambda \subseteq Y$ be a subcategory such that $Y^0 \subseteq \Lambda$.  For $t \in Y$, $W_t \in \CW_0$ if and only if there is a finite set $F \subseteq \CZ_\Lambda$ such that

\begin{enumerate}

\item \label{prop wienerhopf operator 1} for all $\zeta \in F$, $\varphi_\zeta = t|_{A(\zeta)}$,

\item \label{prop wienerhopf operator 2} $\Lambda \cap t^{-1} \Lambda = \bigcup_{\zeta \in F} A(\zeta)$.

\end{enumerate}

\end{Proposition}

\begin{proof}
($\Leftarrow$):  It follows from \eqref{prop wienerhopf operator 1} that the operators $\{ W_\zeta : \zeta \in F \}$ are coherent, so that $\bigvee_{\zeta \in F} W_\zeta$ is defined (see the remarks preceding \cite[Theorem 6.3]{spi_pathcat}).  We claim that $W_t = \bigvee_{\zeta \in F} W_\zeta$, from which it follows that $W_t \in \CW_0$.  To prove the claim, let $\beta \in \Lambda$.  If $\beta \in \Lambda \cap t^{-1} \Lambda$ let $\alpha = t \beta \in \Lambda$.  Then $W_t e_\beta = e_\alpha$.  Also, by \eqref{prop wienerhopf operator 2} there is $\zeta \in F$ such that $\beta \in A(\zeta)$.  Then $W_\zeta e_\beta = e_{\varphi_\zeta(\beta)} = e_{t \beta} = e_\alpha$, so that $(\bigvee_{\zeta \in F} W_\zeta) e_\beta = e_\alpha$.  If $\beta \not\in \Lambda \cap t^{-1}\Lambda$, then $t\beta \not \in \Lambda$.  Then $W_t e_\beta = 0$.  On the other hand, for $\zeta \in F$ we have $\beta \not\in A(\zeta)$.  For otherwise, we would have $\beta \in A(\zeta)$ and $\varphi_\zeta(\beta) \in \Lambda$ but $\varphi_\zeta(\beta) = t\beta$, a contradiction.  Hence $W_t e_\beta = 0 = (\bigvee_{\zeta \in F} W_\zeta) e_\beta$.  This proves the claim.

\noindent
($\Rightarrow$):  Let $\| W_t - \sum_{i=1}^n c_i W_{\zeta_i} \| < 1$.  Let $F = \{ \zeta_i : \varphi_{\zeta_i} = t|_{A(\zeta_i)} \}$, and put $A = \sum_{i \in F} c_i W_{\zeta_i}$, $B = \sum_{i \not\in F} c_i W_{\zeta_i}$.  For $\beta \in \Lambda$, if $\beta \in t^{-1} \Lambda$ then $(W_t - A) e_\beta \in \IC e_{t\beta}$ and $B e_\beta \perp e_{t\beta}$, while if $\beta \not\in t^{-1}\Lambda$ then $(W_t - A) e_\beta = 0$.  Hence for all $\beta \in \Lambda$,
\begin{equation*}
(W_t - A) e_\beta \perp B e_\beta. \tag{$*$}
\end{equation*}
Moreover, if $\beta_1 \not= \beta_2$ in $\Lambda$ then
\begin{equation*}
(W_t - A) e_{\beta_1} \perp (W_t - A) e_{\beta_2}, \tag{$**$}
\end{equation*}
since $(W_t - A) e_\beta \in \IC e_{t\beta}$ if $t\beta \in \Lambda$, and equals 0 otherwise.  We have
\begin{align*}
\| (W_t - A) e_\beta \|^2
&\le \| (W_t - A) e_\beta \|^2 + \| B e_\beta \|^2 \\
&= \| (W_t - A - B) e_\beta \|^2, \text{ by } (*), \\
&\le \| W_t - A - B \|^2. \\
\noalign{For $\xi \in \ell^2(\Lambda)$,} \\
\| (W_t - A)\xi \|^2
&= \| \sum_{\beta \in \Lambda} \langle \xi, e_\beta \rangle (W_t - A) e_\beta \|^2 \\
&= \sum_{\beta \in \Lambda} |\langle \xi, e_\beta \rangle|^2 \| (W_t - A) e_\beta \|^2, \text{ by } (**), \\
&\le \| W_t - A - B \|^2 \| \xi \|^2.
\end{align*}
Hence $\| W_t - A \| \le \| W_t - A - B \| < 1$.  Thus we may assume that $B = 0$.  Let $\beta \in \Lambda \cap t^{-1} \Lambda$.  Put $\alpha = t \beta \in \Lambda$.  Then
\[
\| \sum_i c_i W_{\zeta_i} e_\beta \|
= \| W_t e_\beta - (W_t - \sum_i c_i W_{\zeta_i}) e_\beta \|
= \| e_\alpha - (W_t - \sum_i c_i W_{\zeta_i}) e_\beta \|
\ge 1 - \| W_t - \sum_i c_i W_{\zeta_i} \|
> 0.
\]
Then there is $i$ such that $W_{\zeta_i} e_\beta \not= 0$, and hence $\beta \in A(\zeta_i)$.  Therefore $\Lambda \cap t^{-1} \Lambda = \bigcup_i A(\zeta_i)$.
\end{proof}

\begin{Proposition} \label{prop lambda inverse lambda}

Let $\Lambda$ be a subcategory of a groupoid $Y$, and let $t \in \Lambda \Lambda^{-1}$.  Suppose that $t \in \Lambda^{-1} \Lambda$.  Then $W_t \in \CW_0$.

\end{Proposition}

\begin{proof}
Write $t = \gamma^{-1} \delta$ with $\gamma$, $\delta \in \Lambda$.  Then $\Lambda \cap t^{-1} \Lambda = \Lambda \cap \delta^{-1}\gamma \Lambda = \delta^{-1}(\delta \Lambda \cap \gamma \Lambda)$.  Let $\zeta = (\gamma,\delta) \in \CZ_\Lambda$.  Then $\varphi_\zeta = t|_{A(\zeta)}$, and $A(\zeta) = \sigma^\delta (\gamma \Lambda \cap \delta \Lambda) = \Lambda \cap t^{-1} \Lambda$.  Then $W_t \in \CW_0$ by Proposition \ref{prop wienerhopf operator}.
\end{proof}

\begin{Corollary} \label{cor abelian wienerhopf}

If $Y$ is abelian then $\CW = \CW_0$.

\end{Corollary}

We recall that a monoid $\Lambda$ is {\it right reversible} if for all $\alpha$, $\beta \in \Lambda$, $\Lambda \alpha \cap \Lambda \beta \not= \varnothing$.  We may extend this in the obvious way to small categories.

\begin{Definition} \label{def right reversible}

The small category $\Lambda$ is {\it right reversible} if for all $\alpha$, $\beta \in \Lambda$ with $s(\alpha) = s(\beta)$ we have $\Lambda \alpha \cap \Lambda \beta \not= \varnothing$.

\end{Definition}

\begin{Lemma}

Let $\Lambda$ be a subcategory of a groupoid. Then $\Lambda$ is right reversible if and only if $\Lambda \Lambda^{-1} \subseteq \Lambda^{-1} \Lambda$.

\end{Lemma}

\begin{proof}
First suppose that $\Lambda$ is right reversible.  Let $t \in \Lambda \Lambda^{-1}$.  Then $t = \alpha \beta^{-1}$ for some $\alpha$, $\beta \in \Lambda$.  Then $s(\alpha) = s(\beta)$, so there exist $\gamma$, $\delta \in \Lambda$ such that $\gamma \alpha = \delta \beta$, and hence $t = \alpha \beta^{-1} = \gamma^{-1} \delta \in \Lambda^{-1} \Lambda$.

Now suppose that $\Lambda^{-1} \Lambda \subseteq \Lambda^{-1} \Lambda$.  Let $\alpha$, $\beta \in \Lambda$ with $s(\alpha) = s(\beta)$.  Then $t = \alpha \beta^{-1} \in \Lambda \Lambda^{-1}$, so there are $\gamma$, $\delta \in \Lambda$ with $t = \gamma^{-1} \delta$.  Then $\gamma \alpha = \delta \beta \in \Lambda \alpha \cap \Lambda \beta$.
\end{proof}

Ore's theorem for semigroups states that a right reversible cancellative semigroup embeds into a group of fractions.  The same theorem holds for small categories; we give a precise statement and proof in Theorem \ref{thm ore theorem}.  We give a consequence of this result here.

\begin{Corollary}

If $\Lambda$ is a right cancellative right reversible LCSC then $\CW = \CW_0$.

\end{Corollary}

Let $Y$ be a groupoid and $\Lambda \subseteq Y$ a subcategory with $Y^0 \subseteq \Lambda^0$.  As mentioned above, $\CW = \CW_0$ if $(Y,\Lambda)$ is finitely aligned, while $\CW \not= \CW_0$ if $\Lambda$ is finitely aligned but $(Y,\Lambda)$ is not.  If $\Lambda$ is not finitely aligned then $\CW$ and $\CW_0$ are sometimes equal and sometimes not.

\begin{Example}

Here we give an example of a nonfinitely aligned submonoid of a free group, for which $\CW = \CW_0$.  Let $Y = \langle \alpha, \beta, \gamma_1, \gamma_2, \ldots \mid\ \rangle$, and let $\Lambda$ be the submonoid generated by $\{ \alpha, \beta, \gamma_1, \gamma_2, \ldots, \beta^{-1}\alpha\gamma_1, \beta^{-1}\alpha\gamma_2, \ldots \}$. We claim that $\Lambda$ is not finitely aligned (and hence $(Y,\Lambda)$ is not a finitely aligned ordered group), and that $\CW = \CW_0$.

To see this, first let $\Lambda_0$ be the nonfinitely aligned 2-graph
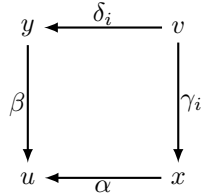
\[
\begin{tikzpicture}[scale=2]

\node (0_1) at (0,1) [rectangle] {$y$};
\node (0_0) at (0,0) [rectangle] {$u$};
\node (1_1) at (1,1) [rectangle] {$v$};
\node (1_0) at (1,0) [rectangle] {$x$};

\draw[-latex,thick] (1_0) -- (0_0) node[pos=0.5, inner sep=0.5pt, anchor=north] {$\alpha$};
\draw[-latex,thick] (0_1) -- (0_0) node[pos=0.5, inner sep=0.5pt, anchor=east] {$\beta$};
\draw[-latex,thick] (1_1) -- (1_0) node[pos=0.5, inner sep=0.5pt, anchor=west] {$\gamma_i$};
\draw[-latex,thick] (1_1) -- (0_1) node[pos=0.5, inner sep=0.5pt, anchor=south] {$\delta_i$};

\end{tikzpicture}
\]
where $i = 1$, 2, $\ldots$, and $\alpha \gamma_i = \beta \delta_i$. Then let $\Lambda$ be the monoid obtained by identifying all of the vertices of $\Lambda_0$, as in Definition \ref{def amalgamation} (and in $\Lambda$ we see that $\delta_i = \beta^{-1} \alpha \gamma_i$).  By Corollary \ref{cor finite alignment in amalgamations}, $\Lambda$ is not finitely aligned.  If $t = \mu \nu^{-1} \in \Lambda \Lambda^{-1}$, let $\mu = (\mu_1, \ldots, \mu_m)$ and $\nu = (\nu_1, \ldots, \nu_n)$ in normal form (as in Lemma \ref{lem normal form}).  We may as well assume that $\mu_m \nu_n^{-1}$ involves no cancellations.  It follows from Proposition \ref{prop common extensions in amalgamations} that $\mu$ is the least upper bound of $t$ and $e$, in that if $\theta \in t \Lambda \cap \Lambda$ then $\theta \in \mu \Lambda$.  Therefore if $\langle W_t e_\xi,e_\eta \rangle \not= 0$ for some $\xi$, $\eta \in \Lambda$, then $t \xi = \eta \in t \Lambda \cap \Lambda$, and hence $\eta = \mu \eta'$ for some $\eta' \in \Lambda$.  Then $\xi = t^{-1} \mu \eta' = (\mu \nu^{-1})^{-1} \mu \eta' = \nu \eta'$.  Then $\langle W_t e_\xi,e_\eta \rangle = \langle W_\mu W_\nu^* e_\xi, e_\eta \rangle$.  It follows that $W_t = W_\mu W_\nu^* \in \CW_0$.

\end{Example}

\begin{Example}

We give an example with $\Lambda$ not finitely aligned such that $\CW \not= \CW_0$.  Let $(Y_i,\Lambda_i)$ be groups with submonoids for $i = 1$, 2.  Suppose that both $(Y_i,\Lambda_i)$ are not finitely aligned, that $\Lambda_1$ is finitely aligned, and that $\Lambda_2$ is not finitely aligned.  Put $Y = Y_1 * Y_2$ and $\Lambda = \Lambda_1 * \Lambda_2$ (so $\Lambda$ is the amalgamation of $\Lambda_1$ and $\Lambda_2$ as in Definition \ref{def amalgamation}).  By Corollary \ref{cor finite alignment in amalgamations}, $\Lambda$ is not finitely aligned.  We will show that $\CW(Y,\Lambda) \not= \CW_0(Y,\Lambda)$.  By \cite[Lemma 8.9]{spi_pathcat} there is $t \in Y_1$ such that $W^{(1)}_t \not \in \CW_0(Y_1,\Lambda_1)$ (we use the superscript $^{(1)}$ to indicate that this is the Wiener-Hopf operator for $t$ for the ordered group $(Y_1,\Lambda_1)$).  We claim that $W_t \not\in \CW_0(Y,\Lambda)$.  For suppose otherwise.  By Proposition \ref{prop wienerhopf operator} there is a finite set $F \subseteq \CZ_\Lambda$ such that $\varphi_\zeta = t|_{A(\zeta)}$ for all $\zeta \in F$, and such that $\Lambda \cap t^{-1} \Lambda = \cup_{\zeta \in F} A(\zeta)$.  Consider $\zeta \in F$:  $\zeta = (\alpha_p,\beta_p, \ldots, \alpha_1,\beta_1)$ with $\alpha_i$, $\beta_i \in \Lambda$ for all $i$.  Note that $\zeta = (\alpha_p,\beta_p) \cdots (\alpha_1,\beta_1)$.  We may assume that $\alpha_i$ and $\beta_i$ are not comparable (since, e.g., $\varphi_{(\gamma,\gamma\delta)} = \varphi_{(r(\delta),\delta)}$).  Let us consider one pair, say $(\alpha_i,\beta_i)$.  Write $\alpha_i = [\alpha_{i1}, \ldots, \alpha_{im_i}]$ and $\beta_i = [\beta_{i1}, \ldots, \beta_{in_i}]$ in normal form, as in Lemma \ref{lem normal form}.  Since $\varphi_\zeta \not= \varnothing$ we must have $\alpha_i \Cap \beta_i$ .  Since $\alpha_i$ and $\beta_i$ are not comparable, Proposition \ref{prop common extensions in amalgamations} implies that $m_i = n_i$, $\alpha_{ij} = \beta_{ij}$ for $j < m_i$, there is $k_i \in \{1,2\}$ such that $\alpha_{im_i}$, $\beta_{im_i} \in \Lambda_{k_i}$, and $\alpha_{im_i} \Cap \beta_{im_i}$ but are not comparable (in $\Lambda_{k_i}$).  But then $\varphi_{(\alpha_i,\beta_i)} = \varphi_{(\alpha_{im_i},\beta_{im_i})}$.  It follows that $t = \alpha_{pm_p}^{-1} \beta_{pm_p} \cdots \alpha_{1m_1}^{-1} \beta_{1m_1}$.  Since $t \in Y_1$ we must have $k_1 = \cdots = k_p = 1$.  Then $\varphi_\zeta = \varphi_{\zeta'}$, where $\zeta' = (\alpha_{pm_p}, \beta_{pm_p}, \ldots, \alpha_{1m_1}, \beta_{1m_1}) \in \CZ_{\Lambda_1}$.  Now Proposition \ref{prop wienerhopf operator} implies that $W_t^{(1)} \in \CW_0(Y_1,\Lambda_1)$, a contradiction.

\end{Example}

\section{Ore's theorem for LCSC's}
\label{sec ore theorem}

Recall the term {\it right reversible} for LCSC's from  Definition \ref{def right reversible}.  Ore's theorem for semigroups states that a right reversible cancellative semigroup embeds in a group of fractions, and that this embedding is universal (\cite[Theorem 1.23]{cliffordpreston}).  The analogous result for small categories has essentially the same proof.  We provide one here.

Let $\Lambda$ be a right cancellative, right reversible LCSC. Let $S = \{ (\alpha,\beta) \in \Lambda^2 : r(\alpha) = r(\beta) \}$. We think of $(\alpha,\beta)$ as representing ``$\alpha^{-1}\beta$''.  To make this rigorous, we define an equivalence relation $\sim$ on $S$ as follows:  $(\alpha,\beta) \sim (\gamma,\delta)$ if there are $x$, $y \in \Lambda$ such that $x \alpha = y \gamma$ and $x \beta = y \delta$.  (Note that $(\alpha,\beta) \sim (\gamma,\delta)$ implies that $s(\alpha) = s(\gamma)$ and $s(\beta) = s(\delta)$.)  It is immediate that $\sim$ is reflexive and symmetric.  We prove transitivity.  Let $(\alpha,\beta) \sim (\gamma,\delta)$ and $(\gamma,\delta) \sim (\zeta,\eta)$.  There are $x$, $y$, $z$, $w$ such that $x\alpha = y\gamma$ and $x\beta = y\delta$, and such that $z\gamma = w\zeta$ and $z\delta = w\eta$.  Note that $s(y) = r(\gamma) = s(z)$.  Then by right reversibility there are $u$, $v$ such that $uy = vz$.  Now we have $ux\alpha = uy\gamma = vz\gamma = vw\zeta$ and $ux\beta = uy\delta = vz\delta = vw\eta$.  Therefore $(\alpha,\beta) \sim (\zeta,\eta)$.

We will define the structure of a groupoid on $S / \sim$.  The multiplication on equivalence classes must be defined in terms of representatives, so we give a preliminary version in that context.

\begin{Definition} \label{def preliminary product}

Let $(\alpha,\beta)$, $(\gamma,\delta) \in S$ with $s(\beta) = s(\gamma)$.  By right reversibility there are $x$, $y \in \Lambda$ such that $x\beta = y\gamma$. We write
\[
(\alpha,\beta) \underset{x \cdot \beta = y \cdot \gamma}{\times} (\gamma,\delta) = (x\alpha,y\delta).
\]
(The idea, of course, is that in terms of fractions we are performing the product:
\[
(\alpha^{-1}\beta)(\gamma^{-1}\delta) = (x\alpha)^{-1} (x\beta) (y\gamma)^{-1} (y\delta) = (x\alpha)^{-1} (y\delta).)
\]

\end{Definition}

\begin{Lemma} \label{lem ore theorem one}

Let $(\alpha,\beta)$, $(\gamma,\delta) \in S$ with $s(\beta) = s(\gamma)$.  Choose $x$, $y \in \Lambda$ such that $x\beta = y\gamma$. Let $\mu \in \Lambda r(\alpha)$, and choose $z$, $w \in \Lambda$ such that $z \mu \beta = w \gamma$.  Then $(\alpha,\beta) \underset{x \cdot \beta = y \cdot \gamma} {\times} (\gamma,\delta) \sim (\mu\alpha,\mu\beta) \underset{z \cdot \mu \beta = w \cdot \gamma}{\times} (\gamma,\delta)$.

\end{Lemma}

\begin{proof}
Note that $(\alpha,\beta) \underset{x \cdot \beta = y \cdot \gamma} {\times} (\gamma,\delta) = (x\alpha,y\delta)$, and $(\mu\alpha,\mu\beta) \underset{z \cdot \mu \beta = w \cdot \gamma}{\times} (\gamma,\delta) = (z \mu \alpha,w \delta)$.  By right reversibility there are $\zeta$, $\eta \in \Lambda$ such that $\zeta x = \eta z \mu$.  Then $\zeta x \alpha = \eta z \mu \alpha$, and $\zeta y \gamma = \zeta x \beta = \eta z \mu \beta = \eta w \gamma$.  By right cancellation we have $\zeta y = \eta w$, and hence $\zeta y \delta = \eta w \delta$.  Therefore $(x\alpha,y\delta) \sim (z \mu \alpha,w\delta)$.
\end{proof}

\begin{Lemma} \label{lem ore theorem two}

Let $(\alpha,\beta)$, $(\gamma,\delta) \in S$ with $s(\beta) = s(\gamma)$.  Let $(\alpha,\beta) \sim (\alpha',\beta')$.  Choose $x$, $y$, $x'$, $y'$ such that $x\beta = y\gamma$ and $x'\beta' = y' \gamma$. Then $(\alpha,\beta) \underset{x \cdot \beta = y \cdot \gamma} {\times} (\gamma,\delta) \sim (\alpha',\beta') \underset{x' \cdot \beta' = y' \cdot \gamma} {\times} (\gamma,\delta)$.
\end{Lemma}

\begin{proof}
Let $\mu$, $\mu'$ be such that $\mu \alpha = \mu' \alpha'$ and $\mu \beta = \mu' \beta'$.  Choose $z$, $w$ as in Lemma \ref{lem ore theorem one}. Then
\begin{align*}
(\alpha,\beta) \underset{x \cdot \beta = y \cdot \gamma} {\times} (\gamma,\delta)
&= (x\alpha,y\delta) \\
&\sim (z\mu\alpha,w\delta), \text{ by Lemma \ref{lem ore theorem one},} \\
&= (z \mu' \alpha', w\delta) \\
&\sim (x' \alpha', y' \delta), \text{ by Lemma \ref{lem ore theorem one},} \\
&= (\alpha',\beta') \underset{x' \cdot \beta' = y' \cdot \gamma} {\times} (\gamma,\delta). \qedhere
\end{align*}
\end{proof}

\begin{Lemma} \label{lem ore theorem three}

Let $(\alpha,\beta)$, $(\gamma,\delta) \in S$ with $s(\beta) = s(\gamma)$.  Let $(\gamma,\delta) \sim (\gamma',\delta')$.  Choose $x$, $y$, $x'$, $y'$ such that $x\beta = y\gamma$ and $x'\beta = y' \gamma'$. Then $(\alpha,\beta) \underset{x \cdot \beta = y \cdot \gamma} {\times} (\gamma,\delta) \sim (\alpha,\beta) \underset{x' \cdot \beta = y' \cdot \gamma'} {\times} (\gamma',\delta')$.

\end{Lemma}

\begin{proof}
The proof is analogous to the proof of Lemma \ref{lem ore theorem two}.
\end{proof}

\begin{Lemma} \label{lem ore theorem four}

Let $(\alpha, \beta)$, $(\gamma,\delta) \in S$ with $s(\beta) = s(\gamma)$.  Let $(\alpha,\beta) \sim (\alpha',\beta')$ and $(\gamma,\delta) \sim (\gamma',\delta')$.  Let $x$, $y$, $x''$, $y''$ be such that $x \beta = y \gamma$ and $x'' \beta' = y'' \gamma'$. Then $(\alpha,\beta) \underset{x \cdot \beta = y \cdot \gamma}{\times} (\gamma,\delta) \sim (\alpha',\beta') \underset{x'' \cdot \beta' = y'' \cdot \gamma'}{\times} (\gamma',\delta')$.

\end{Lemma}

\begin{proof}
Choose $x'$, $y'$ such that $x' \alpha ' = y' \gamma$. We have
\begin{align*}
(\alpha,\beta) \underset{x \cdot \beta = y \cdot \gamma} {\times} (\gamma,\delta)
&\sim (\alpha',\beta') \underset{x' \cdot \beta' = y' \cdot \gamma} {\times} (\gamma,\delta), \text{ by Lemma \ref{lem ore theorem two},} \\
&\sim (\alpha',\beta') \underset{x'' \cdot \beta' = y'' \cdot \gamma'} {\times} (\gamma',\delta'), \text{ by Lemma \ref{lem ore theorem three}.} \qedhere
\end{align*}
\end{proof}

\begin{Lemma} \label{lem ore theorem five}

Let $(\alpha,\beta)$, $(\alpha',\beta') \in S$, and assume that $(\alpha,\beta) \sim (\alpha',\beta')$.  Then $(\beta,\alpha) \sim (\beta',\alpha')$.

\end{Lemma}

\begin{proof}
The proof is immediate.
\end{proof}

\begin{Definition} \label{def ore groupoid}

Let $\Lambda$ be a right cancellative right reversible LCSC.  Let $G = S / \sim$.  We define multiplication on $G$ as follows.  Let $G^2 = \{([\alpha,\beta],[\gamma,\delta]) \in G \times G : s(\beta) = s(\gamma)\}$.  For $([\alpha,\beta],[\gamma,\delta]) \in G^2$ define  $[\alpha,\beta] [\gamma,\delta] = [(\alpha,\beta) \underset{x \cdot \beta = y \cdot \gamma}{\times} (\gamma,\delta)]$ for any choice of $x$, $y$ such that $x\beta = y \gamma$.  We define inversion by $[\alpha,\beta] = [\beta,\alpha]$.  These operations are well defined by Lemmas \ref{lem ore theorem four} and \ref{lem ore theorem five}.

\end{Definition}

\begin{Theorem} \label{thm ore theorem}

Let $\Lambda$ be a right cancellative right reversible LCSC.  The operations of Definition \ref{def ore groupoid} make $G$ into a groupoid.  The map $\iota : \alpha \in \Lambda \mapsto [r(\alpha),\alpha] \in G$ is an injective homomorphism.  Every element of $G$ can be written in the form $\iota(\alpha)^{-1} \iota(\beta)$. If $\pi : \Lambda \to H$ is a homomorphism of $\Lambda$ to a groupoid $H$, there is a unique homomorphism $\widetilde{\pi} : G \to H$ such that $\pi = \widetilde{\pi} \circ \iota$.

\end{Theorem}

\begin{proof}
We verify \cite[p. 7, conditions (i) and (ii)]{pat}.  For \cite[p. 7, condition (i)]{pat}, let $([\alpha,\beta], [\gamma,\delta])$, $([\gamma,\delta], [\zeta,\eta]) \in G^2$.  It is immediate that $([\alpha,\beta][\gamma,\delta], [\zeta,\eta])$, $([\alpha,\beta], [\gamma,\delta][\zeta,\eta]) \in G^2$.  We show that $([\alpha,\beta][\gamma,\delta])[\zeta,\eta] = [\alpha,\beta]([\gamma,\delta][\zeta,\eta])$. Let $x$, $y$ with $x\beta = y\gamma$.  Let $z$, $w$ be such that $z y \delta = w \zeta$.  Then
\[
([\alpha,\beta][\gamma,\delta])[\zeta,\eta] 
= [(\alpha,\beta) \underset{x \cdot \beta = y \cdot \gamma}{\times} (\gamma,\delta)][\zeta,\eta]
= [x\alpha,y\delta][\zeta,\eta] 
=[(x\alpha,y\delta) \underset{z \cdot y\delta = w \cdot \zeta}{\times} (\zeta,\eta)]
= [zx\alpha,w\eta].
\]
Since $x\beta = y \gamma$ we also have $zx\beta = zy\gamma$.  Then
\[
[\alpha,\beta]([\gamma,\delta][\zeta,\eta]) 
= [\alpha,\beta][(\gamma,\delta) \underset{zy \cdot \delta = w \cdot \zeta}{\times} (\zeta,\eta)]
= [\alpha,\beta][zy\gamma,w\eta] 
= [(\alpha,\beta) \underset{zx \cdot \beta = r(z) \cdot zy \gamma}{\times} (zy\gamma,w\eta)]
= [zx\alpha,w\eta].
\]
For \cite[p. 7, condition (ii)]{pat}, note that $([\alpha,\beta],[\beta,\alpha]) \in G^2$ for any $[\alpha,\beta] \in G$. Moreover, elementary calculations show that $[\alpha,\beta][\beta,\alpha] = [\alpha,\alpha] = [s(\alpha),s(\alpha)] = s([\beta,\alpha])$.  Now let $([\alpha,\beta],[\gamma,\delta]) \in G^2$.  Then by condition (i),
\[
[\beta,\alpha]([\alpha,\beta][\gamma,\delta])
= ([\beta,\alpha][\alpha,\beta])[\gamma,\delta]
= [\beta,\beta][\gamma,\delta]
= s([\beta,\beta])[\gamma,\delta]
= [\gamma,\delta],
\]
and similarly, $([\alpha,\beta][\gamma,\delta])[\delta,\gamma] = [\alpha,\beta]$.

Define $\iota : \Lambda \to G$ as in the statement.  Note that if $s(\alpha) = r(\beta)$, then
\[
\iota(\alpha) \iota(\beta)
= [r(\alpha),\alpha][r(\beta),\beta]
= [(r(\alpha),\alpha) \underset{r(\alpha) \cdot \alpha = \alpha \cdot r(\beta)}{\times} (r(\beta),\beta)]
= [r(\alpha), \alpha\beta]
= \iota(\alpha \beta).
\]
Therefore $\iota$ is a homomorphism.  Suppose that $\iota(\alpha) = \iota(\beta)$.  Then $(r(\alpha),\alpha) \sim (r(\beta),\beta)$, so there are $x$, $y$ such that $x r(\alpha) = y r(\beta)$ and $x\alpha = y \beta$.  But then $x = y$, and hence $\alpha = \beta$.  Therefore $\iota$ is injective.

Finally, let $\pi : \Lambda \to H$ be a homomorphism of $\Lambda$ to a groupoid $H$.  Define a map $\pi_S : S \to H$ by $\pi_S(\alpha,\beta) = \pi(\alpha)^{-1} \pi(\beta)$.  If $(\alpha,\beta) \sim (\alpha',\beta')$, then there are $\mu$, $\mu'$ such that $\mu \alpha = \mu' \alpha'$ and $\mu \beta = \mu' \beta'$.  Then
\[
\pi_S(\alpha,\beta)
= \pi(\alpha)^{-1} \pi(\beta)
= \pi(\alpha)^{-1} \pi(\mu)^{-1} \pi(\mu) \pi(\beta)
= \pi(\mu \alpha)^{-1} \pi(\mu \beta)
= \pi(\mu' \alpha')^{-1} \pi(\mu' \beta')
= \pi_S(\alpha',\beta').
\]
Therefore there is a well defined map $\widetilde{\pi} : G \to H$ given by $\widetilde{\pi}([\alpha,\beta]) = \pi_S(\alpha,\beta) = \pi(\alpha)^{-1} \pi(\beta)$. It follows from this formula that $\widetilde{\pi}$ is a homomorphism.  Moreover
\[
\widetilde{\pi} \circ \iota (\alpha) = \widetilde{\pi}([r(\alpha),\alpha]) = \pi_S(r(\alpha),\alpha) = \pi(r(\alpha))^{-1} \pi(\alpha) = r(\pi(\alpha))^{-1} \pi(\alpha) = \pi(\alpha).
\]
Since
\[
\widetilde{\pi}([\alpha,\beta]) = \widetilde{\pi}([\alpha,r(\alpha)][r(\beta),\beta]) \widetilde{\pi}([r(\alpha),\alpha]^{-1} [r(\beta),\beta]) = (\widetilde{\pi} \circ \iota (\alpha))^{-1} \widetilde{\pi} \circ \iota (\beta),
\]
it is clear that $\widetilde{\pi}$ is unique.
\end{proof}

\section{Subcategories}
\label{sec subcategories}

In this section we present some results concerning functoriality of $\CT(\Lambda)$ with respect to subcategories.  In the case of graph algebras, this was developed in \cite{spi_graphalg}.  For LCSC's the results are less definitive, but still give useful decompositions of the algebras.  In the case of a subgraph of a directed graph, the (Toeplitz) algebra of the subgraph is a subalgebra of that of the larger graph.  However for more general LCSC's this is not true.  It is necessary to use the larger LCSC in some way in the construction of the algebra of the subcategory.  There are various ways to do this.  We first give a construction that is built directly on the unit space of the groupoid of the larger LCSC.  This has the advantage that the $C^*$-algebra of the pair is automatically a subalgebra of the $C^*$-algebra of the larger LCSC.  The disadvantage is that there is not a characterization by generators and relations using only the categories.  Such a characterization can be given, and we do that afterwards.  This is modeled on the treatment for {\it relative categories of paths} in \cite{spi_pathcat}, but there are subtle errors in that treatment.  These errors become magnified in the context of LCSC's.  In this paper we have carefully managed these difficulties for the nonrelative case by considering the two groupoids of Section \ref{sec groupoid}.  In fact, if relative LCSC's are defined in analogy with relative categories of paths as in \cite{spi_pathcat}, then nearly all of the results of this paper apply with only notational changes to relative LCSC's.  We indicate below how this is done, correcting the errors of \cite{spi_pathcat} in the process.

\vspace*{.1 in}

\noindent
{\bf Constructions for an LCSC with respect to a subcategory}

\begin{Definition} \label{def subcategory groupoids}

Let $\Lambda$ be an LCSC, and let $\Lambda_0 \subseteq \Lambda$ be a subcategory.  We define $X^{\Lambda_0}(\Lambda) = \bigsqcup_{v \in \Lambda_0^0} X_v(\Lambda)$.  We set $\CZ(\Lambda_0) * X^{\Lambda_0}(\Lambda) = \{ (\zeta,x) \in \CZ(\Lambda_0) \times X^{\Lambda_0}(\Lambda) : s(\zeta) = r(x) \}$.  We restrict $\sim_1$ and $\sim_2$ of Definition \ref{def groupoids} to $\CZ(\Lambda_0) * X^{\Lambda_0}(\Lambda)$ (and continue to use $\sim_1$ and $\sim_2$ to denote the restrictions). Note that for $\zeta \in \CZ(\Lambda_0)$ the partial homeomorphism $\Phi_\zeta$ of $X(\Lambda)$ actually has domain $\widehat{A(\zeta)}$ contained in $X^{\Lambda_0}(\Lambda)$, and the partial bijection $\varphi$ has domain $A(\zeta)$ contained in $\Lambda_0^0 \cdot \Lambda \cdot \Lambda_0^0$.  For $i = 1$, 2 we let $G_i^{\Lambda_0}(\Lambda) = ( \CZ(\Lambda_0) * X^{\Lambda_0}(\Lambda) ) / \sim_i$.  Then $G_i^{\Lambda_0}(\Lambda)$ is a clopen subgroupoid of $G_i(\Lambda)$, with unit space $X^{\Lambda_0}(\Lambda)$.  We refer to $G_i^{\Lambda_0}(\Lambda)$ as the {\it groupoids of $\Lambda$ with respect to the subcategory $\Lambda_0$}.

\end{Definition}

\begin{Theorem} \label{thm subcategories}

Let $\Lambda$ be an LCSC, and let $\Lambda_0 \subseteq \Lambda$ be a subcategory.  The inclusion $C_c(G_i^{\Lambda_0}(\Lambda)) \subseteq C_c(G_i(\Lambda))$ induces an injective $*$-homomorphism $C^*(G_i^{\Lambda_0}(\Lambda)) \to C^*(G_i(\Lambda))$.

\end{Theorem}

\begin{proof}
(This theorem is true generally for a clopen subgroupoid of an \'etale groupoid.) Fix $i \in \{1,2\}$.  Let $G = G_i(\Lambda)$ and $H = G_i^{\Lambda_0}(\Lambda)$.  Then $H$ is a clopen subgroupoid of $G$.  Let $\alpha : C^*(H) \to C^*(G)$ be the $*$-homomorphism defined by the inclusion $C_c(H) \subseteq C_c(G)$.  We follow \cite[Section 2]{ionwil}. Then $G \cdot H^{(0)}$ is a free and proper right $H$-space.  Let $H^G = (G \cdot H^{(0)} *_s G\cdot H^{(0)}) / H$ (where $H$ acts on the right by the diagonal action).  Then $H^G$ is an (ample) \'etale groupoid and $G \cdot H^{(0)}$ is a free and proper left $H^G$-space. There are $C_c(H^G)$- and $C_c(H)$-valued inner products on $C_c(G \cdot H^{(0)}$.  Let $X$ be the completion of $C_c(G \cdot H^{(0)})$ as a $C^*(H^G)$-$C^*(G)$ imprimitivity bimodule.  There is a left action of $C_c(G)$ on $C_c(G \cdot H^{(0)})$ that extends to a nondegenerate homomorphism $C^*(G) \to L(X)$.  We restrict this action to $C_c(H) \subseteq C_c(G)$.  Write $C_c(G \cdot H^{(0)}) = C_c(H) \oplus C_c(G \cdot H^{(0)} \setminus H)$, an invariant decomposition for the left and right actions of $C_c(H)$.  Completing in $X$ gives $X = X_0 \oplus X_1$, invariant for the left and right actions of $C^*(H)$. The formulas in \cite{ionwil} restrict to the usual left and right convolution of $C_c(H)$.  Therefore $X_0 = C^*(H)$ is the standard correspondence over $C^*(H)$.  It follows that if $L$ is a faithful representation of $C^*(H)$, and if Ind$\, L$ is the induced representation of $C^*(G)$, then Ind$\, L \circ \alpha$ contains $L$.  Therefore Ind$\, L \circ \alpha$ is injective, and hence $\alpha$ is injective.
\end{proof}

\begin{Definition} \label{def toeplitz algebra with respect to a subcategory}

We let $\CT^{\Lambda_0}(\Lambda) = C^*(G_2^{\Lambda_0}(\Lambda))$, the {\it Toeplitz algebra with respect to the subcategory} $\Lambda^0$.

\end{Definition}

\begin{Corollary}

Let $\Lambda$ be an LCSC, and let $\Lambda_1 \subseteq \Lambda_2 \subseteq \cdots \subseteq \Lambda$ be subcategories such that $\Lambda = \cup_n \Lambda_n$.  Then $C^*(G_i(\Lambda)) = \overline{ \cup_n C^*(G_i^{\Lambda_n}(\Lambda))}$.

\end{Corollary}

\noindent
{\bf Relative LCSC's}

\vspace*{.1 in}

We now turn to {\it relative LCSC's}.  The general theory for categories of paths in \cite[Sections 2 - 6]{spi_pathcat} was developed for relative categories of paths in an attempt to build functoriality into the construction.  However there is an error in that paper in the relative case.  Namely, the definition of the groupoid uses the equivalence relation that we denote by $\sim_1$ in this paper, whereas the characterization by generators and relations (\cite[Theorem 6.1]{spi_pathcat}) uses the relation (4)$_2$ (from Definition \ref{def toeplitz relations}).  As we have seen in Proposition \ref{prop lcsc without inverses}, this doesn't matter for categories of paths, since in that case we have $\sim_1 = \sim_2$. (Therefore \cite[Theorem 6.1]{spi_pathcat} and its proof are valid for (nonrelative) categories of paths.)  For relative categories of paths this need not be the case.  (Specifically, this error occurs in two places. First, in the claim that $\varphi_\zeta = \text{id}_{A(\zeta)}$ if and only if $\Phi_\zeta = \text{id}_{\widehat{A(\zeta)}}$ in \cite[Section 6, paragraph 1]{spi_pathcat}, the {\it if} direction is false in general for relative categories of paths.  Second, in the proof of \cite[Theorem 6.1, p. 579 line 5]{spi_pathcat} the assertion that $\Phi_{\zeta_j}|_{\widehat{A(\xi)}} = \Phi_\xi$ for $j = 1$, 2 implies that $\varphi_{\zeta_1 \overline{\xi}\xi} = \varphi_{\zeta_2 \overline{\xi}\xi}$ is not valid for relative categories of paths.)

We describe explicitly the error in the proof of \cite[Theorem 6.1]{spi_pathcat}.  The assertion that a family $\{T_\zeta : \zeta \in \CZ(\Lambda_0,\Lambda) \subseteq B(H) \}$ satisfying (1) - (4) of that theorem must define a representation of $C^*(G(\Lambda_0,\Lambda))$ is false.  Namely, since \cite[Theorem 6.1(4)]{spi_pathcat} is actually (4)$_2$ of this paper, such a family defines a representation $\pi$ of $C^*(G_2(\Lambda_0,\Lambda))$ (by Theorem \ref{thm toeplitz gens and rels}) given by $\pi(t_\zeta) = T_\zeta$.  Corollary \ref{cor 2 maps to 1} (adapted to relative LCSC's) implies that there is a surjection $C^*(G_2(\Lambda_0,\Lambda)) \to C^*(G_1(\Lambda_0,\Lambda))$, but $\pi$ need not factor through this surjection.  For example, suppose that there is $\zeta \in \CZ(\Lambda_0,\Lambda)$ such that $\varphi_\zeta \not= \text{id}_{A(\zeta)}$ but $\Phi_\zeta = \text{id}_{\widehat{A(\zeta)}}$.  Then $\pi_\ell(t_\zeta) \not= \pi_\ell(t_\zeta)^* \pi_\ell(t_\zeta)$, and hence $\pi_\ell$ is not a representation of $C_c(G_1(\Lambda_0,\Lambda))$.  However $\{ \pi_\ell(t_\zeta) \}$ do satisfy (1) - (4) of \cite[Theorem 6.1]{spi_pathcat} (since (4) does not apply to this particular $\zeta$).

However, our treatment in this paper of the two groupoids $G_1$ and $G_2$ corrects these errors.  We now present the relative theory for LCSC's.  (In the process we will indicate one other minor error in \cite{spi_pathcat}.)

\begin{Definition} \label{def relative lcsc}
(cf. \cite[Definitions 2.11 and 2.12, Remark 2.13]{spi_pathcat})
Let $\Lambda$ be an LCSC and $\Lambda_0 \subseteq \Lambda$ a subcategory.  We refer to the pair $(\Lambda_0,\Lambda)$ as a {\it relative LCSC}.  We let $\CZ_{(\Lambda_0,\Lambda)} = \{ \zeta \in \CZ_\Lambda : \zeta = (\alpha_1,\beta_1,\ldots,\alpha_n,\beta_n) \text{ with } \alpha_i,\beta_i \in \Lambda_0 \}$.  For $v \in \Lambda_0^0$ we let $\CD(\Lambda_0,\Lambda)_v^{(0)} = \{A_\Lambda(\zeta) : \zeta \in \CZ_{(\Lambda_0,\Lambda)}v,\ A(\zeta) \not= \varnothing \}$ and we let $\CA(\Lambda_0,\Lambda)_v$ be the ring of sets generated by $\CD(\Lambda_0,\Lambda)_v^{(0)}$.

\end{Definition}

We remark that while $A_\Lambda(\zeta)$ and $\varphi_\zeta$ in Definition \ref{def relative lcsc} do not depend on $\Lambda_0$, the collections $\CD(\Lambda_0,\Lambda)_v^{(0)}$ and the rings $\CA(\Lambda_0,\Lambda)_v$ do.  All of the definitions and results of sections \ref{sec groupoid} - \ref{sec gens and rels} of this paper hold for relative LCSC's, with the same proofs --- except for Propositions \ref{prop lcsc without inverses} and \ref{prop equality of groupoids}, and the results about finitely aligned LCSC's.  In particular we explicitly state the relative version of Theorem \ref{thm toeplitz gens and rels}.

\begin{Theorem} \label{thm relative toeplitz gens and rels}

Let $(\Lambda_0,\Lambda)$ be a relative LCSC.  Let $i \in \{1,2\}$.  Then $C^*(G_i(\Lambda_0,\Lambda))$ is the universal $C^*$-algebra generated by a family $\{T_\zeta : \zeta \in \CZ(\Lambda_0,\Lambda) \}$ satisfying  (1) - (3) and (4)$_i$ of Definition \ref{def toeplitz relations}.

\end{Theorem}

Note that $\CA(\Lambda_0,\Lambda)_v \subsetneq \CA(\Lambda)_v$, and hence $A(\Lambda_0,\Lambda)_v$ is a unital $C^*$-subalgebra of $A(\Lambda)_v$.  Therefore there is a continuous surjection $\pi_v : X(\Lambda)_v \to X(\Lambda_0,\Lambda)_v$.  Thus we have a continuous proper surjection $\pi : X(\Lambda) \cap r^{-1}(\Lambda_0^0) \to X(\Lambda_0,\Lambda)$.

We now consider consequences of Theorem \ref{thm toeplitz gens and rels} in the case of a relative LCSC.  For $\zeta \in \CZ(\Lambda_0,\Lambda)$ we have $A_{(\Lambda_0,\Lambda)}(\zeta) = A_\Lambda(\zeta)$ and $\varphi_\zeta^{(\Lambda_0,\Lambda)} = \varphi_\zeta^\Lambda$.  Therefore relations (3) and (4)$_2$ of Definition \ref{def toeplitz relations} are the same for $(\Lambda_0,\Lambda)$ as for $\Lambda$.  Thus the generators $\{t_\zeta : \zeta \in \CZ(\Lambda_0,\Lambda) \}$ of $C^*(G_2(\Lambda_0,\Lambda))$ satisfy the relations defining $C^*(G_2(\Lambda))$.  We therefore have the following results.

\begin{Theorem} \label{thm relative lcsc homomorphism}

Let $(\Lambda_0,\Lambda)$ be a relative LCSC.  There is a $*$-homomorphism $\CT(\Lambda_0,\Lambda) \to \CT(\Lambda)$ defined by mapping generators to (corresponding) generators.

\end{Theorem}

\begin{Corollary} \label{cor lcsc inductive limit}

Let $\Lambda$ be an LCSC, and let $\Lambda_1 \subseteq \Lambda_2 \subseteq \cdots \subseteq \Lambda$ be subcategories such that $\Lambda = \cup_n \Lambda_n$.  Then $\CT(\Lambda) = \underset{\rightarrow}{\lim} \CT(\Lambda_n,\Lambda)$.

\end{Corollary}

\begin{Remark}

It is not in general true that the $*$-homomorphism of Theorem \ref{thm relative lcsc homomorphism} is injective, even in the case that $G_2(\Lambda)$ is amenable. The claim to the contrary of \cite[Corollary 6.2]{spi_pathcat} is incorrect.

\end{Remark}

\begin{Remark}

For $\zeta \in \CZ(\Lambda_0,\Lambda)$, $\Phi_\zeta^{(\Lambda_0,\Lambda)}$ is a partial homeomorphism of $X(\Lambda_0,\Lambda)$, while $\Phi_\zeta^\Lambda$ is a partial homeomorphism of $X(\Lambda)$.  Therefore relation (4)$_1$ of Definition \ref{def toeplitz relations} is not the same in $C^*(G_1(\Lambda_0,\Lambda))$ as in $C^*(G_1(\Lambda))$. Thus there is not in general an analog of Theorem \ref{thm relative lcsc homomorphism} for $C^*(G_1(\Lambda_0,\Lambda))$.

\end{Remark}

\end{document}